\documentclass[11pt,english]{article}
\usepackage[utf8]{inputenc}
\usepackage{geometry}
\usepackage{color}
\usepackage{babel}
\usepackage{amsthm}
\usepackage{amsmath}
\usepackage{amssymb}
\usepackage{esint}
\usepackage{hyperref}
\hypersetup{ unicode=true,pdfusetitle, bookmarks=true,bookmarksnumbered=false,bookmarksopen=false, breaklinks=false,pdfborder={0 0 1},backref=false,colorlinks=true, linktocpage=true}

\makeatletter
\numberwithin{equation}{section}
\theoremstyle{plain}
\newtheorem{thm}{\protect\theoremname}[section]
  \theoremstyle{plain}
  \newtheorem{conjecture}[thm]{\protect\conjecturename}
  \theoremstyle{remark}
  
  \theoremstyle{plain}
  \newtheorem{prop}[thm]{\protect\propositionname}
  \theoremstyle{plain}
  \newtheorem{lem}[thm]{\protect\lemmaname}
  \theoremstyle{definition}
  \newtheorem{defn}[thm]{\protect\definitionname}

\newtheorem{innercustomthm}{Conjecture}
\newenvironment{customthm}[1]
  {\renewcommand\theinnercustomthm{#1}\innercustomthm}
  {\endinnercustomthm}
\date{}

\usepackage{amsthm}
\usepackage{epsfig}
\usepackage{amsbsy}\usepackage{amsfonts}\usepackage{tikz}\usetikzlibrary{calc,fadings,decorations.pathreplacing}
\usepackage{tikz}
\usepackage{epsfig}
\usepackage{appendix}
\usepackage{endnotes}

\usetikzlibrary{arrows}
\usetikzlibrary{shapes}

\usepackage{epsfig}\usepackage{booktabs}

\newenvironment{claim}[1]{\par\noindent\underline{Claim:}\space#1}{}

\usepackage{amsthm}
\usepackage{caption}
\usepackage{blkarray}

\theoremstyle{plain}

\author{}
\date{\today}

\makeatother

  \providecommand{\conjecturename}{Conjecture}
  \providecommand{\definitionname}{Definition}
  \providecommand{\lemmaname}{Lemma}
  \providecommand{\propositionname}{Proposition}
  \providecommand{\remarkname}{Remark}
\providecommand{\theoremname}{Theorem}

\begin{document}
\global\long\def\A{\mathcal{A}}
\global\long\def\R{\mathbf{R}}
\global\long\def\ab{\mathrm{ab}}
\global\long\def\Homeo{\mathrm{Homeo}}
\global\long\def\SL{\mathrm{SL}}
\global\long\def\G{\Gamma}
\global\long\def\H{\mathcal{H}}
\global\long\def\F{\mathcal{F}}
\global\long\def\C{\mathcal{M}}
\global\long\def\Sp{\mathrm{Sp}}
\global\long\def\Z{\mathbf{Z}}
\global\long\def\RR{\mathfrak{R}}
\global\long\def\ZR{\mathcal{ZR}}
\global\long\def\zip{\mathsf{\mathrm{zip}}}
\global\long\def\M{\mathcal{M}}
\global\long\def\k{\kappa}
\global\long\def\End{\mathrm{End}}
\global\long\def\Mod{\mathrm{Mod}}
\global\long\def\X{\mathcal{X}}
\global\long\def\MCG{\mathrm{MCG}}
\global\long\def\Aut{\mathrm{Aut}}
\global\long\def\i{\iota}
\global\long\def\V{\mathcal{V}}
\global\long\def\area{\mathrm{area}}
\global\long\def\J{\mathcal{J}}
\global\long\def\S{\mathfrak{S}}
\global\long\def\e{\epsilon}
\global\long\def\Leb{\mathrm{Leb}}
\global\long\def\B{\mathcal{B}}
\global\long\def\RI{Q}
\global\long\def\hRI{\hat{\RI}}
\global\long\def\T{\mathcal{T}}
\global\long\def\K{\mathcal{K}}
\global\long\def\L{\mathcal{L}}
\global\long\def\new{\mathrm{new}}
\global\long\def\pr{\mathrm{pr}}
\global\long\def\Var{\mathrm{Var}}
\global\long\def\Ad{\mathrm{Ad}}
\global\long\def\Ind{\mathrm{Ind}}
\global\long\def\Hom{\mathrm{Hom}}
\global\long\def\SO{\mathrm{SO}}
\global\long\def\spec{\mathrm{spec}}
\global\long\def\Comp{\mathrm{Comp}}
\global\long\def\P{\mathcal{P}}
\global\long\def\LL{\mathbf{L}}
\global\long\def\cover{\mathrm{cover}}
\global\long\def\K{\mathcal{K}}

\global\long\def\k{\kappa}

\global\long\def\N{\mathbf{N}}
\global\long\def\E{\mathcal{E}}

\title{On Selberg's Eigenvalue Conjecture for moduli spaces of abelian differentials}

\author{Michael Magee}

\maketitle
\begin{abstract}
J.-C. Yoccoz proposed a natural extension of Selberg's Eigenvalue Conjecture to moduli spaces of abelian differentials. We prove an approximation to this conjecture.
This gives a qualitative generalization of Selberg's $\frac{3}{16}$ Theorem to moduli spaces of abelian differentials on surfaces of genus $\geq 2$.
 \end{abstract}

\section{Introduction}
Let $\Lambda := \SL_2(\Z)$ be the modular group. Then  $\Lambda$ acts on the hyperbolic upper half plane $\mathbb{H}$ by M\"{o}bius transformations and the quotient $X := \Lambda \backslash \mathbb{H}$ is an orbifold Riemann surface. We denote by $\Lambda(q)$ the \emph{principal congruence subgroup} of $\Lambda$ given by the kernel of the reduction modulo $q$ map $\Lambda \to \SL_2(\Z/q \Z)$. Then $\Lambda(q)$ is a normal subgroup of $\Lambda$ and for $q\geq 2$
$$
X(q) := \Lambda(q) \backslash \mathbb{H}
$$ 
is a Riemann surface.

If we parameterize points in $\mathbb{H}$ by $x + iy$ with $x,y\in \mathbf{R}$ and $y>0$ then the Laplacian on $\mathbb{H}$ is given by
$$
\Delta = - y^2\left(  \frac{\partial^2}{\partial x ^2} +  \frac{\partial^2}{\partial y ^2}  \right).
$$
This operator is invariant under M\"{o}bius transformations and hence descends to an operator  on smooth
 functions on $X(q)$. The surface $X(q)$ also has a measure $\mu_q$ induced from the $\Lambda$-invariant volume form $\frac{ dx \wedge dy}{y^2}$ on $\mathbb{H}$.  The Laplacian  extends to an unbounded operator  $\Delta_{X(q)}$ on $L^2(X,\mu_q)$. For all $q\geq 2$, $\Delta_{X(q)}$ has a simple eigenvalue at $0$ and the spectrum of $\Delta_{X(q)}$ below $\frac{1}{4}$ is discrete. Therefore we may write $\lambda_1(X(q))$ for the smallest non-zero eigenvalue of $\Delta_{X(q)}$. In a celebrated 1965 paper \cite{SELBERG316}, Selberg proved 
\begin{thm}[Selberg's $\frac{3}{16}$ Theorem]
For all $q\geq 2$, $\lambda_1(X(q))\geq \frac{3}{16}$.
\end{thm}
At the same time, Selberg made the following conjecture.
\begin{conjecture}[Selberg's Eigenvalue Conjecture]\label{conj:selberg}
For all $q\geq 2$, $\lambda_1(X(q))\geq \frac{1}{4}$.
\end{conjecture}
This conjecture cannot be pushed any further since there are examples \cite{MAASS} of $q$ such that $\Delta_{X(q)}$ has an eigenvalue at $\frac{1}{4}$.
Progress on Conjecture \ref{conj:selberg} has been made by several authors over the interim decades, including works of Gelbart and Jacquet ($\lambda_1 > \frac{3}{16}$) \cite{GJ}, Luo, Rudnick and Sarnak  ($\lambda_1 \geq \frac{171}{784}$) \cite{LRS}, and Iwaniec ($\lambda_1 \geq \frac{10}{49}$) \cite{IWANIEC}. The current best result is due to Kim and Sarnak \cite{KIM} who proved for all $q\geq 2$,
$$
\lambda_1(X(q))\geq\frac{975}{4096} \approx 0.238 .
$$

Selberg's conjecture remains one of the fundamental open questions of automorphic forms; see the expository articles of Sarnak  \cite{SARNAK1, SARNAK2}.

Selberg's conjecture can also be stated in terms of representation theory. This is fitting with Selberg's original motivation\footnote{See \cite{HS}.} of Conjecture \ref{conj:selberg} as an archimedean analog of the Ramanujan-Petersson conjectures. The equivalence classes of irreducible unitary representations (\emph{unitary dual}) of $\SL_2(\R)$ were classified by 
Bargmann  \cite{BARGMANN}: one has 
the trivial representation, the principal series,
complementary series, discrete series and limits of discrete series. Of particular interest to
us are the complementary series $\Comp^{u}$ that are indexed by a
parameter $u\in(0,1)$, see \cite[pg. 36]{KNAPP} for a precise description
of these representations. 

For each $q\geq 2$, we obtain a unitary representation of $\SL_2(\R)$ on $L^2( \Lambda(q) \backslash \SL_2(\R) )$ by right translation. This
 representation can be decomposed as a direct integral over a projection valued measure on the unitary dual of $\SL_2(\R)$. Conjecture \ref{conj:selberg} is equivalent to 
\begin{customthm}{\ref{conj:selberg}*}\label{eight}
For all $q \geq 2$, the measure on the unitary dual of $\SL_2(\R)$ that decomposes $L^2( \Lambda(q) \backslash \SL_2(\R) )$ is supported away from complementary series representations.
\end{customthm}

The point of view taken in this work is that $\Lambda \backslash \SL_2(\R)$ is the moduli space of unit area translation surfaces of genus 1 and hence Selberg's Eigenvalue Conjecture is a conjecture about moduli spaces and their covering spaces. A \emph{ translation surface} is a topological surface $S$ with a finite subset $\Sigma$, together with a set of complex charts on $S - \Sigma$ such that all transition functions are translations, and the charts extend to conical singularities at $\Sigma$. Let $\Sigma = \{A_1 , \ldots , A_s \}$. The conical singularity at $A_i$ is required to have cone angle $2 \pi (\kappa_i + 1)$ with $\kappa_i \in \Z_+$ and Gauss-Bonnet forces the relation
$$
\sum_{i=1}^s \kappa_i = 2g-2.
$$
Translation surfaces can be equivalently be thought of as abelian differentials with respect to a complex structure on $S$. The zeros of the differential correspond to the conical singularities of the translation surface.

The moduli space of translation surfaces of genus $g \geq 2$ is stratified according to the partitions $\kappa = (\kappa_1 , \ldots ,\kappa_s)$. A stratum $\H(\kappa)$ need not be connected, but there are finitely many connected components that are understood by work of Kontsevich and  Zorich \cite{KZ}. We let $\H^{(1)}(\k)\subset \H(\kappa)$ denote the unit area translation surfaces in $\H(\k)$. In this paper, $\C$ will be a connected component of $\H^{(1)}(\k)$.
Since $\H^{(1)}(\k)$ can be obtained as a quotient of a Teichm\"{u}ller space by the mapping class group $\Gamma=\Gamma(S,\Sigma)$ of $(S,\Sigma)$ (see Section \ref{sub:Abelian-differentials-and}), we may define 
\emph{congruence covers} via the natural family of maps
\begin{equation}
\Pi_q : \Gamma \mapsto \Aut( H_1(S , \Z/q\Z)) .\label{eq:reduction}
\end{equation}
The \emph{principal congruence subgroup}  $\Gamma(q)$ is defined to be the kernel of $\Pi_q$. By considering moduli only up to $\Gamma(q)$, and not $\Gamma$, for each connected component $\C$ of $\H^{(1)}(\kappa)$ we obtain a congruence cover $\C(q)$ generalizing $\Lambda(q) \backslash \SL_2(\R)$. The details of this construction are given in Section \ref{sub:Abelian-differentials-and}.

Each component $\C$ has the following associated objects generalizing those attached to $\Lambda \backslash \SL_2(\R)$:
\begin{itemize}
\item There is an action of $\SL_2(\R)$ on $\C$. The restriction of the $\SL_2(\R)$ action to the one parameter diagonal subgroup gives a flow on $\C$ called the \emph{Teichm\"{u}ller flow} that generalizes the geodesic flow on the unit tangent bundle of $X$.
\item There is a unique   probability measure $\nu_\C$ on $\C$ that is $\SL_2(\R)$-invariant,  ergodic for the Teichm\"{u}ller flow, and in the Lebesgue class with respect to a natural affine orbifold structure on $\C$.  This is due to works of Masur \cite{MASUR} and Veech
\cite{VEECHGAUSS}. 
\item The space $\SO(2) \backslash \C$ is locally foliated by $\mathbb{H}$ and hence it is possible to define a \emph{foliated Laplacian} $\Delta_\C$ on $\SO(2) \backslash \C$ generalizing $\Delta_X$. This operator has a simple eigenvalue at zero and by a result of Avila and Gou\"{e}zel \cite{AG}, its spectrum below $\frac{1}{4}$ has no accumulation points other than possibly at $\frac{1}{4}$.
\end{itemize} 
Each of these objects lifts to $\C(q)$, so there is an $\SL_2(\R)$ action on $\C(q)$ preserving a finite measure $\nu_{\C(q)}$, and a foliated Laplacian $\Delta_{\C(q)}$ whose spectrum below $\frac{1}{4}$ does not accumulate\footnote{By \cite[Remark 2.4]{AG} this result also applies to $\C(q)$.} away from $\frac{1}{4}$. Hence we can write $\lambda_1(\C(q))$ for the infimum of the non-zero spectrum\footnote{In contrast to the situation with $X$, where it is known \cite{SELBERGTRACE} that there are infinitely many eigenvalues of $\Delta_X$, we do not know whether $\Delta_{\M}$ or $\Delta_{\M(q)}$ have any non-zero eigenvalues.}
 of $\Delta_{\C(q)}$. The following extension of Selberg's conjecture to genus $g\geq 2$ was proposed by Yoccoz\footnote{The formulation of the conjecture appears in print in \cite{AG}, although  Avila and Gou\"{e}zel  stopped short of making the conjecture
 because of lack of evidence. We learned from C. Matheus that Yoccoz had made this conjecture in private.}.
\begin{conjecture}[Yoccoz] \label{selbergextension} For all $q\geq 2$, and any connected component $\C$ of a stratum,
\begin{description}
\item[A.] $\lambda_1(\C(q)) \geq \frac{1}{4}$.
\item[B.] The measure on the unitary dual of $\SL_2(\R)$ that decomposes $L^2( \C(q) , \nu_{\C(q)} )$ is supported away from complementary series representations.
\end{description}
\end{conjecture}
The  main theorem of this paper gives an approximation to Conjecture \ref{selbergextension}.
\begin{thm}\label{main}  For any connected component $\C$ of a stratum, there exists $\epsilon, \eta>0$ and $Q_0 \in \Z_+$ such that for all $q$ coprime to $Q_0 $ the following hold.
\begin{description}
\item[A.] $\lambda_1(\C(q)) \geq \epsilon$.
\item[B.] The measure on the unitary dual of $\SL_2(\R)$ that decomposes $L^2( \C(q) , \nu_{\C(q)} )$ is supported away from complementary series representations $\Comp^u$ with $u \in (1-\eta,1)$.
\item[C.] The Teichmüller
flow on $\C(q)$ has exponential decay of correlations on compactly
supported $C^{1}$ observables with a rate of decay that is independent of $q$. 
\end{description}
\end{thm}
The corresponding theorem for $\C$, i.e. without any congruence aspect, was obtained by Avila, Gou\"{e}zel and Yoccoz in \cite{AGY}. In an earlier version of this manuscript, for certain types of components $\C$, Theorem \ref{main} was conditional on a conjecture of Zorich \cite{ZORICHLEAVES} that has since been proved by Guti\'{e}rrez-Romo \cite{GR}.

It is known that Parts {\bf{A.}}, {\bf{B.}}, and {\bf{C.}} of Theorem \ref{main} are equivalent. That Part  {\bf{B.}} implies Part {\bf{C.}}, namely, that one can use representation theory to deduce rates of mixing of the diagonal flow, is due to Ratner \cite{RATNER}. The argument that Part {\bf{C.}}
implies Part {\bf{B.}} is given by the `reverse Ratner
estimates' in \cite[Appendix B]{AGY}. 
The equivalence  between Parts {\bf{A.}} and {\bf{B.}} is due to the interpretation of the foliated Laplacian as a Casimir operator.
This is discussed in detail in \cite[Section 3.4]{AG}.

So it is sufficient to prove the dynamical statement of Part {\bf{C.}} This is made into a precise statement in Theorem \ref{thm:decay-of-correlations-precise.}.

We mention that in recent work \cite{GMR}, joint with R\"{u}hr and Guti\'{e}rrez-Romo, we extend Theorem \ref{main} to congruence covers coming from relative homology of $(S,\Sigma)$, and apply both Theorem \ref{main} and the extended result to the problem of counting saddle connections in a homology class modulo $q$.

\subsection{The ideas of the proof}

The reader is invited to read this section before the rest of the paper for the main ideas of the proof.

While we will prove  Theorem \ref{main} in dynamical terms, the philosophy of the proof goes back to works of Brooks \cite{BROOKS} and Burger \cite{BURGER1, BURGER2} that were originally stated in terms of the first non-zero eigenvalue $\lambda_1$. Both Brooks and Burger realized that if one has a Galois covering $Y \to X$ of Riemann surfaces, with deck transformation group $G$, then one can transfer bounds on the spectral gap of the Cayley graph of $G$ with respect to certain generators, to bounds on the first non-zero eigenvalue $\lambda_1(Y)$ of the Laplacian on $Y$. In particular, if $X$ is fixed, and $Y$ ranges over a family of Galois covers, if the associated Cayley graphs have a uniform spectral gap, then $\lambda_1(Y)$ is uniformly bounded below away from zero.

	The classic construction of an infinite family of graphs of bounded degree with a uniform spectral gap, known as an \emph{expander family}, is take a fixed generating set $U$ in an arithmetic lattice $G(\Z)$ that has Kazhdan's property (T), and then form the Cayley graphs for $G(\Z / q\Z)$ with respect to the projection of $U$ modulo $q$. This construction is due to Margulis \cite{MARGULIS}.
	
	Since the covering spaces $\C(q)$ of this paper have deck transformation groups contained in $\Sp( (H_1(S , \Z / q\Z),\cap) \cong \Sp_{2g}(\Z/ q\Z)$, and $\Sp_{2g}(\Z)$ has property (T) for $g\geq 2$, one might expect the Brooks-Burger philosophy to apply directly here, as long as one can prove that the deck transformation group is all of $\Sp_{2g}( \Z /q\Z)$, or in other words, $\C(q)$ is connected. However even if the issue of $\C(q)$ being connected is resolved\footnote{And this issue can be resolved as follows, however these arguments are not used in the paper.  The image $\Gamma_{\C}$ of the natural representation of the fundamental
group of $\C$ in $\Aut( H_1 (S ,\Z))$ is known to be a Zariski-dense
subgroup of $\Sp_{2g}(\Z)$ by a result of Filip from \cite[Corollary 1.3]{FILIPZERO}.
Then one has  the strong approximation result of Matthews, Vaserstein and Weisfeiler
\cite{MVW} that says if $\Gamma_{\C}$ is Zariski-dense in $\Sp_{2g}(\Z)$
then $\Gamma_{\C}$ maps onto $\Sp_{2g}(\Z/q\Z)$ for all $q$ coprime
to some fixed modulus $q_{0}$ and hence that $\C(q)$ is connected
for the same $q$.}, the Brooks-Burger philosophy does not obviously apply. The core issue is that the foliated Laplacian is not elliptic and only measures fluctuations of functions in the direction of $\SL_2(\R)$-leaves.

Instead we take a dynamical viewpoint. We think of functions on $\C(q)$ as sections of a $\Sp_{2g}(\Z /q\Z)$ principal bundle over $\C$. We know the dynamics on $\C$ is exponentially mixing by the work of Avila, Gou\"{e}zel, and Yoccoz \cite{AGY}. The key point for obtaining  uniform exponential mixing as in Theorem \ref{main}.{\bf C} is to exploit the following fact: when one travels along the Teichm\"{u}ller  flow and returns close to the initial point, we move in the fibre by a monodromy element of $\Sp_{2g}(\Z / q\Z )$. This monodromy is globally defined in the sense that for a given approximate loop, the monodromy at different levels $q$ are obtained by reduction mod $q$ of some element of $\Sp_{2g}(\Z)$. Moreover, if one can argue that the dynamics on the base $\C$ is sufficiently combinatorially complicated, then we can obtain many monodromy elements in this way. Then we hope to use property (T) to prove this dynamics in the fibre spreads out exponentially fast. So one has exponential mixing in the base, and some form of exponential mixing in the fibres, and hopes to combine these two. The problem is that the two processes are not independent. So we will use hyperbolicity of the dynamics on the base $\C$ to `decouple' these aspects of the dynamics. However, the base dynamics is not uniformly hyperbolic, so one needs to perform `time acceleration' as in \cite{AGY} to induce uniform hyperbolicity and then incorporate this into the method.

The previous paragraph was a high level overview of the approach. Now we give details of how this is implemented.

Our framework for understanding the dynamics of $\C$ is that same as Avila, Gou\"{e}zel, and Yoccoz in \cite{AGY}. Namely, instead of working with $\C$, we pass to a finite cover called the moduli space of zippered rectangles $\mathrm{Rect}_\C$ for $\C$. This finite cover carries a lift of the Teichm\"{u}ller flow that has some very nice properties that were worked out in \cite{AGY}. A key insight of \cite{AGY} is that by carefully chosing a cross section,  one obtains a model of the flow on $\mathrm{Rect}_\C$ as a suspension flow over a hyperbolic skew product $\hat{Z} :\hat{\Xi} \to \hat{\Xi}$ with a base transformation $Z : \Xi \to \Xi$ that is a uniformly expanding Markoff map (Lemmas \ref{lem:-is-a-hyperbolic-skew-product} and Proposition \ref{prop:The-map-is uniformly epxanding markov}). Moreover the roof function for this suspension model has desirable properties, it is `good' in the sense of \cite{AGY} (Lemma \ref{lem:The-roof-function-is-good}) and it has exponential tails (Theorem \ref{thm:The-roof-function-has-exponential-tails}). The latter statement is quite hard and relies on exponential recurrence estimates for the Teichm\"{u}ller flow that were first obtained by Athreya \cite{ATHREYA}.

This suspension model has another key property that is not explicitly used in \cite{AGY}: the symbolic coding is very well adapted to keeping track of what happens to the homology of the surface when we follow the flow. Indeed, there is a linear group $G$ attached to $\C$ called the \emph{Rauzy-Veech group} that is defined purely in terms of the symbolic dynamics of $\mathrm{Rect}_\C$ and the return maps on the base of the suspension model. This group $G$ performs the desired function of keeping track of monodromy in homology around approximate loops and is defined precisely in Section \ref{sub:The-Rauzy-Veech-group}. It was a conjecture of Zorich \cite{ZORICHLEAVES} that $G$ is Zariski-dense in its ambient symplectic group. Recently, it has been proven in works of Avila, Matheus and Yoccoz \cite{AMY}, and  Guti\'{e}rrez-Romo \cite{GR}, that the Rauzy-Veech group is finite index in $\Sp_{2g}(\Z)$. Therefore, in particular, it has property (T). The precise statement about the Rauzy-Veech group that we use is given in Theorem \ref{thm:zorich-conjecture}.

Other than discussing the Rauzy-Veech group, the main purpose of Section \ref{sec:back} is to go through the setup of \cite{AGY} and explain how to keep track of what happens to homology along the flow, as well as stating the results we need from \cite{AGY}.

In Section \ref{sec:decay}, we follow the strategy of Avila, Gou\"{e}zel, and Yoccoz of reducing Theorem \ref{thm:decay-of-correlations-precise.}, the precise formulation of our main theorem, to exponential mixing of the flow on $\mathrm{Rect}_\C$ (Theorem \ref{thm:full-C^1-decay-of-correlations}), and then to exponential mixing of a suspension flow over the base $\Xi$ of the hyperbolic skew product (Theorem \ref{thm:Decay of correlations roof over Markov}). These statements must now be uniform in $q$.

A well known technique for proving exponential mixing of suspension flows is to take a Laplace transform of the correlation function, and express this transform in terms of iterates of \emph{transfer operators}. To deal with the $q$ aspect, one uses \emph{skew transfer operators}, one operator for each $q$.  The transfer operators act on vector valued $C^1$ functions on $\Xi$ and one needs spectral estimates for the transfer operators that are uniform in $q$.
This strategy of proving uniform exponential mixing via $q$-uniform bounds on transfer operators originates in work of Oh and Winter 
\cite{OW}. One needs estimates for the transfer operators in two regimes: high frequency (given by Proposition \ref{prop:contraction-transfer-warped})
 and low frequency (given by Proposition \ref{prop:Contraction-congruence-transfer}). 
 
The technique for carrying out the necessary high frequency estimates are due to Dolgopyat \cite{DOLG} and extended to the current setting, with no $q$-aspect, by Avila, Gou\"{e}zel, and Yoccoz \cite{AGY}. The use of the Dolgopyat argument to establish $q$-uniform versions of the high frequency estimates was first done by Oh and Winter \cite{OW}, and then in a different setting by Magee, Oh, and Winter \cite{MOW}. In Section \ref{sec:dolg} we explain how to extend the arguments of Avila, Gou\"{e}zel, and Yoccoz in this regime to skew transfer operators.

The technique for proving $q$-uniform low frequency estimates for skew transfer operators goes back to the work of Bourgain, Gamburd, and Sarnak\footnote{Bourgain, Gamburd, and Sarnak were interested in spectral bounds for transfer operators for reasons that are related to exponential mixing but in \cite{BGS2} phrased in terms of counting problems.}
 \cite{BGS2}. The philosophy here, mirroring the Brooks-Burger philosophy, is that an iterate of the transfer operator looks somewhat like an iterate of the adjacency operator of a Cayley graph of $\Sp_{2g}(\Z /q\Z)$.
In work of Bourgain, Kontorovich, and Magee \cite[Appendix]{MOW}, an improvement was made to this method that allows one to use uniform expansion of Cayley graphs (in the current setting, furnished by property (T)) as a `black box'\footnote{The original argument of Bourgain, Gamburd, and Sarnak in \cite{BGS2} involved unraveling the proof that the associated Cayley graphs are uniform expanders.} to prove $q$-uniform estimates for transfer operators. 

We give the details of how this method can be extended to the current setting in Section \ref{sec:Expansion-and-the-twisted-transfer-operator}. It requires not only the uniform expansion of certain Cayley graphs as an input, but also an extra input that the dimensions of representations of $\Sp_{2g}(\Z / q\Z)$ that do not arise from representations of $\Sp_{2g}(\Z / q' \Z )$ with $q'|q$ have a lower bound that is polynomial in $q$. This is a version of \emph{quasirandomness}\footnote{ Gowers \cite[Theorem 4.5]{GOWERS} made the definition
that a finite group $G$ should be regarded as quasirandom relative
to an ambient parameter $C$ if the dimension of any nontrivial irreducible
representation of $G$ has dimension $\geq C$. Prior to this formal
notion, the concept had been used in the construction of Ramanujan
graphs by Lubotzky, Phillips and Sarnak \cite{LPS}, the work of Sarnak
and Xue on multiplicities of automorphic representations \cite{SX},
and the construction of uniformly expanding Cayley graphs of $\SL_{2}(\mathbb{F}_{p})$
by Bourgain and Gamburd \cite{BGEXPAND}. } for $\Sp( \Z / q\Z)$. The reason for needing this kind of bound is that is allows us to obtain information on the spectral radius of a complex-valued measure $\mu$ on $\Gamma_q = \Sp_{2g}(\Z / q\Z)$ acting by convolution on a certain subspace of $\ell^2(\Gamma_q)$ if we have information on the spectral radius of a real-valued measure $\mu'$ that majorizes $|\mu|$. This is a key idea in Section \ref{sec:Expansion-and-the-twisted-transfer-operator}. We state the precise quasirandomness estimate we need in Proposition \ref{prop:quasirandomness} and then prove it following an argument of Kelmer and Silberman \cite{KS}.

\subsection{Acknowledgements}

We would like to thank Alex Eskin, Ilya Gekhtman, S\'{e}bastien Gou\"{e}zel, Rodolfo Guti\'{e}rrez-Romo, Carlos Matheus, Babak Modami, Hee Oh and Giulio Tiozzo for
helpful discussions related to this work.

\section{Background}\label{sec:back}

\subsection{Abelian differentials and translation surfaces\label{sub:Abelian-differentials-and}}

Let $g\geq1$ and let $S=S_{g}$ be a fixed topological surface of
genus $g$. Let $\Sigma = \{A_1 , \ldots , A_s \}$ be a finite subset of $S$. An \emph{abelian differential }on $(S,\Sigma)$ is a pair $(\J,\omega)$
where $\J$ is a complex structure on $S$ and $\omega$ is a holomorphic
one form with respect to $\J$, and with zeros contained in $\Sigma$. 
As is well known, an abelian
differential $\omega$ on $(S,\Sigma)$  gives $S$ the
structure of a  translation surface with conical singularities in $\Sigma$; the complex structure comes from integrating the differential. Hence we may speak about the area of an abelian differential as the area of the corresponding translation surface.

One may further specifiy that the abelian differential has a zero of order $\kappa_i\in \Z_+$ at $A_i$. This is possible whenever $\sum\k_{i}=2g-2$.
For such $\kappa=(\kappa_{1},\ldots,\kappa_{s})$
we let $\X(\kappa)$ denote the collection of abelian differentials on $(S,\Sigma)$
with  zeros of orders $\kappa_{1},\ldots,\kappa_{s}$ at $A_1 ,\ldots, A_s$, up to isotopies of $S$ preserving $\Sigma$.
This Teichm\"{u}ller space has a natural affine manifold structure arising through period coordinates as described in \cite[Section 2.2.1]{AGY}. Let $\X^{(1)}(\k)\subset \X(\k)$ be the abelian differentials whose corresponding translation surface has unit area, up to isotopy. Then $\X^{(1)}(\k)$ is an affine submanifold of $\X(\k)$.

The modular group $\Gamma = \Gamma(S,\Sigma)$ is defined to be the homeomorphisms of $S$ that fix $\Sigma$ pointwise, modulo homeomorphisms that are isotopic to the identity relative to $\Sigma$. Thus $\Gamma$ acts on $\X(\kappa)$, preserving $\X^{(1)}(\k)$, and we define
$\H(\kappa)$ to be $\Gamma \backslash \X(\k)$ and $\H^{(1)}(\k) = \Gamma \backslash \X^{(1)}(\k)$.
This $\H^{(1)}(\k)$ is often referred to as a \emph{stratum} of the moduli space of unit area abelian differentials. The connected components of these
strata have been classified by Kontsevich and Zorich \cite{KZ}. Throughout the paper we write $\C$ for a connected component of $\H^{(1)}(\kappa)$.

Any connected component $\C$ of $\H^{(1)}(\k)$ inherits, from the manifold structure of $\X^{(1)}(\kappa)$, the structure
of an affine orbifold. We define $\H^{(1)}(\kappa; q) = \Gamma_g(q) \backslash \X^{(1)}(\kappa)$ where $\Gamma_g(q)$ is the kernel of $\Pi_q$ defined in \eqref{eq:reduction}. We thus have a covering map $\H^{(1)}(\kappa; q)\to \H^{(1)}(\kappa)$. We define $\C(q)$ to be the preimage of $\C$ under this map. For each $q$ the lift of $\C(q)$
to $\X(\kappa)$ is a submanifold. 

Recall that a Finsler manifold is a smooth manifold together with
a continuous assignment of norm on each tangent fibre. The norm is
called a Finsler metric. As described in \cite[Section 2.2.2]{AGY}
there is a $\Gamma$-invariant Finsler metric on $\X(\kappa)$ arising from period coordinates making $\X(\kappa)$ into a Finsler manifold. This induces a Finsler manifold structure on $\X^{(1)}(\k)$.

\subsection{The Hodge bundle\label{sub:The-Hodge-bundle}}

The Hodge bundle is defined
to be the fibred product 
\[
H_{1}(\H^{(1)}(\kappa)):=\Gamma\backslash(\X^{(1)}(\kappa) \times H_{1}(S,\R))\to\Gamma\backslash\X^{(1)}(\k)= \H^{(1)}(\k)
\]
where the mapping class group $\Gamma$ acts diagonally. Let $\H^{(1)}(\kappa)^{0}$
be the complement of the orbifold points in $\H^{(1)}(\k)$. The Hodge
bundle restricts to a vector bundle $H_{1}(\H^{(1)}(\k)^{0})$ over $\H^{(1)}(\k)^{0}$.
At any orbifold point $[(\J,\omega)]$ of $\H^{(1)}(\k)$ the fibre degenerates
to $\Aut(\J,\omega)\backslash H_{1}(S,\R)$. Note that by Hurwitz's
automorphisms theorem, $\Aut(\J,\omega)$ is a finite group.

The total space of the Hodge bundle contains as a discrete subset
the lattice bundle
\[
\Gamma\backslash(\X^{(1)}(\k)\times H_{1}(S,\Z)).
\]
Then one may specify the \emph{Gauss-Manin} connection on the Hodge
bundle by the requirement that lattice valued continuous sections
be parallel. This gives a flat vector bundle connection on $H_{1}(\H^{(1)}(\k)^{0})$
that extends to a flat connection on $H_{1}(\H^{(1)}(\k))$ in the following
sense. A section of $H_{1}(\H^{(1)}(\k))$ can be viewed as a function
$\sigma:\X^{(1)}(\k)\to H_{1}(S,\R)$ that transforms according to
\[
\sigma(\gamma.x)=\gamma_{*}\sigma(x),\quad\gamma\in\Gamma.
\]
Then a local section is parallel by definition if it takes values
in $H_{1}(S,\Z)$ and this specifies the connection on general sections.

The action of $\Gamma$ on $H_{1}(S,\Z)$ lies in the integral symplectic
group $\Sp(H_{1}(S,\Z),\cap)$ where $\cap$ is the (symplectic) intersection
form on integral homology. Therefore for any unitary representation
%\marginpar{$\rho,V$}
$(\rho,V)$ of $\Sp(H_{1}(S,\Z),\cap)$ we obtain an \emph{associated
orbifold vector bundle}\footnote{By \emph{orbifold vector bundle} we mean that the fibres are vector
spaces of constant rank away from the orbifold points of the base
space, where the fibres degenerate only to a quotient of a vector
space by a finite group.}\emph{ }
%\marginpar{$H_{1}(\protect\M_{\protect\k};\rho)$}
 $H_{1}(\H^{(1)}(\k);\rho).$
The total space of this bundle is
\begin{equation}
H_{1}(\H^{(1)}(\k);\rho)=\Gamma\backslash(\X^{(1)}(\k)\times V)\label{eq:H1(stratum,rho)}
\end{equation}
where the action of $\Gamma$ on $\X^{(1)}(\k) \times V$ is given by $\gamma.(\omega,v)=(\gamma.\omega,\rho(\gamma_{*}).v)$,
where $\gamma_{*}\in\Sp(H_{1}(S,\Z),\cap)$ is the map induced by
$\gamma$ on homology. This bundle also has a flat connection, in
the same sense as before, coming from the fibred product structure
in (\ref{eq:H1(stratum,rho)}).

Of course, for any connected component $\C$ of the stratum $\H^{(1)}(\k)$
we may restrict $H_{1}(\H^{(1)}(\k))$ or $H_{1}(\H^{(1)}(\k);\rho)$ to $\C$. We denote by
%\marginpar{$H_{1}(\protect\C;\rho)$}
$H_{1}(\C;\rho)$ the obtained orbifold
vector bundle. 

For a lot of the rest of the paper we deal with abstract unitary $\rho$
but in reality we are interested in the following specific examples.
Recall the map  $\Pi_{q}$ from \eqref{eq:reduction}. Because the symplectic intersection product $\cap$ on $H_1( S, \Z/q\Z)$ is preserved by the mapping class group, we have
$$
\Pi_q : \Gamma \to \Sp(H_{1}(S;\Z/q\Z),\cap).
$$
We let $\Gamma_q = \Sp(H_{1}(S;\Z/q\Z),\cap)$. Let $\ell_{0}^{2}(\Gamma_{q})$ be the subspace of functions in $\ell^{2}(\Gamma_{q})$
that are orthogonal to constant functions with respect to the $\ell^{2}$
inner product. This gives a subrepresentation
%\marginpar{$\rho_{q}$}
$(\rho_{q},\ell_{0}^{2}(\Gamma_{q}))$ of the action of $\Gamma$
on $\ell^{2}(\Gamma_{q})$ by reduction mod $q$ and then left translation\footnote{In other words, the inflation of the left regular representation of
$\Gamma_{q}$ to $\Gamma$.}. 

We will also consider the subspace of $\ell_{0}^{2}(\Gamma_{q})$
consisting of functions that are orthogonal to all functions lifted
from $\Gamma_{q'}$ with $q'|q$ via the natural mapping of reduction
modulo $q'$ 
\[
\Gamma_{q}\to\Gamma_{q'}.
\]
We denote by $\ell_{\new}^{2}(\Gamma_{q})$ this \emph{new subspace}
of functions. This gives a subrepresentation\\
% \marginpar{$\rho_{q}^{\protect\new}$} 
 $(\rho_{q}^{\new},\ell_{\new}^{2}(\Gamma_{q}))$
of $(\rho_{q},\ell_{0}^{2}(\Gamma_{q}))$.

\subsection{The Teichmüller flow on moduli space \label{sub:The-Teichmuller-flow}}

There is a postcomposition action of $\SL_{2}(\R)$ on the space of
abelian differentials on $S$ as follows. For $h\in\SL_{2}(\R)$ we
define
\[
h.(\J,\omega)=(\J_{h},\omega_{h})
\]
 where
\[
\omega_{h}=h\left(\begin{array}{c}
\Re(w)\\
\Im(w)
\end{array}\right)
\]
and $\J_{h}$ is the unique complex structure on $S$ that makes $\omega_{h}$
holomorphic. This action preserves the area of abelian differentials.
As this action also commutes with any homeomorphism of $S,$
it descends to both the Teichmüller spaces $\X(\k)$, $\X^{(1)}(\k)$ and $\H(\k)$, $\H^{(1)}(\k)$ and $\C$. The \emph{Teichmüller geodesic flow} on any of these
objects is the restriction of the $\SL_{2}(\R)$ action to the diagonal
subgroup:
%\marginpar{$\protect\T_{t}$}
\[
\T_{t}(\J,\omega):=\left(\begin{array}{cc}
e^{t} & 0\\
0 & e^{-t}
\end{array}\right).(\J,\omega).
\]
The Teichmüller flow also preserves each connected component $\C$. 
By results of Masur \cite{MASUR} and Veech \cite{VEECHGAUSS} there
is a unique probability measure
%\marginpar{$\nu_{\protect\C}$}
$\nu_{\C}$ that is invariant and ergodic for the Teichmüller
flow on $\C$. This measure is in the Lebesgue class with respect
to period coordinates on $\C$. We pull back the measure $\nu_{\C}$ on $\C$, using the counting measure on the fibres of the covering map, to obtain a measure $
\nu_{\C(q)}$ on $\C(q)$. Note that $\nu_{\C(q)}$ is not a probability measure.

Since in Section \ref{sub:The-Hodge-bundle} we specified a connection
on each of $H_{1}(\H^{(1)}(\k))$, $H_{1}(\H^{(1)}(\k);\rho)$ the Teichmüller
flow acts on sections of each of these bundles by pullback along parallel
transport. For example, viewing a section of $H_{1}(\H^{(1)}(\k);\rho)$
as a $V$-valued function $\sigma$ on $\X^{(1)}(\k)$ satisfying $\sigma(\gamma.x)=\rho(\gamma)\sigma(x)$
for each $\gamma\in\Gamma$, we have the following defining equation
for $\T_{t}^{*}$:
\begin{equation}
[\T_{t}^{*}\sigma](\J,\omega):=\sigma(\T_{t}(\J,\omega)).\label{eq:pullback-sections}
\end{equation}
This action also restricts to an action on sections of $H_{1}(\C;\rho)$. 

We now explain the relationship between
sections of $H_{1}(\C;\rho_{q})$ and functions on $\C(q)$.
Let $L_{
\star}^{2}(\C(q))$ be the subspace of functions
in $L^{2}(\C(q))$ orthogonal to lifts from $L^{2}(\C)$,
w.r.t. the measure $\nu_{\C(q)}$. Let $L^{2}(H_{1}(\C;\rho_{q}))$
denote the $L^{2}$ sections of $H_{1}(\C;\rho_{q})$ w.r.t.
the natural Hermitian fibre metric and measure $\nu_{\C}$.
We say that a function $f$ on $\C(q)$ or a section $\sigma$
of $H_{1}(\C;\rho)$ is $C^{1}$ if its lift to $\tilde{f}:\X^{(1)}(\k)\to\mathbf{C}$
(resp. $\tilde{\sigma}:\X^{(1)}(\k)\to V$) is $C^{1}$ (bounded
with bounded derivative\footnote{\label{foot}In case of $V$-valued $F$ on a Finsler manifold $X$ with $V$ a
Hilbert space, to define the norm of the derivative we view the derivative
at $x\in X$ as a map $DF_{x}:T_{x}X\to T_{F(x)}V\cong V$ then use
the operator norm w.r.t. the Finsler metric at $x$ and the Hilbert
space norm on $V$.}) w.r.t. the the Finsler manifold structure on $\X^{(1)}(\k)$. Define
$\|f\|_{C^{1}}=\|\tilde{f}\|_{\infty}+\|D\tilde{f}\|_{\infty}$ and
similarly $\|\sigma\|_{C^{1}}$. Write $C^{1}(\C(q))$ for the
$C^{1}$ complex valued functions on $\C(q)$ and $C^{1}(H_{1}(\C;\rho))$
for the $C^{1}$ sections of $H_{1}(\C;\rho)$. These are Banach
spaces w.r.t the respective $C^{1}$ norms. 
\begin{lem}
\label{lem:We-have-the-following-correspondences}We have the following
correspondences
\begin{enumerate}
\item For each $q$ there is a natural linear isometry 
\[
\Phi_{q}:L_{\star}^{2}(\C(q))\to L^{2}(H_{1}(\C;\rho_{q})).
\]
\item The map $\Phi_{q}$ intertwines the maps $\T_{t}^{*}$ defined by
pullback on $L_{\star}^{2}(\C(q))$ and by (\ref{eq:pullback-sections})
on $L^{2}(H_{1}(\C;\rho_{q}))$.
\item The restriction 
\[
\Phi_{q}:C^{1}(\C(q))\cap L_{\star}^{2}(\C(q))\to C^{1}(H_{1}(\C;\rho_{q}))
\]
preserves $C^{1}$ norms.
\end{enumerate}
\end{lem}

\subsection{Combinatorial data and Rauzy classes}

Now we begin an account of the dynamics of the Teichmüller flow, viewed
through the lens of Veech's zippered rectangles construction. We draw
in the following sections from the sources \cite{AGY}, \cite{VIANA}
that both build on work of Marmi, Moussa and Yoccoz \cite{MMY}.

The relevant combinatorial objects are as follows\emph{. }Let 
%\marginpar{$\protect\A$}
$\A$
denote a finite alphabet with
%\marginpar{$d$}
 $|\A|=d$. Eventually,
$\A$ will be chosen depending on $g,\k$ and the component $\C$.
We let
%\marginpar{$\protect\S(\protect\A)$} 
$\S(\A)$ denote the set
of pairs 
\[
(\pi_{t},\pi_{b})
\]
where each $\pi_{\epsilon}:\A\to\{1,\ldots,d\}.$ Henceforth, $\epsilon$
will index one of the symbols $t,b$ (`top' or `bottom'). As in \cite{AGY}
it is convenient to visualize $(\pi_{t},\pi_{b})$ as a pair of rows
each of which contains the elements of $\A$ in some order, where
the top corresponds to $\pi_{t}$ and the bottom to $\pi_{b}$. We
say $(\pi_{t},\pi_{b})$ is \emph{irreducible }if there is no $d'<d$
such that the set of the first $d'$ elements of the top row is the
same as the first $d'$ elements of the bottom. Let $\S^{0}(\A)\subset\S(\A)$
denote the irreducible combinatorial data.

We now define `top' and `bottom' operations on
% \marginpar{$\protect\S^{0}(\protect\A)$}
$\S^{0}(\A)$.
For the next paragraph, let $\alpha$ and $\beta$ denote the last
elements of the top and bottom rows of $\pi\in\S^{0}(\A)$ respectively.
The top operation on $\pi$ modifies the bottom row by moving the
occurrence of $\beta$ to the immediate right of the occurrence of
$\alpha$. The bottom operation modifies the top row by moving $\alpha$
to the right of $\beta$. As in \cite{AGY} we say that the last element
of the unchanged row is the \emph{winner }and the last element of
the row of $\pi$ that is to be changed the \emph{loser. }

By adding directed `top' and `bottom' labelled edges according to
these operations we obtain an edge-labeled directed graph on the vertex
set of irreducible combinatorial data $\S^{0}(\A)$. Each vertex has
exactly one incoming top (resp. bottom) and one outgoing top (resp.
bottom) edge. A \emph{Rauzy diagram }is a connected component of this
graph and a \emph{Rauzy class }is the vertex set of a\emph{ }Rauzy
diagram\emph{. }

\subsection{Suspension data and zippered rectangles}

Let
%\marginpar{$\protect\RR$} $\RR$ 
be a Rauzy class. For each $\pi\in\RR$
we form a cell 
%\marginpar{$X_{\pi}$}
\[
X_{\pi}=\{\pi\}\times\R_{+}^{\A}\times\K_{\pi}
\]
where 
%\marginpar{$\protect\K_{\pi}$}
\[
\K_{\pi}=\left\{ \tau\in\R^{\A}\::\:\sum_{\pi_{t}(\xi)\leq k}\tau_{\xi}>0,\:\sum_{\pi_{b}(\xi)\leq k}\tau_{\xi}<0\:\text{for all }1\leq k\leq d-1.\right\} 
\]
The set $\K_{\pi}$ is an open convex cone. Let $X_{\RR}=\cup_{\pi\in\RR}X_{\pi}.$
We may drop the dependence on $\RR$ since we usually view it as fixed.
We associate to each $\pi\in\RR$ a linear map
%\marginpar{$\Omega_{\pi}$}
$\Omega_{\pi}:\R^{\A}\to\R^{\A}$ given by
\[
[\Omega_{\pi}]_{\alpha,\beta}=\begin{cases}
+1 & \text{if }\pi_{t}(\alpha)>\pi_{t}(\beta),\:\pi_{b}(\alpha)<\pi_{b}(\beta),\\
-1 & \text{if }\pi_{t}(\alpha)<\pi_{t}(\beta),\:\pi_{b}(\alpha)>\pi_{b}(\beta),\\
0 & \text{else.}
\end{cases}
\]

There is a construction due to Veech \cite{VEECHGAUSS} that builds
a point in the moduli space of translation surfaces from suspension data. This mapping is called the \emph{zippered
rectangles }construction that we denote by
%\marginpar{$\protect\zip$}
\[
\zip:X_{\RR}\to\H(\k),\quad\k=\k(\RR).
\]
The explicit details of this construction are clearly described in
lecture notes of Viana \cite[Chapter 2]{VIANA}. In the current paper
it will be better to simply work with the properties of the map $\zip$
that we give below. Henceforth a superscript $^{(1)}$ on any set of suspension data refers to the subset whose associated zippered rectangles have unit area: for
example 
%\marginpar{$\protect\F^{(1)},\protect\C,X^{(1)}$}
 $X_{\RR}^{(1)}, X_\pi^{(1)}$ etc.
\begin{thm}[Veech \cite{VEECHGAUSS}]
\label{thm:veech}For any connected component $\C$ of the stratum
$\H^{(1)}(\k)$ there is a Rauzy class $\RR=\RR(\C)$ such that $\zip(X^{(1)}_{\RR})\subset\C$
and $\zip(X^{(1)}_{\RR})$ has full measure w.r.t $\nu_{\C}$.
\end{thm}

There is a natural identification 
\begin{equation}
\R^{\A}/\ker\Omega_{\pi}\cong H_{1}(\zip(\pi,\lambda,\tau),\R)\label{eq:homology identification}
\end{equation}
for each $(\pi,\lambda,\tau)\in X_{\pi}$. This descends to an isomorphism
of integral symplectic lattices

\begin{equation}
(\Z^{\A}/\ker(\Omega_{\pi}\lvert_{\Z^{\A}}),\omega_{\pi})\cong(H_{1}(S,\Z),\cap).\label{eq:integral-isomorphism}
\end{equation}
Therefore the pull back of the Hodge bundle to $X_{\pi}$ via $\zip$
is naturally trivialized: 
\begin{equation}
[\zip^{*}H_{1}(\H(\k))]\lvert_{X_{\pi}}\cong X_{\pi}\times\R^{\A}/\ker\Omega_{\pi}.\label{eq:trivialization}
\end{equation}
For a detailed discussion of this map see Viana \cite[Section 2.9]{VIANA}.
The bilinear form 
\[
(v_{1},v_{2})\mapsto\langle\lambda,-\Omega_{\pi}\tau\rangle
\]
descends to a nondegenerate symplectic form
%\marginpar{$\omega_{\pi}$}
$\omega_{\pi}$ on $\R^{\A}/\ker\Omega_{\pi}$. Under the identification
(\ref{eq:homology identification}), the form $\omega_{\pi}$ is precisely
the intersection form on homology. We also note here that the area
of $\zip(\pi,\lambda,\tau)$ is given by 
\begin{equation}
\area(\zip(\pi,\lambda,\tau))=\langle\lambda,-\Omega_{\pi}\tau\rangle.\label{eq:area_formula}
\end{equation}

\subsection{The Rauzy induction map}

Given $\pi,$ let $\alpha$ be the last element of the top row of
$\pi$ and $\beta$ the last element of the bottom row. Say that a
pair $(\pi,\lambda)$ has type \emph{top} if $\lambda_{\alpha}>\lambda_{\beta}.$
Say it has type \emph{bottom} if $\lambda_{\beta}<\lambda_{\alpha}$.
This splits each cell into two pieces of the form 
%\marginpar{$X_{\pi,\protect\e}$}
\[
X_{\pi,\epsilon}=\{\:(\pi,\lambda,\tau)\in \{\pi \}\times \R_+^\A\times \K_\pi \text{\::\:(\ensuremath{\pi},\ensuremath{\lambda}) of type \ensuremath{\epsilon}\;\}},\quad\epsilon\in\{t,b\}
\]
together with a hyperplane. We also introduce
%\marginpar{$Y_{\pi,\protect\e}$}
$Y_{\pi,\epsilon}=\{(\pi,\lambda) \in \{\pi \}\times \R_+^\A \:  \:\text{of type \ensuremath{\e}}\:\},$
so that 
\[
X_{\pi,\epsilon}=Y_{\pi,\epsilon}\times\K_{\pi}.
\]

We now give an assignment of a linear map
%\marginpar{$\Theta_{\pi,\epsilon}$}
$\Theta_{\pi,\epsilon}:\R^{\A}\to\R^{\A}$ to each pair $(\pi,\epsilon)$.
This is given by \cite[(1.9),(1.10)]{VIANA}
\begin{equation}
[\Theta_{\pi,\epsilon}]_{\alpha,\beta}:=\begin{cases}
1 & \text{if }\mbox{\ensuremath{\alpha}=\ensuremath{\beta}}\\
1 & \text{if }\mbox{\ensuremath{\alpha}}\text{ loses and \ensuremath{\beta} wins in type \ensuremath{\epsilon} move at \ensuremath{\pi}}\\
0 & \text{else.}
\end{cases}\label{eq:Theta-defn}
\end{equation}
If $\pi'$ is obtained from $\pi$ by a type $\epsilon$ move then
the map\footnote{Here and henceforth a $*$ denotes a transpose with respect to the
standard basis of $\R^{\A}$.} 
$\Theta^*_{\pi,\epsilon}$ maps $Y_{\pi'}:=\{\pi'\}\times\R_+^\A$ homeomorphically to $Y_{\pi,\epsilon}$.
Furthermore $(\Theta_{\pi,\epsilon}^{*})^{-1}$ maps $\K_{\pi}$ injectively
into $\K_{\pi'}$ \cite[Lemma 2.13]{VIANA}. We also
 have the intertwining relation
\begin{equation}
\Theta_{\pi,\epsilon}\Omega_{\pi}\Theta_{\pi,\epsilon}^{*}=\Omega_{\pi'}.\label{eq:theta_omega_compatibility}
\end{equation}

The \emph{Rauzy induction map} on suspension data is given by 
%\marginpar{$\protect\hRI$}
\[
\hRI(\pi,\lambda,\tau):=(\pi',(\Theta_{\pi,\e}^{*})^{-1}\lambda,(\Theta_{\pi,\e}^{*})^{-1}\tau)
\]
 when $(\pi,\lambda,\tau)\in X_{\pi,\e}$; here again $\pi'$ is obtained
from $\pi$ by an operation of type $\e$. Using the same notation,
notice that $\hRI$ is a skew extension of the map\footnote{As a comment for the initiated, the map $Q$ is the Rauzy induction
map on Interval Exchange Transformations. See \cite{RAUZY} for Rauzy's
original analysis of this map.} 
%\marginpar{$Q$}
\[
Q(\pi,\lambda)=(\pi',(\Theta_{\pi,\e}^{*})^{-1}\lambda).
\]
The equation (\ref{eq:theta_omega_compatibility}) together with the
area formula (\ref{eq:area_formula}) shows that $\hRI$ preserves
the area of the associated zippered rectangles. 
Hence $\hat{Q}$ preserves $X^{(1)}$.

The zippered
rectangles associated to $(\pi,\lambda,\tau)$ define the same point
in $\H(\k)$ as the zippered rectangles associated to $\hRI(\pi,\lambda,\tau),$
that is, 
\[
\zip\circ\hRI=\zip.
\]
See  Viana \cite[Section 2.8]{VIANA} for a clear explanation
of this fact.

We now define cylinders for the Rauzy induction map. Let $\gamma$
be a path in the Rauzy diagram associated to the class $\RR$. Throughout
the rest of the paper, we consider oriented paths that follow the
given direction of the edges\footnote{While it is not immediately obvious, the equivalence classes induced
by identifying end points of oriented paths coincide with the Rauzy
classes \cite[Lemma 1.23]{VIANA}.}. Suppose that $\gamma$ traverses vertices $\pi(0),\pi(1),\ldots,\pi(N)$
in that order. Then define 
%\marginpar{$X_{\gamma}$}
\[
X_{\gamma}:=X_{\pi(0)}\cap\hRI^{-1}(X_{\pi(1)})\cap\hRI^{-2}(X_{\pi(2)})\cap...\cap\hRI^{-N}(X_{\pi(N)}).
\]
 Notice that $X_{\pi,\e}$ is the same as $X_{\gamma}$ where $\gamma$
is the outgoing type $\e$ arrow from $\pi$. We then define
%\marginpar{$\Theta_{\gamma}$}
$\Theta_{\gamma}$ in terms of the $\Theta_{\pi,\epsilon}$ by stating
that for $(\pi,\lambda,\tau)\in X_{_{\gamma}}$ we have
\[
\hRI^{N}(\pi(0),\lambda,\tau)=(\pi(N),(\Theta_{\gamma}^{*})^{-1}\lambda,(\Theta_{\gamma}^{*})^{-1}\tau).
\]
We define 
%\marginpar{$Y_{\gamma}$} 
$Y_{\gamma}=Y_{\pi(0)}\cap\ldots\cap Q^{-N}(Y_{\pi(N)})$
the analogous cylinder for $Q$. If $\gamma$ begins at $\pi$ then
we define the subcone of $\K_{\pi}$ 
%\marginpar{$\protect\K_{\gamma}$}
\[
\K_{\gamma}:=(\Theta_{\gamma}^{*})^{-1}\K_{\pi}.
\]

\subsection{The Rauzy-Veech group\label{sub:The-Rauzy-Veech-group}}

Observe that $\Theta_{\pi,\epsilon}^{*}$ induces a map $\Z^{\A}/\ker\Omega_{\pi'}\to\Z^{\A}/\ker\Omega_{\pi}$
in light of (\ref{eq:theta_omega_compatibility}) and the fact that
$\Theta_{\pi,\epsilon}$ is integral from (\ref{eq:Theta-defn}).
These facts are discussed by Viana in \cite[Section 2.8]{VIANA}.
As a consequence, (\ref{eq:theta_omega_compatibility}) implies that
if $\gamma$ begins and ends at $\pi$, $\Theta_{\gamma}^{*}$ induces
a symplectic endomorphism of $(\Z^{\A}/\ker(\Omega_{\pi}\lvert_{\Z^{\A}}),\omega_{\pi})\cong^{\eqref{eq:integral-isomorphism}}(H_{1}(S,\Z),\cap)$.
In fact it is easy to check from (\ref{eq:Theta-defn}) that $\Theta_{\gamma}^{*}$
is an automorphism. We therefore view each 
\[
\Theta_{\gamma}^{*}\in\Sp(\Z^{2g},\omega_{\pi}).
\]

For each $\pi\in\RR$ let
%\marginpar{$G_{\pi}$} 
$G_{\pi}$ be the
subgroup of $\Sp(\Z^{2g},\omega_{\pi})$ generated by the $\Theta_{\gamma}^{*}$
obtained as $\gamma$ ranges over loops in $\RR$ beginning and ending
at $\pi.$ This group $G_{\pi}$ is called the \emph{Rauzy-Veech group}
at $\pi.$ 

The key property of $G_{\pi}$ that we rely on is the following recent theorem of Guti\'{e}rrez-Romo \cite[Theorem 1.1]{GR} that was previously known for certain \emph{hyperelliptic} components by work of Avila, Matheus, and Yoccoz \cite[Theorem 1.1]{AMY}.

\begin{thm}[Guti\'{e}rrez-Romo, Avila-Matheus-Yoccoz]
\label{thm:zorich-conjecture} For any Rauzy class $\RR$ there exists $\pi\in\RR$ such that $G_{\pi}$ contains the principal congruence subgroup of level 2 of $\Sp(\Z^{2g},\omega_{\pi})$. Recall the principal congruence subgroup of level 2  is the kernel of reduction modulo 2.
\end{thm}

Theorem \ref{thm:zorich-conjecture} resolved, in a strong form,  the conjecture of Zorich \cite[Appendix A.3 Conjecture 5]{ZORICHLEAVES} that $G_\pi$ should be Zariski-dense.

\subsection{Relationship to the Hodge bundle}

Let $\C$ be a connected component of $\H^{(1)}(\k)$ and let $\sigma$
be a section of the Hodge bundle $H_{1}(\C)$. The
pullback of $\sigma$ to any $X^{(1)}_{\pi}$ under the zippered rectangles
map can be naturally viewed as a
\[
\R^{\A}/\ker\Omega_{\pi}
\]
valued function $\tilde{\sigma}$ via the identifications (\ref{eq:homology identification})
and (\ref{eq:trivialization}). Since Rauzy induction does not change
the modulus of zippered rectangles, the fibre of $\zip^{*}H_{1}(\C)$
at $(\pi,\lambda,\tau)$ should be identified with the fibre at $\hRI(\pi,\lambda,\tau)$.
In fact, the identification involves the previously defined map $\Theta_{\gamma}$
and requires for $(\pi,\lambda,\gamma)\in X^{(1)}_{\pi,\epsilon}$ that
if $\pi'$ is the result of applying a type $\epsilon$ move to $\pi$
then
\begin{equation}
\tilde{\sigma}(\pi,\lambda,\tau)=\Theta_{\pi,\epsilon}^{*}\tilde{\sigma}(\hRI(\pi,\lambda,\tau)).\label{eq:compatibility1}
\end{equation}

The iterated form of the compatibility equation (\ref{eq:compatibility1})
that we will use is the following. If $\gamma$ is a path of $N$
edges in a Rauzy diagram that begins and ends at $\pi$, then for
$(\pi,\lambda,\tau)\in X^{(1)}_{\gamma}$ we have
\[
\tilde{\sigma}(\pi,\lambda,\tau)=\Theta_{\gamma}^{*}\tilde{\sigma}(\hRI^{N}(\pi,\lambda,\tau)).
\]
This is an important point of this paper as it describes the equivariance
properties of sections of the Hodge bundle in the suspension model.
We now extend this formula to the setting of associated orbifold vector
bundles $H_{1}(\C;\rho).$ After fixing $\pi,$ using the isomorphism
(\ref{eq:integral-isomorphism})  we identify
\[
\Gamma_{q}\cong\Sp((\Z/q\Z)^{2g},\omega_{\pi})
\]
so we may view $\rho_{q}$ and $\rho_{q}^{\new}$ as representations
of $\Sp(\Z^{2g},\omega_{\pi})$ that are submodules of\\
 $\ell^{2}(\Sp((\Z/q\Z)^{2g},\omega_{\pi}))$. More generally, using
(\ref{eq:integral-isomorphism}) we may pull back any unitary representation
$(\rho,V)$ of $\Sp(H_{1}(S;\Z),\cap)$ to a representation of $\Sp(\Z^{2g},\omega_{\pi})$
that we also call $\rho$. 

We may now argue by analogy with the Hodge bundle that if $\sigma$
is any section of the associated bundle $H_{1}(\C;\rho)$ then the
pull back $\tilde{\sigma}$ of this section to a $V$-valued function
on $X^{(1)}_{\pi}$ satisfies
\begin{equation}
\tilde{\sigma}(\pi,\lambda,\tau)=\rho(\Theta_{\gamma}^{*})\tilde{\sigma}(\hRI^{N}(\pi,\lambda,\tau)),\quad(\pi,\lambda,\tau)\in X^{(1)}_{\gamma}\label{eq:compatibility-section}
\end{equation}
 for each path $\gamma$ in $\RR$ of length $N$ beginning and ending
at $\pi.$

\subsection{A fundamental domain}

There is a nice fundamental domain for Rauzy induction on $X$ described
in \cite[pg. 159]{AGY}. We let $\F=\F_{\RR}$ denote the set of $(\pi,\lambda,\tau)$
such that either
\begin{enumerate}
\item $\hRI(\pi,\lambda,\tau)=(\pi',\lambda',\tau')\text{ is defined and }\|\lambda'\|<1\leq\|\lambda\|$
\item $\hRI(\pi,\lambda,\tau)\text{ is not defined and }1\leq\|\lambda\|$
\item $\hRI^{-1}(\pi,\lambda,\tau)\text{ is not defined and }\|\lambda\|<1.$
\end{enumerate}
The norm we use is $\|\lambda\|:=\sum_{\alpha\in\A}|\lambda_{\alpha}|$.
The fibres of the zippered rectangles map
\[
\zip:\F^{(1)} \to\C,\quad\C=\C(\RR)
\]
are almost everywhere finite with constant cardinality depending on
$\C$.

\subsection{The Teichm\"{u}ller flow on suspension data}

Recall that $\C$ is a connected component of $\H^{(1)}(\k)$ and $\RR$
the associated Rauzy class. The \emph{Teichm\"{u}ller} \emph{flow
}is a one parameter flow on $X_{\RR}$ that commutes with $\hRI$
and is given by $T_{t}(\pi,\lambda,\tau)=(\pi,e^{t}\lambda,e^{-t}\tau).$
Note that this preserves each $X_{\pi}^{(1)}$ and $X^{(1)}.$ The
flow $T_{t}$ lifts the Teichm\"{u}ller flow on $\C$, that is,
\[
\T_{t}\circ\zip=\zip\circ T_{t}.
\]
Evidently, $T_{t}$ preserves Lebesgue measure on $X.$ The flow $T_{t}$
also preserves Lebesgue measure on $X^{(1)}$, the pushforward of
which under $\zip$ is a multiple of $\nu_{\C}$.

\subsection{Time acceleration and renormalization.\label{sub:Time-acceleration-and}}

The approach of Avila, Gou\"{e}zel and Yoccoz \cite{AGY} to the
Teichmüller flow is to consider the first return time to an appropriately
chosen cross section. This cross section involves the choice of $\pi\in\RR$
and a path $\gamma_{0}$ that begins and ends at $\pi$. We give details
on the choice of $\gamma_{0}$ in Section \ref{sub:The-choice-of-g0}
and \ref{sub:Refining-the-choice-of-gamma-0}. For now, assume we
have chosen $\pi$ and $\gamma_{0}.$ 

We will use the notation 
%\marginpar{$\protect\F_{\pi}$, $\protect\F_{\gamma}$}\
$\F_{\pi}=\F\cap X_{\pi}$
and $\F_{\gamma}=\F\cap X_{\gamma}$. We consider the regions
\[
\hat{\Xi}:=\{\:(\pi,\lambda,\tau)\in\F_{\gamma_{0}}^{(1)}\::\:\|\lambda\|=1\:\}\cap(\{\pi\}\times Y_{\gamma_{0}}\times\K_{\gamma_{0}})
\]
and the closely related
\[
\Xi:=\{\:(\pi,\lambda)\in Y_{\gamma_{0}}\::\:\|\lambda\|=1\:\}.
\]
Let $\hat{m}$ (resp. $m)$ denote the normalized natural Lebesgue
measure on $\hat{\Xi}$ (resp. $\Xi$). It is known that almost all
orbits of the Teichmüller flow pass through $\hRI^{\Z}(\Xi)$, this
is stated in \cite[4.1.3]{AGY} as a consequence of the ergodicity
of the Veech flow\footnote{The Veech flow is not discussed in the current paper.}.\emph{
}For each $x\in\hat{\Xi}$ we denote by $r(x)$ the first return time
of $x$ to $\hRI^{\Z}(\Xi)$ under the Teichm\"{u}ller flow. That
is, 
%\marginpar{$r$} 
$r(x)$ is the smallest positive value such that
\[
T_{r(x)}(x)\in\hRI^{-n}(\hat{\Xi})
\]
for some positive\footnote{Notice that from (\ref{eq:Theta-defn}) that $\Theta$ does not decrease
norms, so if $(\pi',\lambda',\tau')=\hRI(\pi,\lambda,\tau)$ then
$\|\lambda'\|\leq\|\lambda\|.$} integer $n$. This means there is some value
%\marginpar{$\hat{Z}$}
$\hat{Z}(x)\in\hat{\Xi}$ such that
\begin{equation}
T_{r(x)}\hRI^{n}(x)=\hRI^{n}T_{r(x)}(x)=\hat{Z}(x).\label{eq:first-return}
\end{equation}
Suppose that $x=(\pi,\lambda,\tau)\in X_{\gamma}$ with $\hRI^{n}(x)=(\pi,(\Theta_{\gamma}^{*})^{-1}\lambda,(\Theta_{\gamma}^{*})^{-1}\tau)\in\hat{\Xi}.$
Then 
\[
r(x)=-\log\|(\Theta_{\gamma}^{*})^{-1}\lambda\|.
\]
Note here that $r(\pi,\lambda,\tau)$ depends only on the coordinates
$(\pi,\lambda)$ and we can view $r$ also as a function on $\Xi$.

We will write $\gamma_{1}.\gamma_{2}$ or just $\gamma_{1}\gamma_{2}$
for the concatenation of two oriented paths $\gamma_{1}$ and $\gamma_{2}$
in $\RR$ with compatible endpoints. In $\gamma_{1}.\gamma_{2},$
$\gamma_{1}$ is the first path traversed. Consider $\gamma$ with
the property that the $\gamma_{0}$ subpaths of $\gamma.\gamma_{0}$
are precisely the beginning and the end segment. We say that such
a $\gamma$ is \emph{$\gamma_{0}$-adapted.}
%\marginpar{\emph{$\gamma_{0}$-adapted}}
For such a $\gamma$, if $x\in X_{\gamma.\gamma_{0}}\cap\hat{\Xi}$
then 
\[
\hat{Z}(x)=\left(\pi,\frac{(\Theta_{\gamma}^{*})^{-1}\lambda}{\|(\Theta_{\gamma}^{*})^{-1}\lambda\|},\|(\Theta_{\gamma}^{*})^{-1}\lambda\|(\Theta_{\gamma}^{*})^{-1}\tau\right).
\]
The domain of $\hat{Z}$ is therefore $\cup_{\gamma_{0}\text{-adapted }\gamma}\:\hat{\Xi}_{\gamma\gamma_{0}}$
where
%\marginpar{$\hat{\Xi}_{\gamma\gamma_{0}}$} 
\[
\hat{\Xi}_{\gamma\gamma_{0}}:=\hat{\Xi}\cap(Y_{\gamma\gamma_{0}}\times\K_{\gamma_{0}}).
\]
We extend this definition to $\hat{\Xi}_{\gamma_{1}\ldots\gamma_{N}\gamma_{0}}:=\hat{\Xi}\cap(Y_{\gamma_{1}\ldots\gamma_{N}\gamma_{0}}\times\K_{\gamma_{0}})$
where $\gamma_{1},\ldots,\gamma_{N}$ are a sequence of $\gamma_{0}$-adapted
paths with both endpoints equal to $\pi$.

Notice that the mapping $\hat{Z}$ has the following properties.
\begin{enumerate}
\item $\hat{Z}$ is a skew extension of the mapping $Z:\Xi\to\Xi$ defined
Lebesgue almost everywhere by 
\[
Z(\pi,\lambda)=\left(\pi,\frac{(\Theta_{\gamma}^{*})^{-1}\lambda}{\|(\Theta_{\gamma}^{*})^{-1}\lambda\|}\right),\quad(\pi,\lambda)\in Y_{\gamma.\gamma_{0}}.
\]
The connected components of the domain of $Z$ are the sets 
%\marginpar{$\Xi_{\gamma\gamma_{0}}$}
\[
\Xi_{\gamma\gamma_{0}}:=\Xi\cap Y_{\gamma\gamma_{0}}.
\]

\item The maps $\hat{Z}$ and $Z$ preserve $\|\lambda\|=1.$ This is usually
referred to as \emph{renormalization}.
\item The maps $\hat{Z}$ and $Z$ involve many iterations of Rauzy induction
and this is usually referred to as \emph{time acceleration}. This
is first due to Zorich \cite{ZORICHGAUSS}, see also \cite[Section 5.3 ]{zorichflat}
for further discussion.
\item $\hat{Z}$ (resp. $Z$) preserves the Lebesgue measure $\hat{m}$
(resp. $m$).
\end{enumerate}

Following \cite[Section 4.2.1]{AGY}, in order to enforce hyperbolicity
of the map $\hat{Z}$ (cf. Proposition \ref{prop:The-map-is uniformly epxanding markov}
and Lemma \ref{lem:-is-a-hyperbolic-skew-product}) one puts adapted
metrics on $\Xi$ and $\hat{\Xi}$. On $\Xi$ we put the Hilbert metric
$d_{\Xi}$ coming from the inclusion $\Xi\to Y_{\pi}$ and on $\hat{\Xi}$
we consider the product metric
\[
d_{\hat{\Xi}}((\pi,\lambda,\tau),(\pi,\lambda',\tau')):=d_{\Xi}((\pi,\lambda),(\pi,\lambda'))+d_{\K_{\pi}}(\tau,\tau')
\]
 where $d_{\K_{\pi}}$ is the Euclidean distance in $\K_{\pi}$. These
metrics induce Finsler metric structures on $\Xi$ and $\hat{\Xi}$
that make them into complete Finsler manifolds.

\subsection{Flow on sections of associated bundles in the suspension model\label{sub:Koopman}}

We may now map 
%\marginpar{$\hat{\Xi}_{r}$}
\[
\hat{\Xi}_{r}:=\{(x,s):x\in\hat{\Xi},\:s\in[0,r(x))\}
\]
 homeomorphically to a part of $X_{\pi}^{(1)}$ by the map

\begin{equation}
P:(x,s)\mapsto T_{s}(x).\label{eq:suspension-mapping}
\end{equation}
The image $X_{\pi}^{\prime(1)}$ of $P$ is up to a Lebesgue-null
set, a fundamental domain for the action of $\hRI$ on $X^{(1)}$.
Given a section of $H_{1}(\C;\rho)$, its pull back to $X^{(1)}$
is therefore determined (up to a zero measure set) by its values on
$X_{\pi}^{\prime(1)}\subset X_{\pi}^{(1)}$.

The pushforward of Lebesgue measure under the mapping in (\ref{eq:suspension-mapping})
is Lebesgue measure. We write
%\marginpar{$\hat{m}_{r}$} 
$\hat{m}_{r}=\hat{m}\otimes\Leb$
for the Lebesgue measure on $\hat{\Xi}_{r}$.

As explained in Section \ref{sub:The-Teichmuller-flow}, $\T_{t}$
acts by $\T_{t}^{*}$ on sections of $H_{1}(\C;\rho)$. If (after
pullback) we view a section $\tilde{\sigma}$ as a $V$-valued function
satisfying (\ref{eq:compatibility-section}) and then view $\tilde{\sigma}$
as a $V$-valued function $\hat{\sigma}$ on $\hat{\Xi}_{r}$ by the
mapping in (\ref{eq:suspension-mapping}) then the action of $T_{t}^{*}$
on $\hat{\sigma}$ will be denoted by $\hat{T}_{t}^{*}$ and defined
as follows. Let $\gamma$ be $\gamma_{0}$-adapted with $l(\gamma)=n$.
If $x\in X'_{\gamma.\gamma_{0}}\cap\hat{\Xi}$ and $t+s\in[r(x),r(x)+r(\hat{Z}(x)))$
then 
\begin{eqnarray*}
[\hat{T}_{t}^{*}\hat{\sigma}](x,s) & = & \hat{\sigma}(x,t+s)\\
 & = & \tilde{\sigma}(T_{t+s}x)=^{\eqref{eq:compatibility-section}}\rho(\Theta_{\gamma}^{*})\tilde{\sigma}(\hRI^{n}T_{t+s}x)\\
 & =^{\eqref{eq:first-return}} & \rho(\Theta_{\gamma}^{*})\tilde{\sigma}(T_{t+s-r(x)}\hat{Z}(x))=\rho(\Theta_{\gamma}^{*})\hat{\sigma}(\hat{Z}(x),t+s-r(x)).
\end{eqnarray*}
Let 
\[
r^{(N)}(x):=r(x)+r(\hat{Z}(x))+\ldots+r(\hat{Z}^{N-1}(x)).
\]
For $\gamma_{1},\gamma_{2},\ldots,\gamma_{N}$ each $\gamma_{0}$-adapted,
$t+s\in[r^{(N)}(x),r^{(N+1)}(x))$ and $x\in X'_{\gamma_{1}.\gamma_{2}.\ldots.\gamma_{N}.\gamma_{0}}\cap\hat{\Xi}$
we have then 
%\marginpar{$\hat{T}_{t}^{*}$}
\begin{equation}
[\hat{T}_{t}^{*}\hat{\sigma}](x,s)=\rho(\Theta_{\gamma_{1}}^{*})\rho(\Theta_{\gamma_{2}}^{*})\ldots\rho(\Theta_{\gamma_{N}}^{*}).\hat{\sigma}(\hat{Z}^{N}(x),t+s-r^{(N)}(x)).\label{eq:pull-back-sections}
\end{equation}
This is the master equation for the Teichmüller flow on sections of
$H_{1}(\C;\rho)$ in our suspension model. Notice that the argument
of $\hat{\sigma}$ in the right hand side of (\ref{eq:pull-back-sections})
defines a mapping we call 
%\marginpar{$\hat{T}_{t}$}
\[
\hat{T}_{t}:\hat{\Xi}_{r}\to\hat{\Xi}_{r},\quad\hat{T}_{t}(x,s):=(\hat{Z}^{N}(x),t+s-r^{(N)}(x))
\]
for $\:x\in X'_{\gamma_{1}.\gamma_{2}.\ldots.\gamma_{N}.\gamma_{0}}\cap\hat{\Xi}$
and $t+s\in[r^{(N)}(x),r^{(N+1)}(x)).$ Then $\hat{T}_{t}$ is the
suspension flow over $\hat{Z}$ with roof function $r.$ The flow
$\hat{T}_{t}$ lifts the Teichmüller flow under the mapping in (\ref{eq:suspension-mapping})
and as a consequence, Lebesgue measure $\hat{m}_{r}$ on $\hat{\Xi}_{r}$
is invariant under $\hat{T}_{t}$.

Since the roof function $r$ depends only on a coordinate in $\Xi$
we may also define 
%\marginpar{$\Xi_{r}$}
\[
\Xi_{r}=\{(y,s):y\in\Xi,\:s\in[0,r(y))\}.
\]
We write %\marginpar{$m_{r}$ } 
$m_{r}$ for the Lebesgue measure on
$\Xi_{r}$. We also define for $r\in Z^{-(N-1)}(\Xi)$
\[
r^{(N)}(y):=r(y)+r(Z(y))+\ldots+r(Z^{N-1}(y)).
\]
We may define a similar operator to $\hat{T}_{t}^{*}$ that we will
call $T_{t}^{*}$ that will act on $V$-valued functions on $\Xi_{r}$.
For $\sigma:\Xi_{r}\to V$, $\gamma_{1},\gamma_{2},\ldots,\gamma_{N}$
each $\gamma_{0}$-adapted, $t+s\in[r^{(N)}(y),r^{(N+1)}(y))$ and
$y\in\Xi_{\gamma_{1}.\gamma_{2}.\ldots.\gamma_{N}.\gamma_{0}}$ we
define
\begin{equation}
[T_{t}^{*}\sigma](y,s)=\rho(\Theta_{\gamma_{1}}^{*})\rho(\Theta_{\gamma_{2}}^{*})\ldots\rho(\Theta_{\gamma_{N}}^{*}).\sigma(Z^{N}(y),t+s-r^{(N)}(y)).\label{eq:master-T_t-on-Xi-functions}
\end{equation}
We give $\Xi_{r}$ and $\hat{\Xi}_{r}$ Finsler metrics that are the
product of the Finsler metric on $\Xi$ (resp. $\hat{\Xi}$) with
the usual metric in the $s$ direction.

\subsection{Preliminary choice of $\gamma_{0}$.\label{sub:The-choice-of-g0}}

Recall $\gamma_{0}$ is a path in $\RR$ beginning and ending in $\pi$.
We now explain the choice of $\gamma_{0}$ that is made in \cite{AGY}.
Avila, Gou\"{e}zel and Yoccoz require that 
\begin{description}
\item [{(Strongly~Positive)}] $\gamma_{0}$ is a \emph{strongly positive}
path, meaning that all the entries of $\Theta_{\gamma}^{*}$ are positive
and moreover $(\Theta_{\gamma_{0}}^{*})^{-1}$ maps $\overline{\K}_{\pi}-\{0\}$
into $\K_{\pi}$.
\item [{(Neat)}] $\gamma_{0}$ is \emph{neat,} meaning that $\gamma_{0}=\gamma'\gamma_{e}=\gamma_{s}\gamma'$
implies $\gamma'$ is trivial or $\gamma'=\gamma_{0}$. This means
in any path, occurrences of $\gamma_{0}$ are (edge) disjoint. Therefore
$\gamma_{0}$-adapted $\gamma$ are precisely those of the form 
\[
\gamma=\gamma_{0}\gamma'.
\]
where $\gamma'$ does not contain $\gamma_{0}$ as a subpath.
\end{description}

According to \cite[Section 4.13]{AGY}, such a choice of $\gamma_{0}$
is possible. However, in the present paper, we must choose $\gamma_{0}$
more carefully, while still making sure $\gamma_{0}$ is strongly
positive and neat. This is done in Section \ref{sub:Refining-the-choice-of-gamma-0}.
For now, assume that $\gamma_{0}$ is strongly positive and neat.

\section{Decay of correlations}\label{sec:decay}

In this section we state in more precise terms and then prove Theorem
\ref{main}.{\bf C} on uniform exponential decay
of correlations.

\subsection{Dynamical setup}

The following definitions and results are from \cite{AGY}. Recall
the maps $\hat{Z}$ and $Z$ introduced in Section \ref{sub:Time-acceleration-and}.
Throughout we use the Finsler metric on the tangent bundle to $\Xi$
defined in Section \ref{sub:Time-acceleration-and}. We write $D$
for the total derivative of a function. We write $C^{0}(\Xi)$ for
the uniform norm. For a $V$-valued function $F$, $\|DF\|$ refers
to operator norm w.r.t. the Finsler metric on the fibres and the Hilbert
space metric on $V$. When we write
%\marginpar{$\bigcup_{\gamma}^{*}$}
$\bigcup_{\gamma}^{*}$ or 
%\marginpar{$\sum_{\gamma}^{*}$} 
$\sum_{\gamma}^{*}$
it means that we restrict the indexing to $\gamma_{0}$-adapted $\gamma$.
We assume here that $\gamma_{0}$ is strongly positive and neat as
in Section \ref{sub:The-choice-of-g0}, since these conditions are required for
the results of  Avila, Gou\"{e}zel and Yoccoz \cite{AGY}.
\begin{prop}[{\cite[Proof of Proposition 4.3]{AGY}}]
\label{prop:The-map-is uniformly epxanding markov}The map $Z$ is
a \textbf{uniformly expanding Markov map} with respect to Lebesgue
measure $m$ and the Finsler metric structure defined in Section \ref{sub:Time-acceleration-and}.
That is to say
\begin{enumerate}
\item The union 
\[
\bigcup_{\gamma}^{*}\Xi_{\gamma\gamma_{0}}
\]
is a countable union of open sets that are $m$-conull in $\Xi$.
\item \label{enu:expanding}If $\gamma$ is $\gamma_{0}$-adapted, $Z$
maps $\Xi_{\gamma\gamma_{0}}$ diffeomorphically to $\Xi$ and there
are constants
%\marginpar{$\Lambda$} 
$\Lambda>1$ and $c_{1}(\gamma)>0$
such that for all $x\in\Xi_{\gamma\gamma_{0}}$ and $v$ in the tangent
fibre to $x$
\[
\Lambda\|v\|\leq\|[DZ]_{x}.v\|\leq c_{1}(\gamma)\|v\|.
\]

\item \label{enu:uniformly-expanding-logJ-deriv}Let $J$ denote the
inverse of the Jacobian of $Z$ with respect to $m.$ The function
$\log J$ is $C^{1}$ on each $\Xi_{\gamma\gamma_{0}}$ and there
is some $C>0$ such that for any inverse branch $\alpha$ of $Z$,
\[
\sup_{y\in\Xi}\|D(\log J\circ\alpha)(y)\|\leq C.
\]

\end{enumerate}
\end{prop}

\begin{lem}[{\cite[Lemma 4.3]{AGY}}]
\label{lem:-is-a-hyperbolic-skew-product}The pair $(\hat{Z},\hat{m})$
is a \textbf{hyperbolic skew product} over $(Z,m)$. This means, with
all norms and distances coming from the Finsler metric on $\hat{\Xi}$
defined in Section \ref{sub:Time-acceleration-and},
\begin{enumerate}
\item The projection $\pr:\hat{\Xi}\to\Xi$ defined by
%\marginpar{$\protect\pr$}
\[
\pr(\pi,\lambda,\tau)=(\pi,\lambda)
\]
satisfies $Z\circ\pr=\pr\circ\hat{Z}$ whenever both sides of the
equality are defined.
\item The measure $\hat{m}$ gives full mass to the domain of definition
of $\hat{Z}$.
\item There is a family of probability measures $\{\hat{m}_{y}\}_{y\in\Xi}$
on $\hat{\Xi}$ which is a disintegration of $\hat{m}$ over $m$
in the following sense: $y\mapsto\hat{m}_{y}$ is measurable, $\hat{m}_{y}$
is supported on $\pr^{-1}(y)$ and for any measurable $U\subset\hat{\Xi,}\mbox{ }$
$\hat{m}(U)=\int_{y\in\Xi}\hat{m}_{y}(U)dm(y)$. Moreover, there is
a constant $C>0$ such that for any open $V\subset Z^{-1}(\Xi)$,
for any $u\in C^{1}(\pr^{-1}(V))$ the function $\bar{u}(x)=\int u(x)d\hat{m}_{y}(x)$
is in $C^{1}(V)$ with
\[
\sup_{y\in V}\|D\bar{u}(x)\|\leq C\sup_{x\in\pr^{-1}(V)}\|Du(y)\|.
\]

\item There is a constant $K>1$ such that for all $x_{1},x_{2}\in\hat{\Xi}$
with $\pr(x)=\pr(y)$ we have
\[
d_{\hat{\Xi}}(\hat{Z}(x_{1}),\hat{Z}(x_{2}))\leq K^{-1}d_{\hat{\Xi}}(x_{1},x_{2}).
\]

\end{enumerate}
\end{lem}

\begin{lem}[{\cite[Lemma 4.5]{AGY}}]
\label{lem:The-roof-function-is-good}The roof function $r$ is \textbf{good}.
This means
\begin{enumerate}
\item There is $\e_{1}>0$ such that $r\geq\e_{1}.$
\item \label{enu:rooffunctiongoodbounded-deriv}There is $C>0$ such that
for any inverse branch $\alpha$ of $Z$ one has
\[
\sup_{y\in\Xi}\|D(r\circ\alpha)(y)\|\leq C.
\]

\item There is no $C^{1}$ function $\phi$ on $\bigcup_{\gamma}^{*}\Xi_{\gamma\gamma_{0}}$
such that 
\[
r-\phi\circ T+\phi
\]
 is constant on each $\Xi_{\gamma\gamma_{0}}.$
\end{enumerate}
\end{lem}
\begin{thm}[{\cite[Theorem 4.6]{AGY}}]
\label{thm:The-roof-function-has-exponential-tails}The roof function
$r$ has \textbf{exponential tails}. This means there is $\sigma_{0}>0$
such that
\[
\int_{\Xi}\exp(\sigma_{0}r)dm<\infty.
\]
\end{thm}

\subsection{The main technical results}

The following will be the precise version of Theorem \ref{main}.{\bf C}.
Recall the definition of $\C(q)$ from Section \ref{sub:Abelian-differentials-and}.
\begin{thm}
\label{thm:decay-of-correlations-precise.} 
There exists $\delta,\eta >0$ and $Q_0 \in \Z_+$ such that for all $q$ coprime to $Q_0$,
for all $u,v\in C^{1}(\C(q))$ whose supports project into a compact set $K \subset \M$, there exists $C=C(K)>0$ such that
for all $t\geq0$
\[
\left|\int u.v\circ\T_{t}\:d\nu_{\C(q)}-| \Gamma_q |^{-1} \left(\int u\:d\nu_{\C(q)}\right)\left(\int v\:d\nu_{\C(q)}\right)\right|\leq C(K) \|u\|_{C^{1}}\|v\|_{C^{1}}q^{\eta}e^{-\delta t}.
\]
\end{thm}

The key feature of this estimate is that $\delta$ does not depend on $q$.

For any Finsler manifold $X$ and Hilbert space $V$ we may define
the Banach space of $C^{1}$ $V$-valued functions on $X$ as in Section
\ref{sub:The-Teichmuller-flow}. Recall from Sections \ref{sub:Time-acceleration-and}
and \ref{sub:Koopman} that there are Finsler metric structures on
$\Xi,\hat{\Xi},\Xi_{r},\hat{\Xi}_{r}$. If $(\rho,V)$ is a unitary
representation we write e.g. $C^{1}(\Xi;\rho)$ for the $C^{1}$ $V$-valued
functions on $\Xi$, with respect to the Finsler metric. We make a
reduction of Theorem \ref{thm:decay-of-correlations-precise.} to
the following that is analogous to \cite[Theorem 2.7]{AGY}.
\begin{thm}
\label{thm:full-C^1-decay-of-correlations}
There exists $C,\delta,\eta >0$ and $Q_0 \in \Z_+$ such that for all $q$ coprime to $Q_0$, for all $U,V\in C^{1}(\hat{\Xi}_{r};\rho_{q})$
and all $t\geq0$ 
\[
\left|\int\langle U,T_{t}^{*}V\rangle d\hat{m}_{r}\Big\rangle\right|\leq C\|U\|_{C^{1}}\|V\|_{C^{1}}q^{\eta}e^{-\delta t}.
\]

\end{thm}

We now explain how Theorem \ref{thm:decay-of-correlations-precise.}  reduces to Theorem \ref{thm:full-C^1-decay-of-correlations}.

\begin{proof}[Passage from Theorem \ref{thm:full-C^1-decay-of-correlations} to Theorem \ref{thm:decay-of-correlations-precise.} ]
Note that in the context of Theorem \ref{thm:decay-of-correlations-precise.}, we can write $u = \tilde{u}_0 + u'$, $v = \tilde{v}_0
 + v'$  with $u',v'\in L^{2}_\star(\C(q) )$ and $\tilde{u}_0, \tilde{v}_0$ given by lifts of functions from $\C$. In other words, if $\cover_q : \C(q) \to \C$ is the covering map, there are functions $u_0$ and $v_0$ such that $\tilde{u}_0 = u_0 \circ \cover_q$ and $\tilde{v}_0 = v_0 \circ \cover_q$. Since $\tilde{u}_0$ (resp. $\tilde{v}_0$) is obtained from $u$ (resp. $v$) by averaging over $\Gamma_q$, and  the Finsler metric on $\X(\kappa)$ is $\Gamma$-invariant, we have estimates
 $$
 \| \tilde{u}_0 \|_{C^1} = \| u_0 \|_{C^1}\leq \| u \|_{C^1},\quad  \| \tilde{v}_0 \|_{C^1} = \| v_0 \|_{C^1}\leq \| v \|_{C^1}
 $$ and hence also by the triangle inequality
 $$
  \| u' \|_{C^1} \leq 2\| u \|_{C^1}, \quad   \| v' \|_{C^1} \leq 2\| v \|_{C^1}.
 $$
Also note that $\int u' d\nu_{\C(q)} = \int v' d\nu_{\C(q)} = 0$. Moreover since the supports of $u$,$v$ project to $K$ in $\M$, the same holds for $\tilde{u}_0 , \tilde{v}_0, u' , v' , u_0 , v_0$.
 Since $\T_t$ preserves $L^2_\star( \C(q) )$ and its orthogonal complement, we have
 $$
 \int u.v\circ\T_{t}\:d\nu_{\C(q)} = \int \tilde{u}_0.\tilde{v}_0\circ\T_{t}\:d\nu_{\C(q)}+ \int u'.v'\circ\T_{t}\:d\nu_{\C(q)}.
 $$
 We can replace the first term by
 $$
 | \Gamma_q | \int u_0.v_0\circ\T_{t} d\nu_\C
 $$
 which by exponential mixing on $\C$ (\cite[Theorem 2.14]{AGY}) is for some $\delta' >0$
 $$
  | \Gamma_q | \left(\int u_0\:d\nu_{\C}\right)\left(\int v_0\:d\nu_{\C}\right) + O_K( |\Gamma_q\| \| u \|_{C^1} \| v \|_{C^1} e^{-\delta' t} ).
 $$
 Notice that since $\delta'$ depends only on $\C$ and for some $\eta>0$, $| \Gamma_q | \ll q^\eta $ for all $q$, the error term here is of the form as in Theorem \ref{thm:decay-of-correlations-precise.}. This also explains why the error term of Theorem \ref{thm:decay-of-correlations-precise.} must contain a $q^\eta$ factor.
 
 Since $\int u_0\:d\nu_{\C} = | \Gamma_q |^{-1} \int \tilde{u}_0 d\nu_{\C(q)}$ and similarly for $v_0 , \tilde{v}_0$, we have by putting the previous arguments together
\begin{align*}
  \int u.v\circ\T_{t}\:d\nu_{\C(q)} &=  | \Gamma_q |^{-1} \left(\int u\:d\nu_{\C(q)}\right)\left(\int v\:d\nu_{\C(q)}\right) +  O_K( q^\eta\| u \|_{C^1} \| v \|_{C^1} e^{-\delta' t} ) \\
  & + \int u'.v'\circ\T_{t}\:d\nu_{\C(q)}.
\end{align*}
 This reduces Theorem \ref{thm:decay-of-correlations-precise.} to the case of $u=u', v=v' \in L^2_\star( \C(q) )$. Now assume this is the case.
 
 We apply Lemma \ref{lem:We-have-the-following-correspondences} to obtain sections $u^* , v^*\in L^2(H_1( \C ; \rho_q))$ that have the same $C^1$ norms as $u$ and $v$.
 To apply Theorem  \ref{thm:full-C^1-decay-of-correlations} to the correlation function of $u^*$ and $v^*$ and conclude the proof, one needs to use the correspondence from Section \ref{sub:Koopman} to lift $u^*$ and $v^*$ to continuously differentiable $V$-valued functions $u^{**}$ and $v^{**}$ on $\hat{\Xi}_{r}$. However, $u^{**}$ and $v^{**}$ may not have bounded $C^1$ norms, because of distortion between the Finsler metric structures on $\hat{\Xi}_{r}$ and $\C$. So one needs to perform some `chopping' and `smoothing' to conclude the result and it is at this stage that the condition on the support of $u^*$ and $v^*$ must be used. One may obtain estimates for $L^p$ norms of $u^*$ and $v^*$ in terms of their $C^1$ norms and the compact set $K$. Once this is done, the rest of the argument is as in \cite[pp. 166-169]{AGY}. It applies in the same way to vector valued functions as to scalar valued functions.
\end{proof}

\subsection{Entrance of the transfer operator}

We now recall the definition of the spaces $\B_{0}$ and $\B_{1}$
from \cite{AGY}. 
\begin{defn}
\label{def:A-function-B_0}A function $U:\Xi_{r}\to V$ is in $\B_{0}(\Xi_{r};\rho)$
if it is bounded, continuously differentiable on each set 
\[
(\Xi_{r})_{\gamma\gamma_{0}}:=\{(y,t)\::\:y\in\Xi_{\gamma\gamma_{0}},\:t\in(0,r(y))\:\}\quad\gamma\text{ is \ensuremath{\gamma_{0}}-adapted}
\]
and also $\sup_{(y,t)\in\bigcup^{*}(\Xi_{r})_{\gamma\gamma_{0}}}\|DU(y,t)\|<\infty$.
Define the norm
\[
\|U\|_{\B_{0}(\Xi_{r};\rho)}:=\sup_{(y,t)\in\bigcup^{*}(\Xi_{r})_{\gamma\gamma_{0}}}\|U(y,t)\|+\sup_{(y,t)\in\bigcup^{*}(\Xi_{r})_{\gamma\gamma_{0}}}\|DU(y,t)\|.
\]

\end{defn}

\begin{defn}
\label{def:A-function-B_1}A function $U:\Xi_{r}\to V$ is in $\B_{1}(\Xi_{r};\rho)$
if it is bounded and there exists a constant $C>0$ such that for
all fixed $y\in\cup^{*}\Xi_{\gamma\gamma_{0}}$, the function $t\mapsto U(y,t)$
is of bounded variation\footnote{We make the obvious extension of bounded variation to $V$-valued
functions using the norm induced by the inner product on the Hilbert
space $V$.} on the interval $(0,r(y))$ and its variation $\Var_{(0,r(y))}(t\mapsto U(y,t))$
is bounded by $Cr(y)$. Let
\[
\|U\|_{\B_{1}}=\sup_{(y,t)\in\bigcup^{*}(\Xi_{r})_{\gamma\gamma_{0}}}\|U(y,t)\|+\sup_{y\in\bigcup^{*}(\Xi)_{\gamma\gamma_{0}}}\frac{\Var_{(0,r(y))}(t\mapsto U(y,t))}{r(y)}.
\]
\end{defn}

As in \cite{AGY} we reduce to decay of correlations for the $\rho$-skew
extension of $\Xi_{r}$ rather than $\hat{\Xi}_{r}$. 
\begin{thm}[Decay of correlations]
\label{thm:Decay of correlations roof over Markov} There exists $C,\delta,\eta >0$ and $Q_0 \in \Z_+$ such that for all $q$ coprime to $Q_0$,
for all $U\in\B_{0}(\Xi_{r};\rho_{q})$
and $V\in\B_{1}(\Xi;\rho_{q})$, for all $t\geq0$,
\[
\left|\int\langle U,T_{t}^{*}V\rangle dm_{r}\right|\leq C q^\eta \|U\|_{\B_{0}}\|V\|_{\B_{1}}e^{-\delta t}.
\]
\end{thm}

This is proved for scalar valued functions in \cite[Theorem 7.3]{AGY}.
The key point of Theorem \ref{thm:Decay of correlations roof over Markov}
is the uniformity in $q$. The passage from Theorem \ref{thm:Decay of correlations roof over Markov}
to Theorem \ref{thm:full-C^1-decay-of-correlations} is handled as
in \cite[Section 8]{AGY}. In fact, the arguments of \cite[Section 8]{AGY}
are followed closely and extended to the skew setting by Oh and Winter
in \cite[Proof of Theorem 1.5]{OW}. So we have presently explained
the reduction of Theorem \ref{thm:decay-of-correlations-precise.}
to Theorem \ref{thm:Decay of correlations roof over Markov} whose
proof we now take up.

From now on, all integrals are taken with respect to the relevant
Lebesgue measure. Following \cite{AGY} let
\[
A_{t}=\{(y,a)\in\Xi_{r}\::\:a+t\geq r(y)\}
\]
and $B_{t}=\Xi_{r}\backslash A_{t}$. We bound 
\[
\int_{B_{t}}\langle U,T_{t}^{*}V\rangle\leq\|U\|_{\B_{0}}\|V\|_{\B_{1}}\int_{y\in\Xi}\max(r(y),0)\leq\|U\|_{\B_{0}}\|V\|_{\B_{1}}\int_{y:r(y)\geq t}r(y).
\]
By the Cauchy-Schwarz inequality and that $r$ has exponential tails (Theorem
\ref{thm:The-roof-function-has-exponential-tails}) the above contributes
$\leq C'\|U\|_{\B_{0}}\|V\|_{\B_{1}}\exp(-\delta't)$ for some $\delta'>0$
and $C'>0$ that do not depend on $U,V$ or $\rho$. Therefore the
proof of Theorem \ref{thm:Decay of correlations roof over Markov}
reduces to estimating the quantity
\[
I(t):=\int_{A_{t}}\langle U,T_{t}^{*}V\rangle
\]
on the order of
\begin{equation}
I(t)\leq C q^\eta \|U\|_{\B_{0}}\|V\|_{\B_{1}}\exp(-\delta t)\label{eq:I(t) bound}
\end{equation}
for some absolute constants $C,\delta,\eta>0.$

We now begin the proof of (\ref{eq:I(t) bound}). We will estimate
the Laplace transform 
\begin{equation}
\hat{I}(s):=\int_{0}^{\infty}\exp(-st)I(t)dt.\label{eq:Laplace-transform-defn}
\end{equation}
This is convergent for $\Re(s)>0$ since $I$ is bounded using the
finiteness of $m_{r}$. The estimation of $\hat{I}_{s}(t)$ is closely
related to certain skew transfer operators as follows. Using notation
of \cite{AGY}, if $F:\Xi_{r}\to V$ and $s\in\mathbf{C}$, let 
\[
\hat{F}_{s}(y):=\int_{0}^{r(y)}F(y,\tau)\exp(-s\tau)d\tau.
\]
 Then following the proof of \cite[Lemma 7.17]{AGY} and adapting
to our $\rho$-skew setting we have
\begin{eqnarray}
\hat{I}(s) & = & \int_{y\in\Xi}\int_{\tau=0}^{r(y)}\int_{t+\tau\geq r(y)}e^{-st}\langle U(y,\tau),[T_{t}^{*}V](y,\tau)\rangle dtd\tau dy\nonumber \\
 & = & \sum_{k=1}^{\infty}\int_{y\in\Xi}\int_{\tau=0}^{r(y)}\int_{\tau'=0}^{r(Z^{k}y)}e^{-s(r^{(k)}(y)+\tau'-\tau)}\langle U(y,\tau),[T_{t}^{*}V](y,\tau)\rangle d\tau'd\tau dy.\label{eq:Ihat-calc}
\end{eqnarray}
The manipulation above follows from writing for each $y$, $t+\tau=r^{(k)}(y)+\tau'$
with $\tau'\in[0,r(Z^{k}x))$. For each $y$ and $t$ there is a unique
$k$ and $\tau'$ for which this is possible. Supposing more specifically
that $y\in\Xi_{\gamma_{1}.\ldots.\gamma_{k}\gamma_{0}}$with each
$\gamma_{i}$ $\gamma_{0}$-adapted, we get from (\ref{eq:master-T_t-on-Xi-functions})
that 
\begin{equation}
[T_{t}^{*}V](y,\tau)=\rho(\Theta_{\gamma_{1}}^{*})\rho(\Theta_{\gamma_{2}}^{*})\ldots\rho(\Theta_{\gamma_{k}}^{*}).V(Z^{k}(y),\tau').\label{eq:T-t^*V expr}
\end{equation}
Inserting this into (\ref{eq:Ihat-calc}) gives that (throwing out
a measure zero set)
\begin{eqnarray}
\hat{I}(s) & = & \sum_{k=1}^{\infty}\sum_{\gamma_{1},\ldots,\gamma_{k}}^{*}\int_{y\in\Xi_{\gamma_{1}.\ldots.\gamma_{k}\gamma_{0}}}\int_{\tau=0}^{r(y)}\int_{\tau'=0}^{r(Z^{k}y)}e^{-s(r^{(k)}(y)+\tau'-\tau)}\langle U(y,\tau),\eqref{eq:T-t^*V expr}\rangle d\tau'd\tau dy\nonumber \\
 & = & \sum_{k=1}^{\infty}\sum_{\gamma_{1},\ldots,\gamma_{k}}^{*}\int_{y\in\Xi_{\gamma_{1}.\ldots.\gamma_{k}\gamma_{0}}}e^{-sr^{(k)}(y)}\langle\hat{U}_{-s}(y),\rho(\Theta_{\gamma_{1}}^{*})\rho(\Theta_{\gamma_{2}}^{*})\ldots\rho(\Theta_{\gamma_{k}}^{*})\hat{V}_{s}(Z^{k}(y))\rangle dy.\nonumber\\
 \label{eq:Ihat-calc2}
\end{eqnarray}
Here, we write a $\sum^{*}$ to indicate that the $\gamma_{i}$ being
summed over are all $\gamma_{0}$-adapted. The expression (\ref{eq:Ihat-calc2})
is best understood by the \emph{skew transfer operator }that we now
introduce. Recall that $y\in\Xi$ can be written $y=(\pi,\lambda)$.
The inverse branches of $Z$ are indexed by $\gamma_{0}$-adapted
$\gamma$ and are given explicitly by 
%\marginpar{$\alpha_{\gamma}$}
\begin{equation}
\alpha_{\gamma}:(\pi,\lambda)\mapsto\left(\pi,\frac{\Theta_{\gamma}^{*}\lambda}{\|\Theta_{\gamma}^{*}\lambda\|}\right),\quad\Xi\to\Xi_{\gamma\gamma_{0}}.\label{eq:alpha-defn}
\end{equation}
The skew transfer operator $\L_{s,\rho}$ is defined for arbitrary
unitary $(\rho,V)$ and $f:\Xi\to V$ by
\[
\L_{s,\rho}[f](y):=\sum_{\gamma}^{*}e^{-sr\circ\alpha_{\gamma}(y)}J\circ\alpha_{\gamma}(y)\rho(\Theta_{\gamma}^{*})^{-1}.f\circ\alpha_{\gamma}(y).
\]
Recall that $J$ is the inverse of the Jacobian of $Z$ w.r.t. Lebesgue
measure. By results of \cite{AGY} the summation involved in $\L_{s,\rho}$
is convergent
(cf. Theorem \ref{thm:RPF} and the discussion afterwards). With the
operator $\L_{s,\rho}$ in hand, by making a change of variables of
the form $y\mapsto Z^{k}(y)$ one obtains from (\ref{eq:Ihat-calc2})

\begin{equation}
\hat{I}(s)=\sum_{k=1}^{\infty}\int_{y\in\Xi}\langle\L_{s,\rho}^{k}[\hat{U}_{-s}](y),\hat{V}_{s}(y)\rangle dy, \quad \Re(s)>0 \label{eq:Ihat-spectral}
\end{equation}
It is clear from inspection of the above that spectral bounds for
the operator $\L_{s,\rho}$ will be helpful in estimating $\hat{I}$. More precisely, we will aim to analytically continue $\hat{I}(s)$ to a strip $\Re(s) > -\sigma'$ with $\sigma' >0$.

\subsection{Spectral bounds for transfer operators}\label{ssec:spectral}
It will be useful at times to 
compare $\L_{s,\rho}$ to the operator on scalar functions on $\Xi$ given by

\[
\L_{s}[f](y):=\sum_{\gamma}^{*}e^{-sr\circ\alpha_{\gamma}(y)}J\circ\alpha_{\gamma}(y)f(\alpha_{\gamma}(y))
\]
that features in \cite[formula (7.13)]{AGY}. Recall that $\sigma_{0}$
is such that $\int\exp(\sigma_{0}r)dm<\infty$ given by Theorem \ref{thm:The-roof-function-has-exponential-tails}.
The following is given in \cite[pg. 188]{AGY}.
\begin{thm}
\label{thm:RPF}There
is some $0<\sigma_{1}<\sigma_{0}$ such that for $s$ with $|\Re(s)|<\sigma_{1}$,
$\L_{\sigma}$ is a bounded operator on $C^{1}(\Xi)$. Moreover we
have the following properties after suitable choice of $\sigma_{1}$:
\begin{enumerate}
\item $\L_{0}$ has a simple eigenvalue at $1$ and the rest of the spectrum
of $\L_{0}$ is contained in a ball around $0$ of radius $<1$.
\item For real $\sigma$ with $|\sigma|<\sigma_{1}$ the largest eigenvalue
$\lambda_{\sigma}$ of $\L_{\sigma}$ is simple and varies real analytically in
$\sigma$. In particular for all $\eta>0$ there is $\sigma_{2}(\eta)>0$
such that for real $\sigma$ with $|\sigma|\leq\sigma_{2}$ we have
$e^{-\eta} < \lambda_{\sigma}\leq e^{\eta}.$
\item The corresponding  eigenfunctions $h_{\sigma}$ (normalized so $\int h_{\sigma}=1$)
are positive and also vary real analytically as $C^{1}(\Xi)$-valued
functions on $(-\sigma_{1},\sigma_{1}).$ The functions $h_{\sigma}$
are uniformly bounded below when $|\sigma|\leq\sigma_{1}$.
\end{enumerate}
\end{thm}

As a corollary to Theorem \ref{thm:RPF} we may note that for real
$\sigma$ with $|\sigma|<\sigma_{1}$, the infinite sum
\[
\sum_{\gamma}^* e^{-\sigma r\circ\alpha_{\gamma}(y)}J\circ\alpha_{\gamma}(y)=\L_{\sigma}[1](y)
\]
 converges to a $C^{1}$ function of $y\in\Xi$. Moreover (see \cite[Paragraph following Prop. 7.8]{AGY}) since for $\sigma < \sigma_1$, $\L_\sigma$ is a continuous perturbation of $\L_0$, by possibly decreasing $\sigma_1$, we can ensure the sum above is uniformly bounded for all $y\in \Sigma$ and all $\sigma\in (-\sigma_1, \sigma_1)$. This will be useful later.

We now give spectral estimates for $\L_{s,\rho}$ in two regimes: for large imaginary part of $s$ (corresponding to high frequency aspects of the dynamics) and small (bounded) imaginary part of $s$ (corresponding to low frequencies).

\textbf{a. $|\Im(s)|\gg1$}. Here we give spectral bounds for transfer
operators $\L_{s,\rho}$, where $\rho$ is an arbitrary unitary representation,
that come from the method of Dolgopyat \cite{DOLG}. In the case of
scalar valued functions on $\Xi$ these bounds were obtained by Avila,
Gouëzel and Yoccoz in \cite{AGY} by adapting Dolgopyat's argument
to the Teichmüller setting. 

To state the next result we introduce the warped norm on $C^{1}(\Xi;\rho)$
by
\[
\|u\|_{1,t}=\sup_{y\in\Xi}\|u(y)\|+\frac{1}{\max(1,|t|)}\sup_{y\in\Xi}\|Du(y)\|.
\]
\begin{prop}
\label{prop:contraction-transfer-warped} There is $\sigma'_{0}\leq\sigma_{0}$,
$T_{0}>0,C>0$ and $\beta<1$ such that for all $s=\sigma+it$ with
$|\sigma|\leq\sigma_{0}'$ and $|t|\geq T_{0}$, for any unitary $(\rho,V)$,
$u\in C^{1}(\Xi;\rho)$ and for all $k\in\mathbf{N}$
\[
\|\L_{s,\rho}^{k}u\|_{L^{2}(\Xi)}\leq C\beta^{k}\|u\|_{1,t}.
\]
\end{prop}
The version of Proposition \ref{prop:contraction-transfer-warped}
with no twist by $\rho$ can be found in \cite[Proposition 7.7]{AGY}.
We prove Proposition \ref{prop:contraction-transfer-warped} in  Section \ref{sec:dolg}.

\noindent \textbf{b. $|\Im(s)|\ll1.$ }Here we give spectral bounds
for $\L_{s,\rho_{q}}$ that are good when $|\Im(s)|$ is below a fixed
constant. 
\begin{prop}
\label{prop:Contraction-congruence-transfer} Let $s=\sigma+it$.
For all $t_{0}>0$
there are constants $C,\eta,\epsilon>0$, $Q_0\in\Z_+$, and $0<\sigma'_{1}<\sigma_{0}$
such that when $|\sigma|<\sigma'_{1}$ and $|t|<t_{0}$ then for all
$u\in C^{1}(\Xi;\rho_q)$, all $k\in\mathbf{N}$, all $q$  coprime to $Q_0$,
\[
\|\L_{s,\rho_{q}}^{k}u\|_{C^{1}}\leq C(1-\epsilon)^{k}q^{\eta}\|u\|_{C^{1}}.
\]
\end{prop}
Proposition \ref{prop:Contraction-congruence-transfer} is proved in Section \ref{sec:Expansion-and-the-twisted-transfer-operator}. Propositions \ref{prop:contraction-transfer-warped} and \ref{prop:Contraction-congruence-transfer}
together with the expression (\ref{eq:Ihat-spectral}) imply Theorem
\ref{thm:Decay of correlations roof over Markov} by the arguments  that we give now.

\begin{proof}[Proof of  Theorem
\ref{thm:Decay of correlations roof over Markov}]
In the first part of the argument we follow \cite[Lemma 7.18]{AGY}, and extend the argument to vector valued functions.
Recall we aim to prove \eqref{eq:I(t) bound}. We write $s= \sigma + it$. Suppose $|\sigma| \leq \sigma_1 / 4$ for $\sigma_1 >0$ as in Theorem  \ref{thm:RPF}.
By integration by parts in the flow direction, we have for some $c>0$ and all $x\in \Xi$
\begin{equation}\label{eq:UV}
\| \hat{U}_{-s} (x) \|_V \leq \frac{c e^{\sigma_1 r(x)/2} }{\max(1,|t|)}\| U \|_{\B_0} , \quad |\hat{V}_s (x) | \leq \frac{c e^{\sigma_1 r(x)/2} }{\max(1,|t|)}\| V \|_{\B_1}.
\end{equation}

We can estimate
$$
\| \L_{s,\rho} \hat{U}_{-s} (x) \|_V \leq 
\sum_{\gamma}^{*} | e^{-sr\circ\alpha_{\gamma}(x)}J\circ\alpha_{\gamma}(x) |  \| \hat{U}_{-s} (\alpha_\gamma x) \|_V
$$
We were able to remove the factors here coming from $\rho$ since the representation is unitary. By the estimate for $\| \hat{U}_{-s} (x) \|_V $ in \eqref{eq:UV}, this is 
$$
\leq \frac{ c \| U \|_{B_0}}{\max(1,|t|)} \sum_{\gamma}^{*}  e^{ (\sigma_1 / 4) r\circ\alpha_{\gamma}(x)}J\circ\alpha_{\gamma}(x)  e^{(\sigma_1/2)  r\circ\alpha_{\gamma}(x)} .
$$
The sum is bounded by a constant $c' > 0$ given $| \sigma | \leq \sigma_1 /4$ and Theorem \ref{thm:RPF}. Hence by increasing $c$ if necessary
\begin{equation*}
\| \L_{s,\rho} \hat{U}_{-s} (x) \|_V \leq  \frac{ c \| U \|_{B_0}}{\max(1,|t|)}
\end{equation*}
for all $x\in \Xi$.
We have (recalling footnote \ref{foot})
$$
\| D( \L_{s,\rho} \hat{U}_{-s}) (x) \| = \| \sum_{\gamma}^{*} D[ e^{-sr\circ\alpha_{\gamma}} J\circ\alpha_{\gamma} \rho(\Theta_{\gamma}^{*})^{-1}.\hat{U}_{-s}\circ\alpha_{\gamma} ] (x)\| .
$$
We have to differentiate $e^{-sr\circ\alpha_{\gamma}}$, $J\circ\alpha_{\gamma}$, the limits of the integral defining $\hat{U}_{-s}$, or $\rho(\Theta_{\gamma}^{*})^{-1}.\hat{U}_{-s}\circ\alpha_{\gamma} $. The latter is the only deviation from \cite[Lemma 7.18]{AGY}. Since $\rho$ is locally constant, we have 
$$
\| D[ \rho(\Theta_{\gamma}^{*})^{-1}.\hat{U}_{-s}\circ\alpha_{\gamma} ](x) \| = \| \rho(\Theta_{\gamma}^{*})^{-1} D[ \hat{U}_{-s}\circ\alpha_{\gamma} ](x) \| = \|  D[ \hat{U}_{-s}\circ\alpha_{\gamma} ](x) \| .
$$
Therefore, since this is the same estimate \cite{AGY} obtain for the analogous term, the same arguments as in \cite[Lemma 7.18]{AGY} imply $\| D( \L_{s,\rho} \hat{U}_{-s}) (x) \| \leq c$.

Hence putting the previous estimates together, we have $\L_{s,\rho} \hat{U}_{-s} \in C^1(\Xi ; \rho)$ with
\begin{equation}\label{eq:Uhat}
\| \L_{s,\rho} \hat{U}_{-s} \|_{C^1(\Xi; \rho)} \leq c \| U \|_{\B_0}, \quad \| \L_{s,\rho} \hat{U}_{-s} \|_{1,t} \leq \frac{c\| U \|_{\B_0}}{\max(1,|t|)}.
\end{equation}
As a clarifying remark, we would have liked to obtain these bounds for $\hat{U}_{-s}$, but it was not possible, so we used $\L_{s,\rho} \hat{U}_{-s}$ instead. We also have from  the bound for $\| \hat{V}_s \|_\infty$ from \eqref{eq:UV},
\begin{equation}\label{eq:VL2}
\| \hat{V}_s \|^2_{L^2(\Xi) } \leq \int_{x\in \Xi} \| \hat{V}_s(x) \|^2 dm (x) \leq \frac{c^2\|V\|^2_{\B_1}}{\max(1 ,|t|)^2} \int_{x\in \Xi} e^{\sigma_1 r(x) } dm(x) \leq  \frac{c' \|V\|^2_{\B_1}}{\max(1 ,|t|)^2}.
\end{equation}
 These are all the functional norm bounds we need  for the proof. We now proceed to use the spectral bounds for the transfer operator.

Let $\sigma'' = \min( \sigma_1 / 4 , \sigma'_0 , \sigma'_1)$ where $\sigma'_0$ and $\sigma'_1$ are the constants from Propositions  \ref{prop:contraction-transfer-warped} and \ref{prop:Contraction-congruence-transfer}. We now specialize $\rho$ to $\rho=\rho_q$.
 Writing $\L_{s,\rho}^k = \L_{s,\rho}^{k-1}\L_{s,\rho} $ we obtain from 
Proposition \ref{prop:contraction-transfer-warped} and \eqref{eq:Uhat} that for $|\sigma| \leq \sigma''$, if $|t| \geq T_0$,
\[
\|\L_{s,\rho}^{k} \hat{U}_{-s} \|_{L^{2}(\Xi)} \leq C\beta^{k-1}\|\L_{s,\rho} \hat{U}_{-s}\|_{1,t} \leq \frac {C' \beta^k }{\max(1 ,|t| )}.
\]
Using \eqref{eq:VL2} and Cauchy-Schwarz,  we can bound the terms defining $\hat{I}$ in \eqref{eq:Ihat-spectral} by
$$
| \langle\L_{s,\rho}^{k}[\hat{U}_{-s}](y),\hat{V}_{s}(y)\rangle | \leq \frac{ c \beta^k \| U\|_{\B_0} \| V\|_{\B_1} }{t^2}
 $$
  for $|t|\geq T_0$ and some $c>0$. Hence for $|t|\geq T_0$ we have
  $$
  \hat{I}(s) \leq \sum_{k=0}^\infty  \frac{ c \beta^k \| U\|_{\B_0} \| V\|_{\B_1} }{t^2} \leq \frac{c'  \| U\|_{\B_0} \| V\|_{\B_1}}{t^2}.
  $$
% rho should be rho_q here. L should be script L
For $|t| <T_0$ we apply Proposition \ref{prop:Contraction-congruence-transfer} with $t_0 =T_0$ and $u = \L_{s,\rho_q}\hat{U}_{-s}$ to obtain
$$
\|\L_{s,\rho_{q}}^{k}\hat{U}_{-s} \|_{C^{1}}\leq C(1-\epsilon)^{k-1}q^{\eta}\|\L_{s,\rho_{q}}\hat{U}_{-s} u\|_{C^{1}}\leq C'(1-\epsilon)^kq^{\eta}\| U \|_{\B_0},
$$
where the last inequality used \eqref{eq:Uhat}.
Hence for $|t| < T_0$, using Cauchy-Schwarz again to bound $| \langle\L_{s,\rho}^{k}[\hat{U}_{-s}](y),\hat{V}_{s}(y)\rangle |$,
  $$
  \hat{I}(s) \leq \sum_{k=0}^\infty  C'(1-\epsilon)^kq^{\eta}\| U \|_{\B_0}\| V \|_{\B_1} \leq c'  q^\eta\| U\|_{\B_0} \| V\|_{\B_1}.
  $$
  These estimates prove that the expression defining $\hat{I}(s)$ is absolutely uniformly convergent on compact sets in  $|\sigma|\leq \sigma''$. Since each of the terms are analytic, this establishes analytic continuation of $\hat{I}(s)$ to $\Re(s) > -\sigma''$. Since we have established the estimate
  $$
  \hat{I}(s) \leq \frac{ c' q^\eta}{1 + |t|^2} \|U\|_{\B_0} \| V \|_{\B_1} ,\quad |\sigma| \leq \sigma'',
  $$
  by inverting the Laplace transform, using a contour integral over the vertical line $\Re(s) = -\sigma''/2$ as in \cite[Proposition 5.5]{OW}, we obtain for some $\delta> 0$
  $$
  I(t) \leq c'' q^\eta e^{-\delta t} \| U \|_{\B_0} \| V \|_{\B_1}.
  $$
  This completes the proof of Theorem \ref{thm:Decay of correlations roof over Markov}.
\end{proof}

Now, the only outstanding
proofs required for Theorem \ref{thm:decay-of-correlations-precise.} are those of Proposition \ref{prop:contraction-transfer-warped} and
 Proposition \ref{prop:Contraction-congruence-transfer}. These are given in
 Sections \ref{sec:dolg} and \ref{sec:Expansion-and-the-twisted-transfer-operator} respectively.

\section{The Dolgopyat argument for twisted transfer operators}\label{sec:dolg}
In this section we explain the necessary modifications to \cite[Section 7]{AGY}
in order to prove Proposition \ref{prop:contraction-transfer-warped}.

The key idea of the proof, due to Dolgopyat, is to systematically exploit oscillations of the roof function $r$.
As illustrated in Oh and Winter \cite{OW} and Magee, Oh and Winter \cite{MOW}, 
 Dolgopyat's
argument works for skew transfer operators, provided
the twisting unitary cocycle is constant on cylinders of length 1.
 The reason is that because the cocycle is locally constant, it should not
interfere with the oscillations of $r$ during the argument, which is what is being exploited.
In the current setting,  the values of the cocycle $\rho(\Theta_{\gamma})$
only depend on the cylinder $\Xi_{\gamma.\gamma_{0}}$, so the same arguments should in principle apply. We establish this rigorously below.

It will be useful to make the following normalization of the transfer operator as in \cite{AGY}. Recall from Theorem \ref{thm:RPF}  that for $\sigma$ real with $|\sigma|\leq\sigma_1$, $\lambda_\sigma$ is the leading eigenvalue of $\L_\sigma$ and $h_\sigma$ the corresponding positive eigenfunction. The $h_\sigma$ are uniformly bounded below.
We write $s =  \sigma+it$ throughout this section, assume $|\sigma| \leq \sigma_1$, and define
$$
\LL_{s,\rho}[f] := \lambda_\sigma^{-1} h_\sigma^{-1} \L_{s,\rho}[h_\sigma f], 
$$ 
for $f\in C^1(\Xi;\rho)$ and similarly $\LL_{s}[f] := \lambda_\sigma^{-1} h_\sigma^{-1} \L_{s}[h_\sigma f]$ for $f\in C^1(\Xi;\mathbf{C})$.
The purpose of this normalization is that for real $\sigma$, $\LL_\sigma[1]=[1]$, i.e., $\LL_{\sigma}$ is a Markoff operator. The operator $\LL_{s,\rho}$ acts on $C^1(\Xi;\rho)$ by 
\[
\LL_{s,\rho}[f](y):=\sum_{\gamma}^{*}e^{R_s(\alpha_\gamma y)} \rho(\Theta_{\gamma}^{*})^{-1} f(\alpha_{\gamma}(y))
\]
where 
\begin{equation}
R_s := -sr + \log J - \log( h_\sigma \circ Z ) + \log h_\sigma      -\log \lambda_\sigma . \label{eq:RS}
\end{equation}

 Many extensions of the results in \cite[Section 7]{AGY}
rely on two simple observations. Let $(\rho,V)$ be a unitary representation.
When we consider $\LL_{s,\rho}^{n}[f](x)$ we obtain sums of terms
containing factors $\rho(\Theta_{\gamma_{n}}^{*})^{-1}\ldots\rho(\Theta_{\gamma_{1}}^{*})^{-1}f(\alpha_{\gamma_{1}\ldots\gamma_{n}}x)$ (cf. \eqref{eq:transit} below).
The two observations that we will use several times are
\begin{enumerate}
\item \label{enu:p1}Since $(\rho,V)$ is assumed to be unitary, any time
we apply the triangle inequality to $\LL_{s,\rho}^{n}f$, or expressions
derived from $\LL_{s,\rho}^{n}f$ (for example, by taking a derivative),
we can use $\|\rho(\Theta_{\gamma_{n}}^{*})^{-1}\ldots\rho(\Theta_{\gamma_{1}}^{*})^{-1}f(y)\|\leq\|f(y)\|$.
\item \label{enu:p2}If we take a derivative of $\LL_{s,\rho}^{n}f$, since
the factors $\rho(\Theta_{\gamma_{n}}^{*})^{-1}\ldots\rho(\Theta_{\gamma_{1}}^{*})^{-1}$
are constant for each $\gamma_{1},\ldots,\gamma_{n}$, we may commute
the operator $D$ with $\rho(\Theta_{\gamma_{n}}^{*})^{-1}\ldots\rho(\Theta_{\gamma_{1}}^{*})^{-1}$.
Then the previous point may be used.
\end{enumerate}
The following is the extension of \cite[Lemma 7.8]{AGY} to vector
valued functions. Let $\sigma_{1}$ be as in Theorem \ref{thm:RPF}. Recall $\Lambda>1$ is the constant from Proposition \ref{prop:The-map-is uniformly epxanding markov}.
\begin{lem}
\label{lem:7.8}There is $\K>0$ such that for all $n\geq1$,
for all $s=\sigma+it$ with $\sigma\in[-\sigma_{1},\sigma_{1}]$,
$t\in\R$, for all unitary representations $(\rho,V)$, for all $u\in C^{1}(\Xi;\rho)$,
for all $x\in\Xi$,
\[
\|D[\LL_{s,\rho}^{n}u](x)\|\leq\K(1+|t|)\LL_{\sigma}^{n}[\|u\|](x)+\Lambda^{-n}\LL_{\sigma}^{n}[\|Du\|](x).
\]
\end{lem}
\begin{proof}
The proof is the same as that of \cite[Lemma 7.8]{AGY} with the addition
of points \ref{enu:p1} and \ref{enu:p2} above.
\end{proof}
We now fix $\K>5$ satisfying Lemma \ref{lem:7.8}. The next lemma
is the extension of \cite[Lemma 7.9]{AGY} that shows the iterates
of $\LL_{s,\rho}$ are bounded in the operator norm of $\|\bullet\|_{1,t}$.
\begin{lem}
\label{lem:7.9*}There is $C>1$ such that for all $s=\sigma+it$
with $\sigma\in[-\sigma_{1},\sigma_{1}]$ and $|t|\geq10$, for all
$k\in\N$, for all unitary $(\rho,V)$, for all $u\in C^{1}(\Xi;\rho)$
\[
\|\LL_{s,\rho}^{k}u\|_{1,t}\leq C\|u\|_{C^{0}}+\frac{\Lambda^{-k}}{|t|}\|Du\|_{C^{0}}.
\]
Therefore $\|\LL_{s,\rho}^{k}u\|_{1,t}\leq C\|u\|_{1,t}$.
\end{lem}
\begin{proof}
Again, the proof is a straightforward extension of \cite[Lemma 7.9]{AGY}
incorporating point \ref{enu:p1} and Lemma \ref{lem:7.8} in place
of \cite[Lemma 7.8]{AGY}.
\end{proof}
Next we note the extension of \cite[Lemma 7.10]{AGY} to vector valued
functions.
\begin{lem}
\label{lem:7.10*}There is $N_{0}\in\N$ such that for any $n\geq N_{0}$
the following hold. Let $s=\sigma+it$ with $\sigma\in[-\sigma_{1},\sigma_{1}]$
and $|t|\geq10$. Let $(\rho,V)$ be a unitary representation. Let
$v\in C^{1}(\Xi;\rho)$ satisfy 
\[
\sup\|Dv\|\geq2\K|t|\sup\|v\|.
\]
Then 
\[
\|\LL_{s,\rho}^{n}v\|_{1,t}\leq\frac{9}{10}\|v\|_{\text{1,t}}.
\]
\end{lem}
\begin{proof}
Follow the proof of \cite[Lemma 7.10]{AGY} and use point \ref{enu:p1}
and the replacement of \cite[Lemma 7.8]{AGY} by Lemma \ref{lem:7.8}.
\end{proof}
This tells us that to establish contraction of $\LL_{s,\rho}^{n}$
it remains to deal with functions with $\sup\|Dv\|\leq2\K|t|\sup\|v\|$.
Now we make the following natural modification to \cite[Definition 7.11]{AGY}.
\begin{defn}
For $t\in\R$ and $(\rho,V)$ a unitary representation, we say a pair
of functions $(u,v)$ on $\Xi$ is in $\E_{t}^{V}$ if $u:\Xi\to\R_{+}$
is $C^{1}$, $v:\Xi\to V$ is $C^{1}$, $0\leq\|v\|\leq u$, and for
all $x\in\Xi$,
\[
\max(\|Du(x)\|,\|Dv(x)\|)\leq2\K|t|u(x).
\]
\end{defn}
The next lemma is the current analog of \cite[Lemma 7.12]{AGY}.
\begin{lem}
\label{lem:7.12*}There exists $N_{1}\in\N$ such that for any $n\geq N_{1}$
the following hold. Let $s=\sigma+it$ with $\sigma\in[-\sigma_{1},\sigma_{1}]$
and $|t|\geq10$. Let $(\rho,V)$ be unitary and $(u,v)\in\E_{t}^{V}.$
Let $\chi\in C^{1}(\Xi)$ with $\|D\chi\|\leq|t|$ and $\frac{3}{4}\leq\chi\leq1$.
Assume for all $x\in\Xi$, $\|\LL_{s,\rho}^{n}v(x)\|\leq\LL_{\sigma}^{n}(\chi u)(x)$.
Then 
\[
(\LL_{\sigma}^{n}(\chi u),\LL_{s,\rho}^{n}v)\in\E_{t}^{V}.
\]
\end{lem}
\begin{proof}
The proof is the same as that of \cite[Lemma 7.12]{AGY} after replacing all uses of \cite[Lemma 7.8]{AGY} by Lemma \ref{lem:7.8}.
\end{proof}
By the arguments of \cite[pg. 191]{AGY}, there exists $n\geq \max(N_0,N_{1})$,
$\alpha_{1},\alpha_{2}$ inverse branches of $Z^{n}$, and a smooth
vector field $y$ on $\Xi$ with $1\leq\|y\|\leq2$, such that for
all $x\in\Xi$, 
\begin{equation}
\|D[r^{(n)}\circ\alpha_{1}](x)y(x)-D[r^{(n)}\circ\alpha_{2}](x)y(x)\|\geq100\K\max(\|D\alpha_{1}(x)y(x)\|,\|D\alpha_{2}(x)y(x)\|).\label{eq:NLI}
\end{equation}
Note we replaced the constant $9$ from \cite{AGY} by $100$, for
technical reasons. This is permissible by the same arguments leading
to the constant $9$ in \cite{AGY} (9 was arbitrary). The estimate \eqref{eq:NLI} lies at the very core of the Dolgopyat argument.

We now fix this $n,\alpha_{1},\alpha_{2}$, and $y$ throughout the
rest of this section.

For $\gamma_{1},\ldots,\gamma_{k}$ $\gamma_{0}$-adapted let
%\marginpar{$\alpha_{\gamma_{1}\ldots\gamma_{k}}$}
$\alpha_{\gamma_{1}.\ldots.\gamma_{k}}$ denote the inverse branch
of $Z^{k}$ that maps $\Xi$ to $\Xi_{\gamma_{1}\ldots.\gamma_{k}\gamma_{0}}$.
Then recalling the previously defined $\alpha_{\gamma}$ from (\ref{eq:alpha-defn})
one has the composition law
\[
\alpha_{\gamma_{1}\ldots\gamma_{k}}=\alpha_{\gamma_{1}}\circ\alpha_{\gamma_{2}}\circ\ldots\circ\alpha_{\gamma_{k}}.
\]If we write for each of $\alpha_{1},\alpha_{2}$, $\alpha_{j}=\alpha_{\gamma_{1}^{(j)}\ldots \gamma_{n}^{(j)}}$
then let us define
\[
\Theta_{1}:=(\Theta_{\gamma_{n}^{(1)}}^{*})^{-1}\ldots(\Theta_{\gamma_{1}^{(1)}}^{*})^{-1},\quad\Theta_{2}:=(\Theta_{\gamma_{n}^{(2)}}^{*})^{-1}\ldots(\Theta_{\gamma_{1}^{(2)}}^{*})^{-1}.
\]
Then (and this is the reason for the definition), $\rho(\Theta_{i})$
is the unitary matrix appearing in the summand of $\LL_{s,\rho}^{n}$
corresponding to the inverse branch $\alpha_{i}$.

The following lemma, analogous to \cite[Lemma 7.13]{AGY}, is the
main point where a new idea is needed to extend the methods of \cite[Section 7]{AGY}
to vector valued functions.  Recall the notation $r^{(n)}$ from Section \ref{sub:Koopman} and the functions $h_\sigma$ from Theorem \ref{thm:RPF}.
\begin{lem}
\label{lem:savings}There are constants $\delta>0$ and $\zeta>0$
such that the following hold. Let $s=\sigma+it$ with $\sigma\in[-\sigma_{1},\sigma_{1}]$
and $|t|\geq10$. For any unitary $(\rho,V)$ let $(u,v)\in\E_{t}^{V}$.
For all $x_{0}\in\Xi$ such that $B(x_{0},(\zeta+\delta)/|t|)$ is
compactly included in $\Xi$, there exists a point $x_{1}$ with $d(x_{0},x_{1})\leq\frac{\zeta}{|t|}$
such that one of the following holds:
\begin{itemize}
\item Either, for all $x\in B(x_{1},\delta/|t|)$
\begin{align*}
\left\Vert e^{-sr^{(n)}\circ\alpha_{1}(x)}J(\alpha_{1}x)\rho(\Theta_{1})[v.h_{\sigma}](\alpha_{1}x)+e^{-sr^{(n)}\circ\alpha_{2}(x)}J(\alpha_{2}x)\rho(\Theta_{2})[v.h_{\sigma}](\alpha_{2}x)\right\Vert  & \leq\\
\frac{3}{4}e^{-\sigma r^{(n)}\circ\alpha_{1}(x)}J(\alpha_{1}x)[u.h_{\sigma}](\alpha_{1}x)+e^{-\sigma r^{(n)}\circ\alpha_{2}(x)}J(\alpha_{2}x)[u.h_{\sigma}](\alpha_{2}x),
\end{align*}
\item or, for all $x\in B(x_{1},\delta/|t|)$
\begin{align*}
\left\Vert e^{-sr^{(n)}\circ\alpha_{1}(x)}J(\alpha_{1}x)\rho(\Theta_{1})[v.h_{\sigma}](\alpha_{1}x)+e^{-sr^{(n)}\circ\alpha_{2}(x)}J(\alpha_{2}x)\rho(\Theta_{2})[v.h_{\sigma}](\alpha_{2}x)\right\Vert  & \leq\\
e^{-\sigma r^{(n)}\circ\alpha_{1}(x)}J(\alpha_{1}x)[u.h_{\sigma}](\alpha_{1}x)+\frac{3}{4}e^{-\sigma r^{(n)}\circ\alpha_{2}(x)}J(\alpha_{2}x)[u.h_{\sigma}](\alpha_{2}x).
\end{align*}
\end{itemize}
\end{lem}
\begin{proof}

We follow \cite[Proof of Lemma 7.13]{AGY} and split into two cases.

\emph{Case 1. }Assume there is $x_{1}\in B(x_{0},\zeta/|t|)$ such
that either $\|v\circ\alpha_{1}(x_{1})\|\leq u\circ\alpha_{1}(x_{1})/2$
or $\|v\circ\alpha_{2}(x_{1})\|\leq u\circ\alpha_{2}(x_{1})/2$. The
same arguments as in \cite{AGY}, incorporating point \ref{enu:p1}
above, prove that the lemma holds in this case by choosing $\delta$ sufficiently small.

The harder case is the alternative one, wherein we must extend the
arguments of Magee, Oh, and Winter \cite[Proof of Lemma 29]{MOW}
to higher dimensions.

\emph{Case 2. }Assume for all $x\in B(x_{0},\zeta/|t|)$, $\|v\circ\alpha_{1}(x_{1})\|>u\circ\alpha_{1}(x_{1})/2$
and $\|v\circ\alpha_{2}(x_{1})\|>u\circ\alpha_{2}(x_{1})/2$. This implies on $\alpha_{1}(B(x_{0},\zeta/|t|))\cup\alpha_{2}(B(x_{0},\zeta/|t|))$
the function $v$ is non-vanishing. 

As in \cite{AGY} let $\phi:[0,\zeta/(2|t|))\to\Xi$ be the solution
of the differential equation $\phi'(\tau)=y(\phi(\tau))$ with $\phi(0)=x_{0}$.
Let $x^{\tau}:=\phi(\tau).$ Define for $x\in\Xi$ 
\begin{align*}
F_{1}(x) & :=e^{-sr^{(n)}\circ\alpha_{1}(x)}J(\alpha_{1}x)\rho(\Theta_{1})[v.h_{\sigma}](\alpha_{1}x),\\
F_{2}(x) & :=e^{-sr^{(n)}\circ\alpha_{2}(x)}J(\alpha_{2}x)\rho(\Theta_{2})[v.h_{\sigma}](\alpha_{2}x).
\end{align*}
Our goal is to find cancellation between these two functions. We must
follow a slightly different approach to \cite{AGY} because we don't
have exactly the same concept of `phase'. Instead we consider
the complex valued function
\[
\Phi(x):=\frac{\langle F_{1}(x),F_{2}(x)\rangle}{\|F_{1}(x)\|\|F_{2}(x)\|}.
\]
Our strategy of proof will be to establish the following Claim:

\emph{Claim: There is a choice of $\zeta>0$ such that for any $x_{0}$
as above, there is $\tau\in[0,\zeta/(8|t|))$ such that $\Re\Phi(x^{\tau})\leq\frac{1}{8}$.}

Before proving the claim, let us see how it implies Lemma \ref{lem:savings}.
Indeed, given $\tau$ as in the Claim, we let $x_{1}=x^{\tau}$. We
have $x^{\tau}\in B(x_{0},\zeta/(4|t|))$. We need to argue as we
perturb from $x^{\tau}$ in any direction, $\Phi$ does not change
too much.

First we control the sizes of $F_{1}$ and $F_{2}$ and their rates of change. Following \cite[pg. 193]{AGY},
one obtains a constant $C>0$, independent of $\delta$, such that
for all $x\in B(x_{0},\zeta/|t|)$ we have
\begin{equation}
\|DF_i(x)\|\leq C|t|\|F_i(x)\|\label{eq:cone}
\end{equation}
from which it follows from Gronwall's inequality that for all $x,x'\in B(x_{0},\zeta/(3|t|))$,
we have
\begin{equation}
\|F_{i}(x')\|\leq e^{C|t|d(x,x')}\|F_{i}(x)\|.\label{eq:cone-integrated}
\end{equation}
Note from \eqref{eq:cone-integrated} it follows that
\begin{equation}
|D\|F_{i}\||\leq C|t|\|F_{i}\|\label{eq:deriv-cone}
\end{equation}
on the domain $B(x_{0},\zeta/(3|t|))$.

Next we have that $D\langle F_{1},F_{2}\rangle=\langle DF_{1},F_{2}\rangle+\langle F_{1},DF_{2}\rangle$
so by the Schwarz inequality and (\ref{eq:cone})
\begin{equation}
|D\langle F_{1},F_{2}\rangle |\leq\|DF_{1}\|\|F_{2}\|+\|F_{1}\|\|DF_{2}\|\leq2C|t|\|F_{1}\|\|F_{2}\|.\label{eq:deriv-inner-prod}
\end{equation}
Now,
\[
D\Phi=\frac{D\langle F_{1},F_{2}\rangle}{\|F_{1}\||F_{2}\|}-\frac{\langle F_{1},F_{2}\rangle D\|F_{1}\|}{\|F_{1}\|^{2}\|F_{2}\|}-\frac{\langle F_{1},F_{2}\rangle D\|F_{2}\|}{\|F_{2}\|\|F_{2}\|}.
\]
By using (\ref{eq:deriv-cone}), (\ref{eq:deriv-inner-prod}),
the Schwarz inequality and the triangle inequality we obtain
\[
|D\Phi |\leq2C|t|+C|t|+C|t|=4C|t|
\]
on $B(x_{0},\zeta/(3|t|))$. Therefore, if $\delta$ is small enough
with $\delta<\zeta/12$, for all $x\in B(x_{1},\delta/|t|)$ we have
$\Re\Phi(x)\leq\frac{1}{4}$.

Assume that $\|F_{1}(x_{1})\|\geq\|F_{2}(x_{1})\|$. This is without
loss of generality since the other case is symmetrical. Then by choosing
$\delta$ small enough, and using (\ref{eq:cone-integrated}), we
may further assume that for all $x\in B(x_{1},\delta/|t|)$, $\|F_{1}(x)\|\geq\|F_{2}(x)\|/2.$ This implies for all $x\in B(x_{1},\delta/|t|)$
\begin{align*}
\|F_{1}(x)+F_{2}(x)\|^{2} & =\|F_{1}(x)\|^{2}+\|F_{2}(x)\|^{2}+2\Re\langle F_{1}(x),F_{2}(x)\rangle\\
 & =\|F_{1}(x)\|^{2}+\|F_{2}(x)\|^{2}+2\Re\Phi(x)\|F_{1}\|\|F_{2}\|\\
 & \leq\|F_{1}(x)\|^{2}+\|F_{2}(x)\|^{2}+\frac{1}{2}\|F_{1}\|\|F_{2}\|\\
 & =(\|F_{1}(x)\|+\frac{3}{4}\|F_{2}(x)\|)^{2}+\|F_{2}(x)\|\left(\frac{7}{16}\|F_{2}(x)\|-\|F_{1}(x)\|\right)
\end{align*}
which is $\leq(\|F_{1}(x)\|+\frac{3}{4}\|F_{2}(x)\|)^{2}$ for all
$x\in B(x_{1},\delta/|t|)$ since $\|F_{1}\|\geq\|F_{2}\|/2$ there.
\emph{This concludes the proof of Lemma \ref{lem:savings} modulo
the proof of our Claim.}

\emph{Now we prove the claim.} Note first that we can assume we have
$|\Phi(x^{\tau})|>1/8$ for all $\tau\leq\zeta/(8|t|)$, otherwise
the claim is established. 

For $i=1,2,$ let $s_{i}(x):=v\circ\alpha_{i}(x)$. Since $v$ is
non-vanishing on $B(x_{0},\zeta/|t|)$ we can write $s_{i}=\|s_{i}\|s_{i}^{*}$
where $\|s_{i}\|$ and $s_{i}^{*}$ are continuously differentiable.
Then we have the expression
\[
\Phi(x^\tau)=e^{-itr^{(n)}\circ\alpha_{1}(x^{\tau})}\overline{e^{-itr^{(n)}\circ\alpha_{2}(x^{\tau})}}\langle\rho(\Theta_{1})s_{1}^{*}(x^{\tau}),\rho(\Theta_{2})s_{2}^{*}(x^{\tau})\rangle.
\]
Note that $|\Phi(x^{\tau})|>1/8$ for all $\tau\leq\zeta/(8|t|)$
implies $|\langle\rho(\Theta_{1})s_{1}^{*}(x^{\tau}),\rho(\Theta_{2})s_{2}^{*}(x^{\tau})\rangle|>1/8$
for $\tau$ in the same range.

We have for $i=1,2$ and any vector field $J$
\[
D_{J}s_{i}=\left(D_{J}\|s_{i}\|\right)s_{i}^{*}+\|s_{i}\|\left(D_{J}s_{i}^{*}\right)
\]
and from $\langle s_{i}^{*},s_{i}^{*}\rangle_{V}=1$ we obtain
$\langle D_{J}s_{i}^{*},s_{i}^{*}\rangle_{V}=0$, so we have 
\[
\left\Vert D_{J}s_{i}\right\Vert ^{2}=\left(D_{J}\|s_{i}\|\right)^{2}+\|s_{i}\|^{2}\left\Vert D_{J}s_{i}^{*}\right\Vert ^{2}.
\]
This yields
\begin{equation}
|D_{J}\|s_{i}\|(x)|,\left\Vert D_{J}s_{i}^{*}\right\Vert \leq\frac{\left\Vert D_{J}s_{i}\right\Vert }{\|s_{i}\|}\label{eq:kato}
\end{equation}
which is a version of `Kato's inequality'. The numerator on the right
hand side is estimated using $\|Dv\|\leq2\K|t|u$ giving
\begin{equation}
\left\Vert D_{J}s_{i}\right\Vert (x)\leq2\K|t|u(x)\|D\alpha_{i}(x).J(x)\|.
\end{equation}
The denominator is estimated using the assumptions of the current
case giving $\|s_{i}\|\geq(u\circ\alpha_{i})/2$. Together this gives
\begin{equation}
|D_{J}\|s_{i}\|(x)|,\left\Vert D_{J}s_{i}^{*}(x)\right\Vert \leq4\K|t|\|D\alpha_{i}(x).J(x)\|.\label{eq:sstarderiv}
\end{equation}

Consider now the function 
\[
\Upsilon(\tau):=\frac{\langle\rho(\Theta_{1})s_{1}^{*}(x^{\tau}),\rho(\Theta_{2})s_{2}^{*}(x^{\tau})\rangle}{|\langle\rho(\Theta_{1})s_{1}^{*}(x^{\tau}),\rho(\Theta_{2})s_{2}^{*}(x^{\tau})\rangle|}.
\]
Since $h_{\sigma}$ is always nonzero, we can locally write $\Phi(x^\tau)=e^{i\theta(\tau)}\|\Phi(\tau)\|$.
Similarly, for some function $\arg\Upsilon$ we can write $\Upsilon(\tau)=\exp(i\arg\Upsilon(\tau))$.
Then
\begin{equation}
\theta(\tau)=-t(r^{(n)}\circ\alpha_{1}(x^{\tau})-r^{(n)}\circ\alpha_{2}(x^{\tau})))+\arg\Upsilon(\tau).\label{eq:theta-exp}
\end{equation}
By the same arguments as led to (\ref{eq:kato}) we have
\begin{equation}
|\Upsilon'(\tau)|\leq\frac{|\frac{d}{d\tau}\langle\rho(\Theta_{1})s_{1}^{*}(x^{\tau}),\rho(\Theta_{2})s_{2}^{*}(x^{\tau})\rangle|}{|\langle\rho(\Theta_{1})s_{1}^{*}(x^{\tau}),\rho(\Theta_{2})s_{2}^{*}(x^{\tau})\rangle|}.\label{eq:ups-deriv}
\end{equation}
Using (\ref{eq:sstarderiv}) and the triangle and Schwarz inequalities
gives
\begin{align*}
|\frac{d}{d\tau}\langle\rho(\Theta_{1})s_{1}^{*}(x^{\tau}),\rho(\Theta_{2})s_{2}^{*}(x^{\tau})\rangle| & =|\langle\rho(\Theta_{1})\frac{d}{d\tau}s_{1}^{*}(x^{\tau}),\rho(\Theta_{2})s_{2}^{*}(x^{\tau})\rangle+\langle\rho(\Theta_{1})s_{1}^{*}(x^{\tau}),\rho(\Theta_{2})\frac{d}{d\tau}s_{2}^{*}(x^{\tau})\rangle|\\
 & \leq4\K|t|\left(\|D\alpha_{1}(x).y(x)\|+\|D\alpha_{2}(x).y(x)\|\right).
\end{align*}
This bounds the numerator of (\ref{eq:ups-deriv}). The denominator
is $>1/8$ by our current assumptions. Hence
\begin{equation}
|\Upsilon'(\tau)|\leq 64\K|t|\max_{i=1,2}\left(\|D\alpha_{i}(x).y(x)\|\right).\label{eq:ups-deriv-bound}
\end{equation}
The inequality (\ref{eq:ups-deriv-bound}) implies, since the values
of $\Upsilon$ have absolute value one, 
\begin{align}
|(\arg\Upsilon)'(\tau)| & \leq 64\K|t|\max_{i=1,2}\left(\|D\alpha_{i}(x).y(x)\|\right).\label{eq:arg-Ups-deriv}
\end{align}
Therefore using (\ref{eq:NLI}) and (\ref{eq:arg-Ups-deriv}) together
with (\ref{eq:theta-exp}) we obtain 
\[
|\theta'(\tau)|\geq100\K|t|\max_{i=1,2}\left(\|D\alpha_{i}(x).y(x)\|\right)-64\K|t|\max_{i=1,2}\left(\|D\alpha_{i}(x).y(x)\|\right)=36\K|t|\max_{i=1,2}\left(\|D\alpha_{i}(x).y(x)\|\right).
\]
Following \cite{AGY} there is a constant $\gamma_{0}>0$ such that
for all $x\in\Xi$ we have\\ $\max_{i=1,2}\left(\|D\alpha_{i}(x).y(x)\|\right)\geq\gamma_{0}$.
Hence $|\theta'(\tau)|\geq c|t|$ for some $c>0$. We now choose $\zeta=16\pi/c$
so that there will be $\tau\in[0,\zeta/(8|t|)]$ with $\theta=-\pi\bmod2\pi$.
Note  this choice of $\zeta$ only depends on constants defined before
this proof so could have been made a priori. 
This gives $\Phi(x^\tau)\in\R$ and $\Phi(x^\tau)\leq-1/8$. \emph{This
proves the claim.}
\end{proof}
We now fix the constants $\zeta$ and $\delta$ given by Lemma \ref{lem:savings}.
In \cite[pg. 194]{AGY} it is explained that there are constants $C_{0}$
and $\epsilon_{0}$ such that for all $\epsilon<\epsilon_{0}$, for
all $x\in\Xi$, there exists $x'\in\Xi$ such that $d(x,x')\leq C_{0}\epsilon$
and $B(x',\epsilon)$ is compactly included in $\Xi$. We now choose
$T_{0}\geq10$ such that $2(\zeta+\delta)/T_{0}<\epsilon_{0}$.

The following lemma is the replacement of \cite[Lemma 7.15]{AGY}.
\begin{lem}
\label{lem:=00005B7.15*=00005D}There exist $\beta_{0}<1$ and $0<\sigma_{2}<\sigma_{1}$
such that the following hold. Let $s=\sigma+it$ with $\sigma\in[-\sigma_{2},\sigma_{2}]$
and $|t|\geq T_{0}$. Let $(\rho,V)$ be unitary and $(u,v)\in\E_{t}^{V}$.
Then there exists $\tilde{u}:\Xi\to\R$ such that $(\tilde{u},\LL_{s,\rho}^{n}v)\in\E_{t}^{V}$
and $\int\tilde{u}^{2}dm\leq\beta_{0}\int u^{2}dm$.
\end{lem}
\begin{proof}
Following the same arguments of \cite[Lemma 7.15]{AGY} with Lemma
\ref{lem:savings} in place of \cite[Lemma 7.13]{AGY} we construct
a function $\chi$ on $\Xi$ with $3/4\leq\chi\leq1$, $\|D\chi\|\leq|t|$,
and 
\begin{equation}
\|\LL_{s,\rho}^{n}v\|\leq\LL_{\sigma}^{n}(\chi u).\label{eq:temp5}
\end{equation}
We let $\tilde{u}= \LL_{\sigma}^{n}(\chi u)$. By \eqref{eq:temp5} combined with Lemma \ref{lem:7.12*} we obtain $(\tilde{u},\LL_{s,\rho}^{n}v)\in\E_{t}^{V}$.
It remains to show for some $\beta_{0}<1$, $\int\tilde{u}^{2}dm\leq\beta_{0}\int u^{2}dm$
when $\sigma$ is sufficiently small. This can be done using the same
arguments as in \cite{AGY}, since it has nothing to do with $V$,
only the construction of $\chi$ which is basically the same as in
\cite{AGY}.
\end{proof}
Finally, to conclude this section, we note that Proposition \ref{prop:contraction-transfer-warped} follows
from Lemmas \ref{lem:7.9*}, \ref{lem:7.10*}, and \ref{lem:=00005B7.15*=00005D}
exactly in the same way that \cite[Prop. 7.7]{AGY} is proved from
the analogous \cite[Lemmas 7.9, 7.10, and 7.15]{AGY}. 

\section{Expansion and the twisted transfer operator\label{sec:Expansion-and-the-twisted-transfer-operator}}

This section contains a proof of Proposition \ref{prop:Contraction-congruence-transfer}.

\subsection{Refining the choice of $\gamma_{0}$\label{sub:Refining-the-choice-of-gamma-0}}

We now assume that $\pi$ is the member of $\RR$ specified by Theorem \ref{thm:zorich-conjecture}. Let $S$ denote a fixed  finite set of generators of $G_{\pi}$, this is possible since we know $G_\pi$ is finite index in $\Sp(\Z^{2g},\omega_\pi)$.
Choose a finite set $\Upsilon_{0}$ of $\gamma$ that are paths in
$\RR$ beginning and ending in $\pi$ and such that
\[
\{\Theta_{\gamma}^{*}\::\:\gamma\in\Upsilon_{0}\:\}
\]
together with their inverses generate $S$. Now let 
%\marginpar{$\Upsilon$}
\[
\Upsilon=\Upsilon_{0}\cup\{\gamma.\gamma\::\:\gamma\in\Upsilon_{0}\:\}.
\]
We note for later on that this
definition guarantees
\begin{lem}
\label{lem:differences-gen-G_pi'}The elements
\[
\Theta_{\gamma}^{*}.(\Theta_{\gamma'}^{*})^{-1}\in\Sp(\Z^{2g},\omega_{\pi})\quad\gamma,\gamma'\in\Upsilon
\]
generate $G_{\pi}.$
\end{lem}

\begin{proof}
For a given $\tilde{\gamma}\in\Upsilon_{0},$ we have $\Theta_{\tilde{\gamma}.\tilde{\gamma}}^{*}(\Theta_{\tilde{\gamma}}^{*})^{-1}=\Theta_{\tilde{\gamma}}^{*}$
and $\Theta_{\tilde{\gamma}}^{*}(\Theta_{\tilde{\gamma}.\tilde{\gamma}}^{*})^{-1}=(\Theta_{\tilde{\gamma}}^{*})^{-1}$.
On the other hand, the $\Theta_{\tilde{\gamma}}^{*}$ with $\tilde{\gamma}\in\Upsilon_{0}$
together with their inverses generate $S$ and hence $G'_{\pi}$.
\end{proof}

We will now choose $\gamma_{0}$ such that no $\gamma\in\Upsilon$
contains $\gamma_{0}$ as a substring and moreover $\gamma_{0}$ is
strongly positive and neat (recall these properties from Section \ref{sub:The-choice-of-g0}).
This can be done simply by ensuring that $\gamma_{0}$ is strongly
positive and neat and longer than all $\gamma\in\Upsilon.$ We now give the details of this construction.

Before stating the next lemma we introduce some language. A path in
$\RR$ is \emph{complete }if every $\alpha\in\A$ is the winner of
some arrow in $\gamma$. It follows from a result of \cite[Section 1.2.3]{MMY} (see also \cite[Lemma 3.2]{AGY})
that there exists a complete path $\gamma_{*}$ beginning and ending
at $\pi$. A path in $\RR$ is said to be $k$-complete if it is the
concatenation of $k$ complete paths. Write $\gamma_{*}^{k}$ for
the $k$-fold concatenation of $\gamma_{*}$ with itself. Then for
example, if $\gamma_{*}$ is complete then $\gamma_{*}^{k}$ is $k$-complete.
\begin{lem}[{\cite[Lemma 4.2]{AGY}}]
\label{lem:A-k-complete-path-strongly positive}A $k$-complete path
with $k\geq3|\A|-4$ is strongly positive.
\end{lem}

As noted in \cite[pg. 162, footnote]{AGY}, a path is neat if it ends
with a type $\epsilon$ arrow and begins with a string of opposite
type arrows at least half the length of the path. Suppose that $\gamma_{*}$
ends with a bottom arrow. Choose then $k$ such that 
\begin{equation}
l(\gamma_{*}).k\geq\max_{\gamma\in\Upsilon}l(\gamma),\quad k\geq3|\A|-4.\label{eq:k-choice}
\end{equation}
Next choose $\gamma'$ beginning and ending at $\pi$ with $l(\gamma_{*}).k+|\RR|$
top arrows at its beginning and $\leq|\RR|$ arrows afterwards (this
is always possible since whatever the endpoint of the first top arrows,
one can quickly return to $\pi$). Then 
\[
\gamma_{0}:=(\gamma'\gamma_{*})\gamma_{*}^{k-1}
\]
 begins with 
\[
l(\gamma_{*}).k+|\RR|=\frac{1}{2}\left(l(\gamma_{*}).k+|\RR|+|\RR|+l(\gamma_{*}).k\right)\geq\frac{1}{2}l(\gamma_{0})
\]
top arrows so is therefore neat. Also, clearly $\gamma'\gamma_{*}$
is complete so $\gamma_{0}$ is $k$-complete. Therefore $\gamma_{0}$
is strongly positive by Lemma \ref{lem:A-k-complete-path-strongly positive}.
Finally, by choice of $k$ in (\ref{eq:k-choice}) $\gamma_{0}$ is
longer than any element of $\Upsilon$. We have shown
\begin{lem}
It is possible to choose $\gamma_{0}$ so that no element $\gamma\in\Upsilon$
contains $\gamma_{0}$ as a substring and moreover $\gamma_{0}$ is
strongly positive and neat.
\end{lem}

We fix such a $\gamma_{0}$ for the remainder of the paper (and retroactively
for the previous sections). From the discussion in Section \ref{sub:The-choice-of-g0}
this has the consequence that the elements of the set 
\[
\gamma_{0}.\Upsilon:=\{\gamma_{0}\gamma\::\:\gamma\in\Upsilon\:\}
\]
are all $\gamma_{0}$-adapted. We will use this later.

\subsection{Decoupling I: Releasing the convolution}

We now perform the decoupling argument of \cite{MOW} with the first part of the argument based on \cite{BGS2} and the latter
part of the argument coming from \cite[Appendix]{MOW}. One
key difference here is the fact that the symbolic dynamics takes place
on an infinite alphabet.

We understand during this section that all $\gamma_{i}$
are $\gamma_{0}$-adapted in all sums and so forth.  It is possible
to show by adapting the proof of \cite[Lemma 7.8]{AGY} that $\L_{s,\rho}$ and $\LL_{s,\rho}$
act on $C^{1}(\Xi;\rho)$ for $|\sigma|  <\sigma_{1}$ .

Recall the definition of the function $R_s$ from \eqref{eq:RS}. It will be convenient to introduce the function $R_{s}^{(n)}(y):=\sum_{i=0}^{n-1} R_s ( Z^i y )$ for $y$ in the domain of $Z^{n-1}$. Also recall the notation $\alpha_{\gamma_{1}.\ldots.\gamma_{k}}$  from Section \ref{sec:dolg}. With these notations, for $F\in C^{1}(\Xi;\rho)$, 
\begin{equation}\label{eq:transit}
\LL_{s,\rho}^{N}[F](y)=\sum_{\gamma_{1},\ldots,\gamma_{N}}e^{R_s^{(N)}(\alpha_{\gamma_{1}\ldots\gamma_{N}}y)}\rho(\Theta_{\gamma_{N}}^{*})^{-1}.\ldots.\rho(\Theta_{\gamma_{1}}^{*})^{-1}F(\alpha_{\gamma_{1}\ldots\gamma_{N}}y).
\end{equation}
We prepare a preliminary lemma.
\begin{lem}\label{lem:RSN}
There is a constant $C>0$ such that for all $M\geq 1$, all $\sigma$ with $|\sigma |\leq \sigma_1$ and $s= \sigma+it$, and all $y\in \Xi$,
\begin{equation}
\|D[R_{s}^{(M)}\circ\alpha_{\gamma_{1}\ldots\gamma_{M}}](y) \|\leq C(1 + |t|) \label{eq:r^M-deriv-bound}.
\end{equation}
\end{lem}
\begin{proof}
Using  Proposition
\ref{prop:The-map-is uniformly epxanding markov}, Lemma \ref{lem:The-roof-function-is-good}, and Theorem \ref{thm:RPF} we obtain
$$
\|D(R_{s}\circ\alpha_{\gamma})(y)\| \leq c (1 + |t|) 
$$
for all $\gamma$ and $y$.
Furthermore, for $k\geq 1$, by the chain rule and Proposition \ref{prop:The-map-is uniformly epxanding markov} Part  \ref{enu:expanding} we have
$$
\|D(R_{s}\circ\alpha_{\gamma_1 \ldots \gamma_k})(y)\| \leq c (1 + |t|) \Lambda^{-{k+1}}
$$
for $\Lambda>1$.
Now,
\begin{align*}
\| D ( R_{s}^{(M)}\circ\alpha_{\gamma_{1}\ldots\gamma_{M}} ) (y) \| =\| \sum_{i=0}^{M-1} D(R_{s}\circ\alpha_{\gamma_{1+i} \ldots \gamma_M})(y) \| \\
\leq \sum_{i=0}^{M-1} c (1 + |t|) \Lambda^{-{k+1}} \leq C(1 +|t|)
\end{align*}
as required, by summing the geometric series.
\end{proof}
We now perform the same initial
decoupling arguments as in  Bourgain, Gamburd and Sarnak \cite{BGS2}.  Let 
$$
N = M + \tilde{M}
$$
and let $o$ be an arbitrary point in $\Xi$. Write $d$ for the distance on $\Xi$ coming from the Hilbert metric, induced by the Finsler metric on $\Xi$. Note that $\Xi$ has bounded diameter with respect to $d$ (since $\Xi$ is a John domain in the sense of \cite{AGY} by \cite[Lemma 4.4]{AGY}).

%Using that the roof function $r$ is good (Lemma \ref{lem:The-roof-function-is-good},
%Part \ref{enu:rooffunctiongoodbounded-deriv} in particular) together
%with the bound on the deriviative of $\log J$ (Proposition \ref{prop:The-map-is uniformly epxanding markov}
%Part \ref{enu:uniformly-expanding-logJ-deriv}) and the expanding
%property in terms of constant $\Lambda>1$ (Proposition \ref{prop:The-map-is uniformly epxanding markov}
%Part \ref{enu:expanding}) it is possible to perform the same initial
%decoupling arguments as in \cite[Section 5.1]{MOW}. This yields after
%writing 

%\[
%N=M+\tilde{M}
%\]
 %and letting $o$ be an arbitrary point in $\Xi$ that
 \begin{lem}\label{lem:dc1}We have
\begin{eqnarray}
\LL_{s,\rho}^{N}[F](y) = \sum_{\gamma_{1},\ldots,\gamma_{M}}Op_{\gamma_{1}\ldots\gamma_{M};y}(\rho).\rho(\Theta_{\gamma_{M}}^{*})^{-1}.\ldots.\rho(\Theta_{\gamma_{1}}^{*})^{-1}F(\alpha_{\gamma_{1}\ldots\gamma_{M}}o)+O(\|F\|_{C^{1}}\Lambda^{-M})\nonumber 
\end{eqnarray}
where
\[
Op_{\gamma_{1}\ldots\gamma_{M};y}(\rho):=\sum_{\gamma_{M+1},\ldots,\gamma_{N}}e^{R_{s}^{(N)}(\alpha_{\gamma_{1}\ldots\gamma_{N}}y)}\rho(\Theta_{\gamma_{N}}^{*})^{-1}.\ldots.\rho(\Theta_{\gamma_{M+1}}^{*})^{-1}
\]
is a member of the algebra generated by the $\rho(\Theta_{\gamma_{i}}^{*})^{-1}$
acting on $V$. The error term is in the norm of $V$. We also have
\begin{eqnarray}
D(\LL_{s,\rho}^{N}[F])(y) & = & \sum_{\gamma_{1},\ldots,\gamma_{M}}Op_{\gamma_{1}\ldots\gamma_{M};y}^{\partial}(\rho).\rho(\Theta_{\gamma_{M}}^{*})^{-1}.\ldots.\rho(\Theta_{\gamma_{1}}^{*})^{-1}F(\alpha_{\gamma_{1}\ldots\gamma_{M}}o)\label{eq:derivL^N}\\
 & + & O((1+|t|)\|F\|_{C^{1}}\Lambda^{-M})\nonumber 
\end{eqnarray}
where
\[
Op_{\gamma_{1}\ldots\gamma_{M};y}^{\partial}(\rho):=\sum_{\gamma_{M+1},\ldots,\gamma_{N}}D[ e^{R_s^{(N)}\circ\alpha_{\gamma_{1}\ldots\gamma_{N}}}] (y)\otimes\rho(\Theta_{\gamma_{N}}^{*})^{-1}.\ldots.\rho(\Theta_{\gamma_{M+1}}^{*})^{-1}
\]
is a member of $\Hom(T_{y}\Xi,\R)\otimes\End(V)\cong\Hom(T_{y}\Xi,\End(V))$
and the big $O$ term is interpreted w.r.t. the operator norm between
the Finsler metric norm on $T_{y}\Xi$ and $\End(V)$ with its own
operator norm. Write $\|\bullet\|_{T_{y}\Xi,\End(V)}$ for this norm
and $\|\bullet\|_{\End(V)}$ for the operator norm on $\End(V)$.
\end{lem}

\begin{proof}
We begin by inspecting \eqref{eq:transit} and noting that $
F( \alpha_{\gamma_1 \ldots \gamma_N} y )$ and $ F(\alpha_{\gamma_1 \ldots \gamma_M } o)$
are distance $\ll \Lambda^{-M}$ apart, where $\Lambda>1$ is the constant from Proposition \ref{prop:The-map-is uniformly epxanding markov}
Part \ref{enu:expanding}. Hence we have 
\begin{equation}\label{fdiff}
\| F( \alpha_{\gamma_1 \ldots \gamma_N} y ) - F(\alpha_{\gamma_1 \ldots \gamma_M } o) \|_V \ll \| F \|_{C^1} \Lambda^{-M}.
\end{equation}
This gives 
\begin{align*}
\LL_{s,\rho}^{N}[F](y)=\sum_{\gamma_{1},\ldots,\gamma_{N}}e^{R_s^{(N)}(\alpha_{\gamma_{1}\ldots\gamma_{N}}y)}\rho(\Theta_{\gamma_{N}}^{*})^{-1}.\ldots.\rho(\Theta_{\gamma_{1}}^{*})^{-1}F(\alpha_{\gamma_{1}\ldots\gamma_{M}}o) \\
+ \sum_{\gamma_{1},\ldots,\gamma_{N}}e^{R_s^{(N)}(\alpha_{\gamma_{1}\ldots\gamma_{N}}y)}\rho(\Theta_{\gamma_{N}}^{*})^{-1}.\ldots.\rho(\Theta_{\gamma_{1}}^{*})^{-1} (F( \alpha_{\gamma_1 \ldots \gamma_N} y ) - F(\alpha_{\gamma_{1}\ldots\gamma_{M}}o)).
\end{align*}
Using that $\rho$ is unitary, the second line above can be bounded in $\| . \|_V$ using  \eqref{fdiff} by
$$
\| F \|_{C^1} \Lambda^{-M} \sum_{\gamma_{1},\ldots,\gamma_{N}}e^{R_{\sigma}^{(N)}(\alpha_{\gamma_{1}\ldots\gamma_{N}}y)} = \| F \|_{C^1} \Lambda^{-M} \LL^N_\sigma[1](y)= \| F \|_{C^1} \Lambda^{-M}  .
$$
This proves the first part of the lemma.
For the second part, note

\begin{align}\label{dc2}
D(\LL_{s,\rho}^{N}[F])(y)=\sum_{\gamma_{1},\ldots,\gamma_{N}}D[ e^{R_s^{(N)}\circ\alpha_{\gamma_{1}\ldots\gamma_{N}}}] (y)\rho(\Theta_{\gamma_{N}}^{*})^{-1}.\ldots.\rho(\Theta_{\gamma_{1}}^{*})^{-1}F(\alpha_{\gamma_{1}\ldots\gamma_{N}}y)  \\
+ \sum_{\gamma_{1},\ldots,\gamma_{N}}e^{R_s^{(N)}(\alpha_{\gamma_{1}\ldots\gamma_{N}}y)}\rho(\Theta_{\gamma_{N}}^{*})^{-1}.\ldots.\rho(\Theta_{\gamma_{1}}^{*})^{-1} D( F\circ \alpha_{\gamma_{1}\ldots\gamma_{N}} )(y). \nonumber
\end{align}
Since $\|  D( F\circ \alpha_{\gamma_{1}\ldots\gamma_{N}} )(y) \| \leq \Lambda^{-N} \| F \|_{C^1}$, the second term can be bounded in the norm of $V$ by $\|F\|_{C^1}\Lambda^{-N}$ by the arguments from the first part of the lemma. 
To deal with the first line of \eqref{dc2}, we argue as before, replacing $F( \alpha_{\gamma_1 \ldots \gamma_N} y )$ with $ F(\alpha_{\gamma_1 \ldots \gamma_M } o)$. By the same arguments, we incur an error that can be bounded by
\begin{equation}\label{errorsum}
\| F \|_{C^1} \Lambda^{-M} \sum_{\gamma_{1},\ldots,\gamma_{N}}\|D[ e^{R_s^{(N)}\circ\alpha_{\gamma_{1}\ldots\gamma_{N}}}] (y)\|.
\end{equation}
We must estimate the sum here. We calculate
$$
D( e^{R_s^{(N)}\circ\alpha_{\gamma_{1}\ldots\gamma_{N}}}) (y) = D( R_s^{(N)}\circ\alpha_{\gamma_{1}\ldots\gamma_{N}}) (y)e^{R_s^{(N)}\circ\alpha_{\gamma_{1}\ldots\gamma_{N}}} (y)
$$ so by Lemma \ref{lem:RSN} we have
$$
\|D( e^{R_s^{(N)}\circ\alpha_{\gamma_{1}\ldots\gamma_{N}}}) (y)  \| \leq C(1 + |t|) e^{R_\sigma^{(N)}\circ\alpha_{\gamma_{1}\ldots\gamma_{N}}} (y).
$$
Therefore the sum in \eqref{errorsum} is $\leq C(1+|t|) \sum_{\gamma_1 ,\ldots, \gamma_N}e^{R_\sigma^{(N)}\circ\alpha_{\gamma_{1}\ldots\gamma_{N}}} (y) = C(1+|t|) $. This concludes the proof.
\end{proof}

We will now aim to give operator norm bounds for $Op_{\gamma_{1}\ldots\gamma_{M};y}(\rho_{q}^{\new})$ and $Op_{\gamma_{1}\ldots\gamma_{M};y}^{\partial}(\rho_{q}^{\new})$ that involve power decay in $q$.

\begin{prop}
\label{prop:measure-convolution-bound}
Let $s=\sigma+it$. There is  $D>0$
such that for all $t_{0}>0,$ there are $\sigma_{1},c,C,q_0>0$ such that
for $|\sigma|<\sigma_{1}$, $|t|\leq t_{0}$, $q$ odd with $q>q_0$ and $\tilde{M}\approx c\log q$,
we have 
\[
\|Op_{\gamma_{1}\ldots\gamma_{M};y}(\rho_{q}^{\new})\|_{\End(V)}\leq Ce^{R_{\sigma}^{(M)}(\alpha_{\gamma_{1}\ldots\gamma_{M}}o)}q^{-D},
\]
\[
\|Op_{\gamma_{1}\ldots\gamma_{M};y}^{\partial}(\rho_{q}^{\new})\|_{T_{y}\Xi,\End(V)}\leq Ce^{R_{\sigma}^{(M)}(\alpha_{\gamma_{1}\ldots\gamma_{M}}o)}q^{-D}.
\]
\end{prop}

The bound for $Op_{\gamma_{1}\ldots\gamma_{M};y}^{\partial}(\rho_{q}^{\new})$
is similar to that for $Op_{\gamma_{1}\ldots\gamma_{M};y}(\rho_{q}^{\new})$
with no added difficulties\footnote{After applying $Op_{\gamma_{1}\ldots\gamma_{M};y}^{\partial}(\rho_{q}^{\new})$
to a test tangent vector $v$ in $T_{y}\Xi$ one obtains an element
of $\End(V)$ with the task of bounding its operator norm, which can
be done in exactly the same way as we will treat $Op_{\gamma_{1}\ldots\gamma_{M};y}(\rho_{q}^{\new})$.
On the other hand, it is worth pointing out that the bound for $Op_{\gamma_{1}\ldots\gamma_{M};y}^{\partial}(\rho_{q}^{\new})$
relies crucially on the fact that $|t|\leq t_{0}$ whereas this is
not a factor in bounding $Op_{\gamma_{1}\ldots\gamma_{M};y}(\rho_{q}^{\new})$.}, so we treat only $Op_{\gamma_{1}\ldots\gamma_{M};y}(\rho_{q}^{\new})$. The reader can consult \cite[Section 5.3]{MOW} for
more details. The proof of Proposition \ref{prop:measure-convolution-bound}
will take up the remaining Subsections \ref{sub:Bounding-the-operator-norm},
\ref{sub:Decoupling-II:-Bounding-the-real-measure} of the present
section.

\begin{proof}[Proof of Proposition \ref{prop:Contraction-congruence-transfer}  from Proposition \ref{prop:measure-convolution-bound}]
Import all the constants from Proposition \ref{prop:measure-convolution-bound}. Recall we are given $t_0$ such that we assume $s = \sigma+it$ with $|t| \leq t_0$. We also assume $|\sigma| < \sigma_1$. We choose $M  \approx c'\log q$ where $c'>0$ is chosen such that $\Lambda^{-M} \approx q^{-D}$. Then $N \approx c_0 \log q$. 
Note that
$$
\| \rho^{\new}_q(\Theta_{\gamma_{M}}^{*})^{-1}.\ldots.\rho(\Theta_{\gamma_{1}}^{*})^{-1}F(\alpha_{\gamma_{1}\ldots\gamma_{M}}o) \|_V \leq \| F \|_{C^1}
$$
as $\rho^{\new}_q$ is unitary. 

Lemma \ref{lem:dc1} and using the triangle inequality gives a constant $C>0$ such that for any $F\in C^1(\Xi;\rho_q^\new)$
$$
\| \LL^N_{s,\rho_q^\new} F \|_{C^1} \leq C q^{-D} \|F \|_{C^1} \sum_{\gamma_1 , \ldots , \gamma_M } e^{R_{\sigma}^{(M)}(\alpha_{\gamma_{1}\ldots\gamma_{M}}o)} = C q^{-D} \|F \|_{C^1} 
$$
since the sum is $\LL^M_{\sigma}[1] = 1$. Now by increasing $q_0$ if necessary, we ensure that when $q > q_0$, $Cq^{-D} \leq q^{-D/2}$.

Now given an arbitrary $N'$, we can write $N' = aN + b$ with $0\leq b<N\approx c_0 \log q$. Since the operator norm of $\LL_{s,\rho_q^\new}$ is bounded (by comparison to $\LL_\sigma$) by a constant $K$ depending on $t_0$, we obtain for any $F\in C^1(\Xi;\rho_q^\new)$
$$
\| \LL^{N'}_{s,\rho_q^\new} F \|_{C^1} \leq K^b q^{-aD/2} \|F \|_{C^1} \leq q^\eta (1- \epsilon)^{N'} \|F \|_{C^1} 
$$
for some $\e , \eta >0$. 

To deal with $\rho_q$ in place of $\rho_q^\new$, we consider the groups $\Gamma_q(q')$ that are defined to be the kernels of reduction modulo $q'$ on $\Gamma_q$. 

We decompose $\rho_q$ as $\oplus_{1\neq q' | q } \rho_{q'}^q$ where $\rho_{q'}^q$ is the subrepresentation of $\ell^2_0(\Gamma_q)$ corresponding to functions invariant under $\Gamma_q(q')$ but not invariant by any $\Gamma_q(q'')$ with $q''|q'$, $q'' \neq q'$. This gives a splitting
\begin{equation}\label{decomp}
C^1( \Xi ; \rho_q ) = \bigoplus_{ 1 \neq q' | q } C^1( \Xi ; \rho^q_{q'} ).
\end{equation}
The action of the transfer operator $\LL_{s, \rho_q}$  on $C^1( \Xi ; \rho^q_{q'} )$ is intertwined with the action of $\LL_{s, \rho^{\new}_{q'}}$ on $C^1(\Xi; \rho_{q'}^\new )$. Thus if $f\in C^1(\Xi; \rho_q)$ we can decompose $f = \sum_{1 \neq q' | q} f_{q'}$ according to \eqref{decomp}.

\emph{This is the point in the paper where the modulus $Q_0$ of Theorem \ref{main} comes into play. We now assume $Q_0$ is the product of primes $\leq q_0$, where $q_0$ is the constant fixed during this proof. In particular, if $q$ is coprime to $Q_0$, then any $q' | q$ has no proper divisors $\leq q_0$.}
Under this assumption, $f_{q'} =0$ if $q' \leq q_0$.

Now we have
$$
\| \LL^{N'}_{s, \rho_q} f \|_{C^1} \leq \sum_{ q_0 < q' | q } \| \LL^{N'}_{s, \rho_{q'}} f_{q'} \| \leq \sum_{ q_0 < q' | q } q^{\eta} (1 - \epsilon)^{N'} \| f_{q'} \|_{C^1}
$$
where we used the bound we previously obtained for the operator norm of $\LL^{N'}_{s, \rho^\new_{q'}}$ to bound $\| \LL^{N'}_{s, \rho_{q'}} f_{q'} \|_{C^1}$. Since $\| f_{q'} \|_{C^1} \leq \|f \|_{C^1}$ for each $q'$, and $q$ has fewer than $q^{\zeta}$ divisors for some $\zeta>0$ and all $q$, by increasing $\eta$ if necessary the above can be bounded by
$$
q^{\eta} (1 - \epsilon)^{N'} \| f \|_{C^1}.
$$
This proves  Proposition \ref{prop:Contraction-congruence-transfer}   with $\LL_{s,\rho_q^\new}$ in place of $\L_{s,\rho_q^\new}$. To convert between estimates for the unnormalized and normalized transfer operators, note that $\L^N_{s,\rho_q} = \lambda_{\sigma}^{N} h_\sigma \LL^N_{s,\rho_q} h_{\sigma}^{-1}$. Multiplication and division by $h_\sigma$ is a uniformly bounded operator in $|\sigma| < \sigma_1$ by Theorem \ref{thm:RPF}. Morever by Theorem \ref{thm:RPF} we can choose $\sigma'_1 < \sigma_1$ such that $\lambda_\sigma < (1- \epsilon)^{1/2}$ for all $|\sigma| < \sigma'_1$. Therefore under these assumptions on $\sigma$ we have for some $C>0$, for all $N'\geq 1$,
$$
\| \L^{N'}_{s, \rho_q} f \|_{C^1} \leq C q^\eta (1- \epsilon)^{N'/2}
$$
which concludes the proof.
\end{proof}

\subsection{Bounding the operator norm of convolution operators\label{sub:Bounding-the-operator-norm}}

Let $\pi_{q}:\Sp(\Z^{2g},\omega_{\pi})\to\Gamma_{q}$ be the reduction
mod $q$ map.  To improve the readability of the following argument we will write
for $\gamma_{0}$-adapted $\gamma$ 
\[
h_{\gamma}:=\pi_{q}(\Theta_{\gamma}^{*})^{-1}\in\Gamma_{q}.
\]
We are tasked with estimating the operator norm of the group algebra
element
\[
\mu{}_{\gamma_{1}\ldots\gamma_{M};y}:=\sum_{\gamma_{M+1},\ldots,\gamma_{N}}e^{R_{s}^{(N)}(\alpha_{\gamma_{1}\ldots\gamma_{N}}y)}h_{\gamma_{N}}h_{\gamma_{N-1}}\ldots h_{\gamma_{M+1}}\in\mathbf{C}[\Gamma_{q}]
\]
as it acts by convolution on $\ell^{2}_\new(\Gamma_{q})$. Indeed, this
is precisely the operator $Op_{\gamma_{1}\ldots\gamma_{M};y}(\rho_{q}^{\new})$
when restricted to $\ell_{\new}^{2}(\Gamma_{q}).$ We view elements
of $\mathbf{C}[\Gamma_{q}]$ interchangeably as complex valued measures
on $\Gamma_{q}$. We write $*$ for the convolution of measures, this
corresponds to multiplication in $\mathbf{C}[\Gamma_{q}]$. Given
$\mu\in\mathbf{C}[\Gamma_{q}]$ we write $|\mu|$ for the non negative
real measure obtained by taking absolute values of coefficients. We let $\tilde{\mu}$ be the measure defined by $\tilde{\mu}(g):=\overline{\mu(g^{-1})}$. If $\mu_1 , \mu_2\in \mathbf{R}[\Gamma_q]$ we write $\mu_1 \leq \mu_2$ if $\mu_1(g) \leq \mu_2 (g)$ for all $g\in \Gamma_q$.

Recall $N=M+\tilde{M}$, $s= \sigma+it$, and $o$ is an arbitrary but fixed point in $\Xi$. 

\begin{lem}
 We have
\begin{equation}
|\mu{}_{\gamma_{1}\ldots\gamma_{M};y}|\leq Ce^{R_{\sigma}^{(M)}(\alpha_{\gamma_{1}\ldots\gamma_{M}}o)}\mu_{1}\label{eq:mu-to-mu1}
\end{equation}
 where $C>0$ is a constant and
\begin{equation}
\mu_{1}=\sum_{\gamma_{M+1},\ldots,\gamma_{N}}e^{R_{\sigma}^{(\tilde{M})}(\alpha_{\gamma_{M+1}\ldots\gamma_{N}}o)}h_{\gamma_{N}}h_{\gamma_{N-1}}\ldots h_{\gamma_{M+1}}.\label{eq:mu1-definition}
\end{equation}
\end{lem}
\begin{proof}
The proof is the same as \cite[Lemma 38]{MOW}.
\end{proof}

We now organize the ingredients for the proof of Proposition \ref{prop:measure-convolution-bound}.
\begin{prop}[Majorization of $\mu_1$]\label{prop:major}
There is a constant $\e > 0$ such that for any $B >1$, there exists an integer $L>0$ such that for all $K > 0$, if $\tilde{M}=LK$, there
is a measure $\mu_2$ such that
\begin{equation}
\mu_1 \leq  \mu_2 \label{eq:mu1major},
\end{equation}
for all $\phi \in \ell^2_0(\Gamma_q)$, 
\begin{equation}\label{m2colv}
\| \mu_2 * \phi \|_{\ell^2} \leq (1 - \e)^{K}\| \mu_2 \|_1 \| \phi \|_{\ell^2},
\end{equation}
and 
\begin{equation}\label{m2m1}
\| \mu_2 \|_1 \leq B^{K} \|\mu_1\|_1. 
\end{equation}
\end{prop}
The proof of this proposition is deferred to the next section.
Note that we would like to have \eqref{m2colv} for $\mu_1$, or even better, the analogous result for $\mu{}_{\gamma_{1}\ldots\gamma_{M};y}$. However we only know $| \mu{}_{\gamma_{1}\ldots\gamma_{M};y} | \ll e^{R_{\sigma}^{(M)}(\alpha_{\gamma_{1}\ldots\gamma_{M}}o)} \mu_2$ from which it is not obvious how to convert \eqref{m2colv} into Proposition \ref{prop:major}.

The solution is to first use Proposition \ref{prop:major} to deduce that the $\ell^2$ norm of $\widetilde{\mu_2} * \mu_2$ is small, hence the $\ell^2$ norm of $\widetilde{\mu{}_{\gamma_{1}\ldots\gamma_{M};y}} *\mu{}_{\gamma_{1}\ldots\gamma_{M};y}$ is small. This will be done using the following lemma.

\begin{lem}[{\cite[Proposition 45]{MOW}}]
\label{lem:spectralgfap-iterate-to-ell2-flat}
For any measure $\nu$ on $\Gamma_q$, we have
$$
\| \tilde{\nu} * \nu \|_2 \leq \frac{ \|\nu\|_1^2 }{ |\Gamma_q |^{1/2}} + \| \nu \|_1 \| \nu \|_{\ell^2_0(\Gamma_q)}.
$$
Here $\| \nu \|_{\ell^2_0(\Gamma_q)}$ is the operator norm of $\nu$ acting by convolution on $\ell^2_0(\Gamma_q)$.
\end{lem}
\begin{proof}
This is proved in \cite[Proof of Prop. 45]{MOW}.
\end{proof}
In the previous lemma, we will take $\nu = \mu_2$. When we succeed in proving  $\| \widetilde{\mu{}_{\gamma_{1}\ldots\gamma_{M};y}} *\mu{}_{\gamma_{1}\ldots\gamma_{M};y}\|_2$ is small, we will then need a way to convert this into information on the operator norm of $ \mu{}_{\gamma_{1}\ldots\gamma_{M};y} $. It is here that the quasirandomness property of $\Sp_{2g}( \Z /q\Z)$ is crucially used.

\begin{lem}
\label{lem:quasi-random}For
some absolute $C, D>0$ 
\[
\|\mu{}_{\gamma_{1}\ldots\gamma_{M};y}\|_{\ell_{\new}^{2}(\Gamma_{q})}\leq C\left(\frac{|\Gamma_{q}|\|\widetilde{\mu{}_{\gamma_{1}\ldots\gamma_{M};y}}*\mu{}_{\gamma_{1}\ldots\gamma_{M};y}\|_{2}^{2}}{q^{D}}\right)^{\frac{1}{4}}.
\]
Here $\|.\|_{2}^{2}$ denotes the $\ell^{2}$
norm of the measure on $\Gamma_{q}$ and $\|\mu{}_{\gamma_{1}\ldots\gamma_{M};y}\|_{\ell_{\new}^{2}(\Gamma_{q})}$
is the operator norm of $\mu{}_{\gamma_{1}\ldots\gamma_{M};y}$ acting
on the new subspace of $\ell^{2}(\Gamma_{q}).$
\end{lem}

\begin{proof}
We need to use the lower bound for the degree
of new irreducible representations of $\Sp((\Z/q\Z)^{2g},\omega_{\pi})$
that is given in Proposition \ref{prop:quasirandomness}. Supposing
that the smallest new irreducible representation has dimension $\gg q^{D}$
then by the trace formula argument of \cite[Lemma 44]{MOW} the largest
eigenvalue of $A^{*}A$ where $A:=\mu{}_{\gamma_{1}\ldots\gamma_{M};y}*$
acting on $\ell_{\new}^{2}(\Gamma_{q})$ satisfies
\[
\lambda^{2}q^{D}\leq C'|\Gamma_{q}|\|\widetilde{\mu{}_{\gamma_{1}\ldots\gamma_{M};y}}*\mu{}_{\gamma_{1}\ldots\gamma_{M};y}\|_{2}^{2}.
\]
The crucial point is that the eigenvalue appears with high multiplicity
in the trace formula, an idea that goes back to Sarnak and Xue \cite{SX}.
Since $\|A\|=\lambda^{1/2}$ the lemma follows.
\end{proof}

Now we can prove Proposition  \ref{prop:measure-convolution-bound}, modulo the deferred proof of Proposition \ref{prop:major}.
\begin{proof}[Proof of Proposition \ref{prop:measure-convolution-bound}]
We now make precise the argument we outlined before.
Let $\epsilon > 0$ be the constant from Proposition  \ref{prop:major}. Choose $B>1$ such that for all $q\geq 2$,
$$
\log B  \leq - \log(1 -\epsilon)   \frac{D}{2}\frac{ \log q}{ \log |\Gamma_q| }
$$
where $D$ is the constant from Lemma \ref{lem:quasi-random}. The reason for this choice will be pointed out shortly. Apply Proposition \ref{prop:major} for this $B$ to obtain a constant $L$ and measure $\mu_2$ such that $\mu_1 \leq \mu_2$.
We let $\tilde{M} = LK$ as in Proposition \ref{prop:major}.  Combining Proposition \ref{prop:major} and Lemma \ref{lem:spectralgfap-iterate-to-ell2-flat} we obtain 
$$
\|\widetilde{\mu_{2}}*\mu_{2}\|_{2} \leq  \|\mu_2 \|_1^2 \left(  \frac{ 1 }{ |\Gamma_q |^{1/2}} + (1 -\e)^{K} \right) \leq B^{2K} \|\mu_1 \|_1^2 \left(  \frac{ 1 }{ |\Gamma_q |^{1/2}} + (1 -\e)^{K} \right) 
$$
We can evaluate $\|\mu_1 \|_1$ by
$$
\|\mu_{1}\|_{1}= \sum_{\gamma_{M+1},\ldots,\gamma_{N}}e^{R_{\sigma}^{(\tilde{M})}(\alpha_{\gamma_{M+1}\ldots\gamma_{N}}o)}=\LL_{\sigma}^{\tilde{M}}[1](o) =1.
$$
Hence
$$
 \|\widetilde{\mu_{2}}*\mu_{2}\|_{2}  \leq B^{2K} \left(  \frac{ 1 }{ |\Gamma_q |^{1/2}} + (1 -\e)^{K} \right).
 $$
 From \eqref{eq:mu-to-mu1} and $\mu_1 \leq \mu_2$ we obtain 
$$
\|\widetilde{\mu{}_{\gamma_{1}\ldots\gamma_{M};y}}*\mu{}_{\gamma_{1}\ldots\gamma_{M};y}\|_{2} \leq C^2 e^{2 R_{\sigma}^{(M)}(\alpha_{\gamma_{1}\ldots\gamma_{M}}o)}B^{2K} \left(  \frac{ 1 }{ |\Gamma_q |^{1/2}} + (1 -\e)^{K} \right) .
$$
Using this as input to Lemma \ref{lem:quasi-random} gives
\begin{align*}
\|\mu{}_{\gamma_{1}\ldots\gamma_{M};y}\|^2_{\ell_{\new}^{2}(\Gamma_{q})} \ll \frac{ |\Gamma_q|^{1/2} }{q^{D/2} } \|\widetilde{\mu{}_{\gamma_{1}\ldots\gamma_{M};y}}*\mu{}_{\gamma_{1}\ldots\gamma_{M};y}\|_{2} \\
 \ll e^{2 R_{\sigma}^{(M)}(\alpha_{\gamma_{1}\ldots\gamma_{M}}o)}\frac {B^{2K}}{q^{D/2}} \left(  1 + |\Gamma_q|^{1/2}(1 -\e)^{K} \right) 
\end{align*}
We now choose our constant $c>0$ such that for $K \approx c \log q$, $|\Gamma_q|^{1/2}(1 -\e)^{K} \approx 1$. The choice of $B$ ensures that for this $K$, $\frac {B^{2K}}{q^{D/2}} \ll q^{-D/4}$. Hence
$$
\|\mu{}_{\gamma_{1}\ldots\gamma_{M};y}\|^2_{\ell_{\new}^{2}(\Gamma_{q})} \ll q^{-D/4}  e^{2 R_{\sigma}^{(M)}(\alpha_{\gamma_{1}\ldots\gamma_{M}}o)}.
$$
This proves the first inequality of Proposition \ref{prop:measure-convolution-bound} (replacing $D/8$ by $D$). As remarked before, the second inequality uses essentially the same argument.
\end{proof}

\subsection{Decoupling II: Majorizing $\mu_{1}$.\label{sub:Decoupling-II:-Bounding-the-real-measure}}
In this section we prove Proposition \ref{prop:major} by adapting arguments from \cite[Appendix]{MOW} to the infinite alphabet setting, using also a
different spectral gap input from property (T) that relies on our preparation of the
set $\Upsilon$ and its relation to the Rauzy-Veech
group $G_{\pi}$.
The key idea in the proof is that while $\mu_1$ is not a convolution, it can be majorized by a carefully chosen sum of convolutions.

We further decompose 
\begin{equation}
\tilde{M}=LK\label{eq:tildeM-ecomp}
\end{equation}
 where $L$ is going to be chosen to be a large constant, and decompose
$\{M+1,\ldots,N\}$ into blocks of size either 1, $L-1$ or $L$.
Let 
\[
I_{i,j}=[\gamma_{i},\gamma_{i+j}]
\]
 denote the block of all $\gamma_{i'}$ with $i\leq i'\leq j$. Rewrite
the summation in (\ref{eq:mu1-definition}) as 
\begin{equation}
\mu_{1}=\sum_{I_{M+1,M+L-1},I_{M+L+1,M+2L-1},\ldots,I_{N-L+1,N-1}}\:\sum_{\gamma_{M+L},\gamma_{M+2L},\ldots,\gamma_{N}}e^{R_{\sigma}^{(\tilde{M})}(\alpha_{\gamma_{M+1}\ldots\gamma_{N}}o)}h_{\gamma_{N}}h_{\gamma_{N-1}}\ldots h_{\gamma_{M+1}}.\label{eq:tempmu1}
\end{equation}
This reordering of summation is permitted since the sums are suitably
absolutely convergent by Theorem \ref{thm:RPF} and the following
discussion. Following \cite[(A.15)]{MOW}, using contraction properties
of $\alpha_{\gamma_{i}}$ and the bound (\ref{eq:r^M-deriv-bound})
for the derivative of $R_{\sigma}^{(\tilde{M})}$, one has the bounds

\begin{equation}
\exp(-c\Lambda^{-L})^{K-1}\beta_{1}\beta_{2}\ldots\beta_{K} \leq e^{R_{\sigma}^{(\tilde{M})}(\alpha_{\gamma_{M+1}\ldots\gamma_{N}}o)}\leq\exp(c\Lambda^{-L})^{K-1}\beta_{1}\beta_{2}\ldots\beta_{K}\label{eq:decouple to beta}
\end{equation}
where
\[
\beta_{K}=e^{R_{\sigma}^{(L)}(\alpha_{\gamma_{N-L+1}\ldots\gamma_{N}}o)},\quad\beta_{j}=e^{R_{\sigma}^{(L)}(\alpha_{\gamma_{M+(j-1)L+1}\ldots\gamma_{M+(j+1)L-1}}o)},\:1\leq j\leq K-1.
\]
and $c>0$ is a constant. Notice the important feature that each $\beta_{j}$
depends on only one of $\gamma_{M+jL}$. Inserting the second inequality of (\ref{eq:decouple to beta})
into (\ref{eq:tempmu1}) gives
\begin{equation}
\mu_{1}\leq \mu_2 := \exp(c\Lambda^{-L})^{K-1}\sum_{I_{M+1,M+L-1},I_{M+L+1,M+2L-1},\ldots,I_{N-L+1,N-1}}\eta_{R}*\eta_{R-1}*\ldots*\eta_{1}\label{eq:mu_1-to-iterated-convolution}
\end{equation}
where the $\eta_{j}=\eta_{j}(I_{M+1,M+L-1},I_{M+L+1,M+2L-1},\ldots,I_{N-L+1,N-1})$
are given by
\begin{eqnarray*}
\eta_{K} & := & \sum_{\gamma_{N}}\beta_{K}(\gamma_{N-L},\ldots,\gamma_{N})h_{\gamma_{N}}\ldots h_{\gamma_{N-L+1}},\\
\eta_{j} & := & \sum_{\gamma_{M+jL}}\beta_{j}(\gamma_{M+(j-1)L+1},\ldots,\gamma_{M+(j+1)L-1})h_{\gamma_{M+jL}}\ldots h_{\gamma_{M+(j-1)L+1}},\quad1\leq j\leq K-1.
\end{eqnarray*}
We point out for the readers convenience that \emph{we have now defined $\mu_2$}.
This proves \eqref{eq:mu1major}.

To prove \eqref{m2colv} we now aim for bounds on the operator norms of the measures $\eta_{j}$
acting by convolution on $\ell_{0}^{2}(\Gamma_{q})$. We write $\|\eta_{j}\|_{op}$
for this operator norm. Consider, taking for example $1\leq j\leq K-1$
\begin{equation}
\eta_{j}*\tilde{\eta_{j}}=\sum_{\gamma_{M+jL},\gamma'_{M+jL}}\beta_{j}(\ldots,\gamma_{M+jL},\ldots)\beta_{j}(\ldots,\gamma'_{M+jL},\ldots)h_{\gamma_{M+jL}}(h_{\gamma'_{M+jL}})^{-1}.\label{eq:etastaret}
\end{equation}
Since 
\begin{equation}
\|\eta_{j}\|_{op}=\|\tilde{\eta}_{j}\|_{op}=\sup_{\phi\in\ell_{0}^{2}(\Gamma_{q}):\|\phi\|=1}\langle\eta_{j}*\tilde{\eta}_{j}\phi,\phi\rangle^{1/2}\label{eq:eta-rayeligh}
\end{equation}
we turn to estimating the operator norm of $\eta_{j}*\tilde{\eta_{j}}$
on $\ell_{0}^{2}(\Gamma_{q}).$ We need to both
\begin{enumerate}
\item estimate the values of $\beta_{j}$ and 
\item discuss the group elements $h_{\gamma_{M+jL}}(h_{\gamma_{M+jL}})^{-1}$. 
\end{enumerate}

These are both points of departure from \cite[Appendix]{MOW}, so
we give more details.

\emph{1) }Continuing with $1\leq j\leq K-1$ (the edge case $j=K$
is similar) we have
\begin{eqnarray}
\beta_{j}(\ldots,\gamma_{M+jL},\ldots) & = & e^{R_{\sigma}^{(L)}(\alpha_{\gamma_{M+(j-1)L+1}\ldots\gamma_{M+(j+1)L-1}}o)}\nonumber \\
 & = & \exp\left(\sum_{i=0}^{L-2}R_{\sigma}(\alpha_{\gamma_{M+(j-1)L+1+i}\ldots\gamma_{M+(j+1)L-1}}o)\right)\exp\left(R_{\sigma}(\alpha_{\gamma_{M+jL}\ldots\gamma_{M+(j+1)L-1}}o)\right)\nonumber \\
 & = & \exp\left(\sum_{i=0}^{L-2}R_{\sigma}(\alpha_{\gamma_{M+(j-1)L+1+i}\ldots\gamma_{M+jL-1}}o)+O(\Lambda^{-i})\right)\exp\left(R_{\sigma}(\alpha_{\gamma_{M+jL}\ldots\gamma_{M+(j+1)L-1}}o)\right)\nonumber \\
 & \asymp & B(\gamma_{M+(j-1)L+1+i},\ldots,\gamma_{M+jL-1})\exp\left(R_{\sigma}(\alpha_{\gamma_{M+jL}\ldots\gamma_{M+(j+1)L-1}}o)\right)\label{eq:beta_j-est-1}
\end{eqnarray}
where $\asymp$ means bounded above and below by a constant independent
of all $\gamma_{i}$ and $L$, and 
\[
B(\gamma_{M+(j-1)L+1+i},\ldots,\gamma_{M+jL-1}):=\exp\left(R_{\sigma}^{(L-1)}(\alpha_{\gamma_{M+(j-1)L+1+i}\ldots\gamma_{M+jL-1}}o)\right).
\]
Note the arguments of $B$ are fixed given $\eta_{j}$. Also note
that for fixed $\eta_{j}$ the values
\[
\alpha_{\gamma_{M+jL+1}\ldots\gamma_{M+(j+1)L-1}}o
\]
lie in a (cylinder) set $U(\eta_{j})$ with diameter $\ll\Lambda^{-L}$.
On the other hand, the derivative of $R_{\sigma}\circ\alpha_{\gamma_{M+jL}}$
is uniformly bounded by  (\ref{eq:r^M-deriv-bound}) so the values in the exponent of (\ref{eq:beta_j-est-1})
fluctuate by at most $\ll\Lambda^{-L}$ while $\alpha_{\gamma_{M+jL}}$
is fixed. Therefore 
\begin{equation}
\beta_{j}(\ldots,\gamma_{M+jL},\ldots)\asymp B(\gamma_{M+(j-1)L+1+i},\ldots,\gamma_{M+jL-1})\exp\left(R_{\sigma}(\alpha_{\gamma_{M+jL}}o)\right).\label{eq:beta_jtoB}
\end{equation}
In light of this estimate and the discussion after Theorem \ref{thm:RPF}
concerning convergence of infinite sums, we see that $\eta_{j}$ and
$\eta_{j}*\tilde{\eta}_{j}$ have finite $\ell_{1}$ norms. This supports
our earlier justification of reordering of summations.

\emph{2) }Recall $\Upsilon$ from Lemma \ref{lem:differences-gen-G_pi'}.
We can write
\[
\eta_{j}*\tilde{\eta_{j}}=\nu+\tilde{\nu}
\]
where $\nu$ is the contribution to (\ref{eq:etastaret}) from $\gamma_{M+jL},\gamma'_{M+jL}\in\gamma_{0}.\Upsilon$
and $\tilde{\nu}$ are the remaining contributions. Then the support
of $\nu$ is the reduction mod $q$ of the set 
\[
\Sigma=\{\Theta_{\gamma}^{*}.(\Theta_{\gamma'}^{*})^{-1}\::\:\gamma,\gamma'\in\gamma_{0}.\Upsilon\}.
\]
By Lemma \ref{lem:differences-gen-G_pi'}, the set $\Sigma$ generates
the conjugate of $G_{\pi}$ by $\Theta_{\gamma_{0}}^{*}$. Call this
conjugate group $G'_{\pi}.$

We now bring these arguments \emph{1) }and \emph{2) }together. Let
$\nu=\eta_{j}*\tilde{\eta_{j}}$. Note that the operator formed from
convolution by $\nu$ on $\ell_{0}^{2}(\Gamma_{q})$ is self-adjoint
and positive. Therefore the operator norm of $\|\nu\|$ acting by
convolution on $\ell_{0}^{2}(\Gamma_{q})$ is bounded by

\begin{equation}
\|\nu\|_{op}\leq\sup_{\phi\in\ell_{0}^{2}(\Gamma_{q}),\|\phi\|=1}\langle\nu*\phi,\phi\rangle.\label{eq:rayleigh-nu}
\end{equation}

We need to use the following property of the action of $G'_{\pi}$
on $\Gamma_{q}.$
\begin{lem}[No almost invariant vectors]
 \label{lem:-no-invariant-vectros}There is  some
$\epsilon>0$ such that for all odd $q$, for all $\phi\in\ell_{0}^{2}(\Gamma_{q})$ with
$\|\phi\|_{\ell^{2}}=1$ there is some $g\in\Sigma$ such that if
$g_{q}:=g\mod q$ then
\[
\|g_{q}*\phi-\phi\|_{\ell^{2}}>\epsilon.
\]
\end{lem}

\begin{proof}
By Theorem \ref{thm:zorich-conjecture}, when $q$ is odd, $G'_\pi$ maps onto $\Gamma_q$. Hence $\ell^2_0(\Gamma_q)$ has no invariant vectors. The statement of the lemma is then a consequence of Kazhdan's property $(T)$ for finite index subgroups of $\Sp(\Z^{2g},\omega_{\pi})$  \cite{KAZHDAN} applied to $G'_\pi$.
\end{proof}

Write 
\[
\nu=\sum_{g_{q}\in\Gamma_{q}}\nu_{g_{q}}g_{q}.
\]
Let $\epsilon$, $g_{q}^{0}$ be the constant (resp. group element)
provided by Lemma \ref{lem:-no-invariant-vectros} on inputting $\phi$
with $\|\phi\|=1$. Then it is straightforward to check that $|\Re(\langle g_{q}^{0}*\phi,\phi\rangle)|<(1-\epsilon')$
where $\epsilon'=\epsilon^{2}/2$. Returning to (\ref{eq:rayleigh-nu}),
using $\nu_{g_{q}}=\nu_{(g_{q})^{-1}}$ from (\ref{eq:etastaret})
we get 
\begin{eqnarray*}
\langle\nu*\phi,\phi\rangle & = & \sum_{g_{q}\in\Gamma_{q}}\nu_{g_{q}}\langle g_{q}*\phi,\phi\rangle=\sum_{g_{q}\in\Gamma_{q}}\nu_{g_{q}}\Re\langle g_{q}*\phi,\phi\rangle\\
 & = & \nu_{g_{q}^{0}}\Re\langle g_{q}^{0}*\phi,\phi\rangle+\sum_{g_{q}\neq g_{q}^{0}}\nu_{g_{q}}\Re\langle g_{q}*\phi,\phi\rangle\\
 & \leq & (1-\epsilon')\nu_{g_{q}^{0}}+\sum_{g_{q}\neq g_{q}^{0}}\nu_{g_{q}}=\|\nu\|_{1}-\epsilon'\nu_{g_{q}^{0}}.
\end{eqnarray*}
Also, from (\ref{eq:beta_jtoB}), and that $\sum_{\gamma_{M+jL}}\exp\left(R_{\sigma}(\alpha_{\gamma_{M+jL}}o)\right) = \LL_{\sigma}[1](o)=1$,
we get
\[
\nu_{g_{q}^{0}}\geq C\|\nu\|_{1}
\]
 with constant $C$ independent of $\nu_{g_{q}^{0}}$ and $\eta_{j}$.
So combining this with the preceding estimate and (\ref{eq:rayleigh-nu})
we get
\[
\|\nu\|_{op}\leq\|\nu\|_{1}(1-\epsilon'')
\]
 for some $\epsilon''>0$. Inserting this into (\ref{eq:eta-rayeligh})
gives
\begin{equation}
\|\eta_{j}\|_{op}\leq\|\eta_{j}\|_{\ell^{1}}(1-\epsilon'')^{1/2}.\label{eq:eta_j-operator}
\end{equation}
Using (\ref{eq:eta_j-operator}) in (\ref{eq:mu_1-to-iterated-convolution})
gives for any $\phi\in\ell_{0}^{2}(\Gamma_{q})$ 
\[
\|\mu_{2}*\phi\|_{\ell^{2}}\le\exp(c\Lambda^{^{-L}})^{K-1}(1-\epsilon'')^{K/2} \|\mu_2 \|_1 \|\phi\|.
\]
This is almost the proof of \eqref{m2colv}; we just have to choose $L$.
 Before we do so, we estimate $\| \mu_2 \|_1$. We have
 $$
\| \mu_2 \|_1 \leq  \exp(c\Lambda^{-L})^{K-1} \sum_{I_{M+1,M+L-1},I_{M+L+1,M+2L-1},\ldots,I_{N-L+1,N-1}}\:\sum_{\gamma_{M+L},\gamma_{M+2L},\ldots,\gamma_{N}}\beta_1 \ldots \beta_K .
$$
We now use the first inequality of \eqref{eq:decouple to beta} to get 
$$
\| \mu_2 \|_1 \leq \exp(2c\Lambda^{^{-L}})^{K-1} \sum_{I_{M+1,M+L-1},I_{M+L+1,M+2L-1},\ldots,I_{N-L+1,N-1}}\:\sum_{\gamma_{M+L},\gamma_{M+2L},\ldots,\gamma_{N}}e^{R_{\sigma}^{(\tilde{M})}(\alpha_{\gamma_{M+1}\ldots\gamma_{N}}o)}.
$$
But from inspection of \eqref{eq:tempmu1}, the above is $\exp(2c\Lambda^{^{-L}})^{K-1} \| \mu_1 \|.$

Recall $B$ is the quantifier from Proposition \ref{prop:major}. { \bf We now choose $L$ } large enough so that both
$$
\exp(2c\Lambda^{^{-L}}) \leq B
$$
and 
$$
\exp(c\Lambda^{^{-L}}) \leq ( 1 - \epsilon'')^{-1/4} .
$$
This gives 
\[
\|\mu_{2}*\phi\|_{\ell^{2}}\le \exp(c\Lambda^{^{-L}})^{K-1}(1-\epsilon'')^{K/2} \|\mu_2 \|_1 \|\phi\|_{\ell^2} \leq (1-\epsilon'')^{K/4} \|\mu_2 \|_1 \|\phi\|_{\ell^2}
\]
and 
$$
\| \mu_2 \|_1 \leq \exp(2c\Lambda^{^{-L}})^{K-1} \| \mu_1 \| \leq  B^K \| \mu_1 \|.
$$
This completes the proof of Proposition \ref{prop:major}.

\section{Quasirandomness\label{sec:Quasirandomness}}

In this section we show that `new' representations of $\Sp((\Z/q\Z)^{2g},\omega_{\C})$ have large dimension.
 This is
a version of the \emph{quasirandomness} property of a group that takes into account the level structure of the family
of groups $\Sp((\Z/q\Z)^{2g},\omega_{\C})$. 
\begin{prop}[Quasirandomness estimates]
\label{prop:quasirandomness}There is $C>0$ and $D>0$ such that
any irreducible representation of $\Sp((\Z/q\Z)^{2g},\omega_{\C})$
that does not factor through
\[
\Sp((\Z/q\Z)^{2g},\omega_{\C})\to\Sp((\Z/q_{1}\Z)^{2g},\omega_{\C})
\]
for some $q_{1}|q$ has dimension $\geq Cq^{D}$. 
\end{prop}

We follow the type of argument given by Kelmer and Silberman in \cite[Section 4]{KS}
for rank one groups (see also \cite{MJEMS} for a small improvement
to that argument). We may treat the group $\Sp_{2g}(\Z)$ without
loss of generality, that is, we assume the symplectic form is the
standard one. Let $g\geq2.$ Let $q\in\mathbf{N}$ and let $(\rho,V)$
be an irreducible unitary representation of $\Sp_{2g}(\Z/q\Z$) that
is not obtained by a composition
\[
\Sp_{2g}(\Z/q\Z)\to\Sp_{2g}(\Z/q_{1}\Z)\xrightarrow{\rho'}U(V)
\]
with $q_{1}|q$. We refer to this property as $\rho$ being \emph{new.}

\subsection{The case when $q$ is prime\label{sub:The-case-when-q-is-prime}}

For $p$ an odd prime, let $\mathbb{F}_{p}$ denote the finite field
with $p$ elements. The table of Seitz and Zalesskii in \cite[Table 1]{SZ}
implies that $P\Sp_{2g}(\mathbb{F}_{p})$ has no projective complex
irreducible representation of dimension $<\frac{1}{2}(p^{g}-1)$ and
hence this is also a lower bound for the dimension of an irreducible
representation of $\Sp_{2g}(\mathbb{F}_{p})$.

\subsection{The case $q=p^{r}$}

In this section we prove the following
\begin{prop}
\label{prop:prime-power-case}There is some $C>0$ depending only
on $g$ such that for all $r\geq2,$ letting $R:=\lfloor r/2\rfloor$
any new representation $(\rho,V)$ of $\Sp_{2g}(\Z/p^{r}\Z)$ has
dimension at least
\[
\dim\rho\geq Cp^{R}.
\]
\end{prop}

Let $q=p^{r}$. Write $H_{q}:=\Sp_{2g}(\Z/q\Z)$ and for $q'|q$ let
$H_{q}(q')$ be the kernel of the reduction modulo $q'$ map
\[
H_{q}\to H_{q'}.
\]
Let $\mathfrak{g}(\Z/q\Z)$ denote the Lie algebra of $\Sp_{2g}$
over $\Z/q\Z$. We view this as an abelian group. Let $R=\lfloor r/2\rfloor$.
The congruence subgroup $H_{p^{r}}(p^{r-R})$ is an abelian normal
subgroup of $H_{p^{r}}$ that is naturally isomorphic to $\mathfrak{g}(\Z/p^{R}\Z)$.
The action of $H_{p^{r}}$ on $H_{p^{r}}(p^{r-R})$ by conjugation
descends to an action of $H_{p^{R}}$. After using the isomorphism
$H_{p^{r}}(p^{r-R})\cong\mathfrak{g}(\Z/p^{R}\Z)$ this conjugation
action is identified with the Adjoint action of $H_{p^{R}}$ on $\mathfrak{g}(\Z/p^{R}\Z)$,
i.e.
\[
\Ad(g)v=gvg^{-1},\quad g\in H_{p^{R}},v\in\mathfrak{g}(\Z/p^{R}\Z).
\]

Let $(\rho,V)$ be a unitary representation of $H_{q}$. Suppose $R\geq1.$
If $\rho$ is trivial when restricted to $H_{q}(p^{R})$ then $\rho$
is not a new representation. More generally, if $\rho$ is new, then
the restriction of $\rho$ to $H_{q}(p^{r-R})$ must not be trivial
on any $H_{q}(p^{r-R+\eta})$ with $\eta\in\Z_{+}$ since these are
also normal subgroups with $H_{q}/H_{q}(p^{r-R+\eta})\cong H_{p^{r-R+\eta}}$.
Notice $H_{q}(p^{r-R+\eta})\leq H_{q}(p^{r-R})$ corresponds to the
inclusion $p^{\eta}\mathfrak{g}(\Z/p^{R-\eta}\Z)\leq\mathfrak{g}(\Z/p^{R}\Z)$.

The strategy is to consider the $H_{p^{R}}$ invariant set of characters
of $\mathfrak{g}(\Z/p^{R}\Z)$ that appear when restricting $\rho$
to $H_{p^{r}}(p^{r-R})\cong\mathfrak{g}(\Z/p^{R}\Z)$, since the size
of this set gives a lower bound for the dimension of $\rho$.

The Killing form on $\mathfrak{g}(\Z/p^{R}\Z)$ is non-degenerate
which allows us to identify the unitary dual $\widehat{\mathfrak{g}(\Z/p^{R}\Z)}$
with $\mathfrak{g}(\Z/p^{R}\Z)$. Under this identification, the co-Adjoint
action on characters becomes an Adjoint action on $\mathfrak{g}(\Z/p^{R}\Z)$.
Moreover any character that is non trivial on each $H_{q}(p^{r-R+\eta})$,
$\eta\in\Z_{+}$, becomes an element of $\mathfrak{g}(\Z/p^{R}\Z)$
which is not $\equiv0\bmod p$.

We have therefore reduced Proposition \ref{prop:prime-power-case}
 to the following Lemma.
\begin{lem}
There is some $C>0$ depending only on $g$ such that for all $R\geq1$
the $H_{p^{R}}$-Adjoint orbit of any $X\in\mathfrak{g}(\Z/p^{R}\Z)$
with $X\not\equiv0\bmod p$ has size 
\[
|\Ad(H_{p^{R}}).X|\geq Cp^{R}.
\]

\end{lem}

\begin{proof}
By orbit-stabilizer theorem the orbit has size at least 
\begin{equation}
\frac{|H_{p^{R}}|}{|C_{H_{p^{R}}}(X)|}\label{eq:quotient-centralizer}
\end{equation}
where we write $C$ to stand for centralizer, therefore $C_{H_{p^{r}}}(X)=\{h\in H_{p^{R}}\::\:hXh^{-1}=X\}.$
Since $H_{p^{R}}$ is an $R-1$ fold extension of $H_{p}$ by groups
isomorphic to $\mathfrak{g}(\mathbb{F}_{p})$ we know $|H_{p^{R}}|=|H_{p}||\mathfrak{g}(\mathbb{F}_{p})|^{R-1}\gg p^{R.\dim(\Sp_{2g})}=p^{g(2g+1)R}.$
This gives the bound we will use for the numerator of (\ref{eq:quotient-centralizer}).

Considering next the denominator of (\ref{eq:quotient-centralizer}),
by an elementary induction argument appearing in \cite[Proof of Proposition 4.3]{KS}

\begin{equation}
|C_{H_{p^{r}}}(X)|\leq|C_{H_{p}}(X\bmod p)||C_{\mathfrak{g}(\mathbb{F}_{p})}(X\bmod p)|^{R-1}\label{eq:centralizer-temp}
\end{equation}
where the latter centralizer is $C_{\mathfrak{g}(\mathbb{F}_{p})}(X\bmod p)=\{y\in\mathfrak{g}(\mathbb{F}_{p})\::\:[y,X\bmod p]\equiv0\}.$
According to Springer and Steinberg \cite[II. 4.1, 4.2, IV. 2.26]{SS},
the algebraic group $C_{\Sp_{2g}}(X\bmod p)$ defined over $\mathbb{F}_{p}$
has a number of components bounded by a constant depending only on
$g$. By a bound of Nori \cite[Lemma 3.5]{NORI} each component can
have at most $\leq(p+1)^{\dim C_{\Sp_{2g}}(X\bmod p)}$ points over
$\mathbb{F}_{p}$. But $\dim C_{\Sp_{2g}}(X\bmod p)$ is also the
dimension of the centralizer of $X\bmod p$ in $\mathfrak{g}(\mathbb{F}_{p})$
so we have now reduced the estimation of the right hand side of (\ref{eq:centralizer-temp})
to a bound for 
\[
\dim C_{\mathfrak{g}(\mathbb{F}_{p})}(X')
\]
where $X'=X\bmod p$ is a nonzero element of $\mathfrak{g}(\mathbb{F}_{p})$. 

Assume the bound $\dim C_{\mathfrak{g}(\mathbb{F}_{p})}(X')\leq\dim(\Sp_{2g})-e=g(2g+1)-e.$
Then putting our previous estimates together the orbit has size at
least
\[
\gg\frac{p^{g(2g+1)R}}{(p+1)^{g(2g+1)-e}(p^{g(2g+1)-e})^{R-1}}\gg p^{eR}.
\]
Since it is not particularly important here to optimize \emph{$e$},
we give the easy argument that one may take $e=1$ since $\mathfrak{g}(\mathbb{F}_{p})$
has no nontrivial center\footnote{One may definitely do better here, and it would be good to work out
the best possible bound, but it is not the purpose of the current
paper.}. This gives the result stated in the lemma.

\end{proof}

\subsection{The case of general moduli}

If $p_{i}$ are primes and 
\[
q=\prod_{i=1}^{M}p_{i}^{m_{i}}
\]
 is the prime factorization of $q$, then we have by the Chinese remainder
theorem
\[
\Sp_{2g}(\Z/q\Z)\cong\prod_{i=1}^{M}\Sp_{2g}(\Z/p_{i}^{m_{i}}\Z).
\]
Then $\rho$ splits as a tensor product 
\[
\rho=\bigotimes_{i=1}^{M}\rho_{i}
\]
 where $\rho_{i}$ are irreducible representations of $\Sp_{2g}(\Z/p_{i}^{m_{i}}\Z)$.
Since $\rho$ is new, all of the $\rho_{i}$ are new. Now using Proposition
\ref{prop:prime-power-case} and the bounds for the case of prime
modulus from Section \ref{sub:The-case-when-q-is-prime} gives
\[
\dim\rho\geq\prod_{i:\:m_{i}=1}\frac{1}{2}(p_{i}^{g}-1)\prod_{i:\:m_{i}>1}Cp_{i}^{\lfloor m_{i}/2\rfloor}\geq q^{1/2}(C')^{-\omega(q)}
\]
given $g\geq2$ for some $C'>1$ and $\omega(q)$ standing for the
number of distinct prime factors of $q.$ But $(C')^{\omega(q)}\ll_{\epsilon}q^{\epsilon}$
for any $\epsilon>0$. This concludes the proof of Proposition \ref{prop:quasirandomness},
in fact, our proof shows that one may take $D$ as close as one likes
to $1/2$ provided one chooses $C$ appropriately.
\begin{verbatim}

\end{verbatim}

\def\cprime{$'$}

\noindent Michael Magee, \\Department of Mathematical Sciences, \\ Durham University,\\  Durham DH1 3LE, U.K. \\  {\tt michael.r.magee@durham.ac.uk}\\ \\

\end{document}